\documentclass[a4paper,11pt,reqno]{amsart}

\usepackage{amssymb}
\usepackage{latexsym}
\usepackage{amsmath}
\usepackage{mathrsfs}
\usepackage{euscript}
\usepackage{mathtools}
\usepackage{amsthm}
\usepackage{upgreek}
\usepackage{cite}
\usepackage{amscd}
\usepackage{xcolor}
\usepackage{tikz}
\usepackage{calc}

\usepackage{diagbox}
\usepackage{tikz-cd}

\usepackage{comment}

\usepackage{color}
\definecolor{ceruleanblue}{rgb}
{0.16, 0.32, 0.75}

\newcommand{\scal}[2]{\langle #1,#2\rangle}

\newcommand{\rr}[1]{\mathbf R^{#1}}
\newcommand{\zz}[1]{\mathbf Z^{#1}}

\newcommand{\cc}[1]{\mathbf C^{#1}}
\newcommand{\nn}[1]{\mathbf N^{#1}}

\newcommand{\nm}[2]{\Vert #1\Vert _{#2}}

\newcommand{\op}{\operatorname{Op}}

\newcommand{\sets}[2]{\{ \, #1\, ;\, #2\, \} }

\newcommand{\ep}{\varepsilon}

\newcommand{\cdo}{\, \cdot \, }

\newcommand{\eabs}[1]{\langle #1\rangle}

\newcommand{\ON}{\operatorname{ON}}

\newcommand{\vrum}{\vspace{0.1cm}}

\newcommand{\wpr}{{\text{\footnotesize $\#$}}}

\newcommand{\phm}{\phantom{-}}

\newcommand{\sfR}{\mathsf{R}}

\newcommand{\scW}{{\scriptscriptstyle{\mascW} }\! }

\newcommand{\maclA}{\mathcal A}
\newcommand{\maclB}{\mathcal B}

\newcommand{\maclC}{\mathcal C}

\newcommand{\maclH}{\mathcal H}

\newcommand{\maclL}{\mathcal L}
\newcommand{\maclM}{\mathcal M}

\newcommand{\maclS}{\mathcal S}
\newcommand{\maclV}{\mathcal V}

\newcommand{\maclWH}{\mathcal V\! \mathcal H}

\newcommand{\mascB}{\mathscr B}

\newcommand{\mascF}{\mathscr F}

\newcommand{\mascI}{\mathscr I}

\newcommand{\mascP}{\mathscr P}
\newcommand{\mascR}{\mathscr R}
\newcommand{\mascS}{\mathscr S}
\newcommand{\mascV}{\mathscr V}
\newcommand{\mascW}{\mathscr W}

\newcommand{\IM}{\operatorname{Im}}

\newcommand{\fka}{\mathfrak a}
\newcommand{\fkb}{\mathfrak b}

\numberwithin{equation}{section}

\newtheorem{thm}{Theorem}
\numberwithin{thm}{section}

\newcommand{\rubrik}{}
\newtheorem{prop}[thm]{Proposition}
\newtheorem{cor}[thm]{Corollary}
\newtheorem{lemma}[thm]{Lemma}

\theoremstyle{definition}

\newtheorem{defn}[thm]{Definition}

\theoremstyle{remark}

\newtheorem{rem}[thm]{Remark}              %T o m hit

\definecolor{darkred}{rgb}{0.8,0,0}

\author[L. Neyt]{Lenny Neyt}
\address{Faculty of Mathematics,
University of Vienna,
Vienna, Austria}

\email{lenny.neyt@univie.ac.at}

\author[J. Toft]{Joachim Toft}

\address{Department of Mathematics,
Linn{\ae}us University,
V{\"a}xj{\"o}, Sweden}

\email{joachim.toft@lnu.se}

\title[Fractional symplectic Fourier transform]{Fractional symplectic
Fourier transforms and their actions on function and distribution
spaces}


 %35Q40; 30Gxx}

\begin{document}

\begin{abstract}
We consider symplectic analogues of the fractional Fourier transforms 
and obtain convenient representations in terms 
of Bargmann transforms and Wigner distributions. 
We show that they act as homeomorphisms on $\mascS$, $L^2$, $\mascS '$, and on Fourier-invariant Gelfand–Shilov spaces and their corresponding distribution spaces. 
More generally, we establish their homeomorphism properties on a broad class of spaces defined through (Wigner–)Hermite expansions. 
In addition, we prove that these transforms preserve norms on Schatten–von Neumann symbol classes associated with the Weyl calculus.
%We consider symplectic analogies of
%fractional Fourier transforms. We find
%convenient descriptions of these transforms
%via Bargmann transforms and Wigner distributions.
%We show that
%they are homeomorphisms on $\mascS$,
%$L^2$, $\mascS '$, and on Fourier invariant
%Gelfand-Shilov (distribution) spaces.
%More generally, it is shown that these
%transforms are homeomorphisms on a broad family
%of (Wigner)-Hermite function expansions.
%We also show that they are norm-preserving
%on Schatten-von Neumann symbol classes in
%the Weyl calculus. 
\end{abstract}

\maketitle

\par

%%%%%%%%%%%%%%%%%%%%%%%%%
\section{Introduction}
\label{sec0}
%%%%%%%%%%%%%%%%%%%%%%%%%

\par

In the paper we investigate different
properties of fractional Fourier transforms,
especially \emph{symplectic fractional
Fourier transforms}.
We deduce mapping properties
for such transforms on classical spaces like
Schwartz spaces, $L^2$ spaces, Fourier invariant
Gelfand-Shilov spaces, and all their
corresponding distribution spaces. For
fractional Fourier transforms we also deduce
mapping properties of fractional Fourier
transforms, when they act on general classes of 
modulation spaces, and on a broad family of 
Pilipovi{\'c} spaces, i.{\,}e., topological vector 
spaces of Hermite function expansions.
For \emph{symplectic} fractional Fourier
transforms, similar investigations are
performed, with Wigner-Pilipovi{\'c} spaces
in place of Pilipovi{\'c} spaces, i.{\,}e.,
Wigner-Hermite functions should replace
Hermite functions in the expansions.
Notice that Wigner-Hermite functions are
thoroughly investigated by M.W. Wong in \cite{Won}.

\par

We remark that the obtained mapping properties
on (Wigner-)Hermite expansions,
are more general compared to analogous
mapping properties on the classical
spaces above, because the latter spaces
are examples on topological vector
spaces of (Wigner)-Hermite function expansions.
Additionally, we deduce mapping properties
for fractional symplectic Fourier transforms
when acting on Schatten-von Neumann
symbol classes in the Weyl calculus.

\medspace

The paper is organized as follows. In Section
\ref{sec1} we recall basic properties for
certain Fourier integrals, Bargmann type transforms,
Gelfand triples, qausi-Banach function spaces,
Schatten-von Neumann classes, Weyl quantization
of pseudo-differential operators, and twisted
convolutions. The Fourier integrals
especially include fractional Fourier transforms
and Wigner distributions. The Bargmann type
transforms include the Bargmann transform
and the symplectic Bargmann transform.

\par

In Section \ref{sec2}, we define fractional 
symplectic Fourier transforms, and explain how
their distribution kernels can be expressed
in terms of the kernels of fractional
Fourier transforms.

\par

In Sections \ref{sec3}--\ref{sec6}, different
mapping properties are deduced for fractional
Fourier transforms and for fractional symplectic
Fourier transforms. In Section \ref{sec3}, it
is proved that fractional
symplectic Fourier transforms are homeomorphisms
on Schatten-von Neumann symbol classes in
the Weyl calculus. In Section \ref{sec4}, it is
proved that fractional Fourier transforms and
symplectic fractional Fourier transforms
are homeomorphisms on Gelfand-Shilov function
spaces, and their distribution spaces.

\par

In Section \ref{sec5}, we introduce
broad families of Pilipovi{\'c} spaces
and Wigner-Pilipovi{\'c} spaces, which
are topological vector spaces of
Hermite function expansions, and Wigner-Hermite
function expansions. As special cases,
any Schwartz space, space of tempered distributions,
Fourier invariant space of Gelfand-Shilov functions
and distributions are special cases of
such Pilipovi{\'c} or Wigner-Pilipovi{\'c}
spaces. Then we prove that fractional
Fourier transforms are homeomorphisms on
such general Pilipovi{\'c} spaces, and that
fractional symplectic Fourier transforms 
are homeomorphisms on
general Wigner-Pilipovi{\'c} spaces. Even
more general, for some of these
Pilipovi{\'c} spaces and Wigner-Pilipovi{\'c} 
spaces, we deduce homeomorphism properties
for fractional Fourier transforms and
fractional symplectic Fourier transforms of
general \emph{complex} order. 

\par

In Section \ref{sec6}, we deduce that
fractional Fourier transforms and
fractional symplectic Fourier transforms
are homeomorphisms between suitable
general classes of modulation spaces.

\par

\par

%%%%%%%%%%%%%%%%%%%%%%%%%
\section{Preliminaries}
\label{sec1}
%%%%%%%%%%%%%%%%%%%%%%%%%

\par

In the first part of the section, we recall
some facts on weight functions.
Thereafter, we discuss basic properties 
for fractional Fourier 
transforms, Wigner distributions, and other
Fourier integrals. Following, we recall
some facts for the Bargmann transform, and
discuss some issues for its
symplectic analogy. Then
Gelfand triples and quasi-Banach function spaces
are considered. In the last part, we recall
some facts on
Schatten-von Neumann classes, Weyl quantization
of pseudo-differential operators, and twisted
convolutions.

\par

\subsection{Weight functions}
A \emph{weight} on $\rr {d}$ is a positive function
$\omega \in L^\infty_{\operatorname{loc}}(\rr d)$ 
such that
$1 / \omega \in L^\infty_{\operatorname{loc}}(\rr d)$.
(For non-specified notation, see \cite{Ho1}.)
We let $\mascP _A(\rr d)$ be the set of all weights 
on $\rr d$.
A weight $\omega$ is called \emph{moderate} when
\begin{equation}
\label{Eq:DefModerateWeight}
\sup_{y \in \rr d}
\left (
\frac{\omega(x + y)}{\omega(y)}
\right )
< \infty ,
\quad \text{for}\quad
x \in \rr d . 
\end{equation}
The set of all moderate weights
on $\rr d$ is denoted by
$\mascP _E(\rr d)$.

\par

We notice that
$\omega \in \mascP _E(\rr d)$,
if and only
if there is a weight $v$ such that
\begin{equation}
\label{Eq:DefModerateWeightAlt}
\omega(x + y)
\le \omega(x)v(y)
\quad \text{when}\quad
x,y \in \rr d . 
\end{equation}
The smallest such $v$, denoted by
$v_{0,\omega}$, is given by
the left-hand side of
\eqref{Eq:DefModerateWeight}. By
straight-forward computations, it
follows that
\begin{equation}
v_\omega (x)
\equiv
\max (v_{0,\omega }(x),v_{0,\omega }(-x)),
\qquad
x\in \rr d,
\end{equation}
is submultiplicative. Here we recall
that $v\in \mascP _E(\rr d)$ is called
submultiplicative, if it is even and
\eqref{Eq:DefModerateWeightAlt} holds
with $\omega =v$. We notice that
\eqref{Eq:DefModerateWeightAlt}
implies
\begin{equation}
\label{eq:WeightFunctionLowerUpperBoundWithv}
\frac{\omega(0)}{v(-x)} \le \omega (x)\le \omega(0) v(x),
\qquad x\in \rr d,
\end{equation}
so in particular
\begin{equation}
\frac{\omega(0)}{v_\omega (x)} \le \omega (x)
\le \omega(0) v_\omega (x), \qquad x\in \rr d.
\end{equation}
We recall that submultiplicative
weights are subexponential. 
More precisely,
if $v$ is submultiplicative, then
there are constants $C,r>0$ such that
\begin{equation}
v(x)\le Ce^{r|x|},
\qquad x\in \rr d.
\end{equation}
(See \cite{Gro3}.) Hence, by \eqref{eq:WeightFunctionLowerUpperBoundWithv},
if $\omega \in \mascP _E(\rr d)$,
then there are $C,r>0$ such that
$$
C^{-1}e^{-r|x|}
\le \omega (x)
\le Ce^{r|x|},
\qquad x\in \rr d.
$$

\par

If, more restrictive,
\eqref{Eq:DefModerateWeightAlt}
holds for $v(x)= C\eabs x^r$, for some
$C,r>0$, then $\omega$ is called
\emph{polynomially moderate}. The set
of polynomially moderate weights
on $\rr d$ is denoted by
$\mascP (\rr d)$.

\par

\subsection{Fourier type transforms}

\par

We begin to settle our definitions of Fourier transforms and
symplectic Fourier transforms.

\par

We let $\mascS (\rr d)$ denote the Schwartz space, and $\mascS '(\rr d)$ its dual, the space of tempered distributions.
\par

The Fourier transform $\mascF$
is the linear and continuous map on $\mascS '(\rr d)$,
which takes the form
$$
(\mascF f )(\xi )
=
\widehat f(\xi )
\equiv 
(2\pi) ^{-\frac d2}\int _{\rr d}f(x)e^{-i\scal x\xi}\, dx
$$
when $f\in \mascS (\rr d)$. Then $\mascF$
is a homeomorphism on $\mascS '(\rr d)$, which
restricts to a homeomorphism on
$\mascS (\rr d)$,
and to a unitary mapping on $L^2(\rr d)$. We observe that
the inverse of $\mascF$ is given by
$$
\mascF ^{-1} = \mascR \circ \mascF = \mascF \circ \mascR,
$$
where $\mascR$ is the operator on $\mascS '(\rr d)$, 
given by
$$
(\mascR f)(x) = \check f(x)=f(-x),
\quad \text{when}\quad f\in \mascS '(\rr d).
$$
For $f\in \mascS (\rr d)$ one has
$$
f(x)=(\mascF ^{-1}\widehat f)(x)
=
(2\pi) ^{-\frac d2}\int _{\rr d}\widehat 
f(\xi )e^{i\scal x\xi}\, d\xi .
$$
In particular it follows that
\begin{equation}
\label{Eq:FTandReflexion}
\mascF ^2=\mascR .
\end{equation}

\par

The symplectic Fourier transform is the linear and
continuous map on $\mascS '(\rr {2d})$, which takes
the form
$$
(\mascF _\sigma \fka )(X)
\equiv
\pi ^{-d} \int _{\rr {2d}} \fka (Y)e^{-2i\sigma (X,Y)}\, dY,
$$
when $\fka \in \mascS (\rr {2d})$.
Here $X=(x,\xi )\in \rr {2d}$, $Y=(y,\eta )\in \rr {2d}$ and
$$
\sigma (X,Y) = \scal y\xi -\scal x\eta 
$$
is the symplectic form. Similarly as for the usual
Fourier transform, $\mascF _\sigma$ is a homeomorphism
on $\mascS '(\rr {2d})$, which
restricts to a homeomorphism on
$\mascS (\rr {2d})$,
and to a unitary mapping on $L^2(\rr {2d})$.
One way to 
verify this is to observe that
$$
(\mascF _\sigma \fka )(x,\xi )
=
2^d (\mascF^{-1}\otimes \mascF )\fka (2\xi ,2x),
\qquad \fka \in \mascS '(\rr {2d}) . 
$$
% and using the appropriate kernel theorems and the 
% fact that similar mapping
% properties are true for $\mascF$.
From this
identification together with Fourier's
inversion formula, one obtains
\begin{equation}
\label{Eq:Fsigma2IsIdent}
\mascF _\sigma ^2=\operatorname{Id}.
\end{equation}
We notice that by \eqref{Eq:FTandReflexion}, 
the minimal number of iterations of the (ordinary) Fourier transform to reach the 
identity map is $4$. That is, the
formula which corresponds to
\eqref{Eq:Fsigma2IsIdent} for
the ordinary Fourier transform is
$\mascF ^4=\operatorname{Id}$.
Evidently, this is a less strong property
compared to \eqref{Eq:Fsigma2IsIdent}.

\par

The linear map $\mascW :
\mascS '(\rr {2d}) \to \mascS '(\rr {2d})$
is defined by
\begin{equation}
\label{Eq:WigMapDef}
\begin{aligned}
(\mascW \! F)(x,\xi )
\equiv
(\mascF _2G)(x,\xi ),
\quad
G(x,y)
&=
F(x+\textstyle{\frac y2},
x-\textstyle{\frac y2}),
\\[1ex]
F&\in \mascS '(\rr {2d}).
\end{aligned}
\end{equation}
Here $\mascF _2$ is the partial Fourier transform with respect to the
second variable.
This means that $\mascF _2U$ is the Fourier transform
of $U(x,y)$ with respect to the variable $y\in \rr d$.

\par

We observe that $\mascW$ is a homeomorphism on
$\mascS '(\rr {2d})$, which restricts to a
homeomorphism on $\mascS (\rr {2d})$,
and to a unitary mapping on $L^2(\rr {2d})$,
since similar facts are true
for the mappings
\begin{equation}
\label{Eq:WigMapDefSplit}
F(x,y)\mapsto F(x+\textstyle{\frac y2},x-\textstyle{\frac y2})
\quad \text{and}\quad
\mascF _2.
\end{equation}
If $F\in \mascS (\rr {2d})$, then it follows that
$$
(\mascW \! F)(x,\xi ) = (2\pi )^{-\frac d2}
\int _{\rr d} 
F(x+\textstyle{\frac y2},x-\textstyle{\frac y2})
e^{-i\scal y\xi}\, dy.
$$

\par

Suppose that $f,g\in \mascS '(\rr d)$.
Then the Wigner distribution $W_{f,g}$
is the distribution in $\mascS '(\rr {2d})$,
given by
$$
W_{f,g} \equiv \mascW \! (f\otimes \overline g).
$$
By the mapping properties of $\mascW$ here above, it follows that
the map $(f,g)\mapsto W_{f,g}$ is continuous from
$\mascS '(\rr d)\times \mascS '(\rr d)$ to
$\mascS '(\rr {2d})$,
and that similar continuity holds with $\mascS$
and $L^2$, respectively, in place of $\mascS '$,
at each occurrence.
If $f,g\in \mascS (\rr d)$, then it follows that
$$
W_{f,g}(x,\xi ) = (2\pi )^{-\frac d2}\int _{\rr d} 
f(x+\textstyle{\frac y2})\overline {g(x-\textstyle{\frac y2})}e^{-i\scal y\xi}\, dy .
$$

\par

\subsection{Fractional Fourier transforms}

\par

We recall that the fractional Fourier transform
of $f\in \mascS (\rr d)$ of order $\rho \in \rr d$
is given by
\begin{align}
(\mascF _{\! \rho}f)(\xi )
&=
\scal {K_{\rho}(\xi ,\cdo)}f
\intertext{where}
K_{\rho}
&\equiv
\bigotimes _{j=1}^d K_{\rho _j}
\label{Eq:FracFourKernel}
\intertext{with}
K_\varrho (\tau ,t)
&=
\begin{cases}
\left (
\frac {1-i\cot (\frac {\pi \varrho}2)}{2\pi}
\right )^{\frac 12}
\exp \left (
i\cdot \frac {(t^2+\tau ^2)
\cos (\frac {\pi \varrho}2)-2\tau t}
{2\sin (\frac {\pi \varrho}2)}
\right ), & 
\varrho \in \mathbf R\setminus 2\mathbf Z,
\\[1ex]
\delta (\tau -t), &
\varrho \in 4\mathbf Z,
\\[1ex]
\delta (\tau +t), &
\varrho \in 2+4\mathbf Z,
\notag
\end{cases}
\end{align}
and $t,\tau \in \mathbf R$.
(See e.{\,}g. 
\cite{BulSul,ChFuGrWu,dGo2,dGoLue,FanHu,KutOza,Nam,OzaZalKut,Toft18,ToBhMa} and the references
therein.)
If
\begin{equation}
\label{Eq:DefAlphaBeta}
\alpha _\varrho =\cos (\varrho \frac \pi 2)
\quad \text{and}\quad
\beta _\varrho =\sin (\varrho \frac \pi 2),
\end{equation}
then $K_\varrho (\tau ,t)$ can also be
written as
\begin{equation}
\label{Eq:FracFourKernelOneDim}
K_\varrho (\tau ,t)
=
\begin{cases}
\left (
\frac {\beta _\varrho -i\alpha _\varrho}
{2\pi \beta _\varrho}
\right )^{\frac 12}
\exp \left (
i\cdot \frac {(t^2+\tau ^2)
\alpha _\varrho - 2\tau t}
{2\beta _\varrho}
\right ), & 
\varrho \in \mathbf R\setminus 2\mathbf Z,
\\[1ex]
\delta (\tau -t), &
\varrho \in 4\mathbf Z,
\\[1ex]
\delta (\tau +t), &
\varrho \in 2+4\mathbf Z.
\end{cases}
\end{equation}

\par

We set
\begin{equation}
\label{Eq:FracFourNoPartial}
\mascF _{\rho _0} =\mascF _\rho
\quad \text{and}\quad
K_{\rho _0} = K_\rho
\quad \text{when}\quad
\rho = (\rho _0,\dots ,\rho _0)\in \rr d.
\end{equation}

\par

\begin{rem}
We observe that fractional Fourier 
transforms are operators
that appear from the Weil representation 
of symplectic, or the metaplectic groups.
The Weil representation contains several
other types of operators. Common for
these operators are that they are
continuous on $\mascS (\rr d)$ and
$\mascS '(\rr d)$, and unitary on
$L^2(\rr d)$ (see e.{\,}g.
\cite{CorGroNicRod,CorNic,dGo2,dGoLue,Ho1}
and the references therein.)
\end{rem}

\par

\subsection{The Bargmann transform}

\par

The \emph{Bargmann transform} is given by
\begin{align}
(\mathfrak V_{d}f)(z)
&\equiv
\pi ^{-\frac d4}
\int _{\rr {d}}  \mathfrak A_d(z,y) f(y)\, dy , \qquad z \in \cc d ,
\label{Eq:BargTransf}
\intertext{when $f\in \mascS (\rr d)$. Here $\mathfrak A_d$ is the kernel}
\mathfrak A_d(z,y)
&=
\exp \left ( -\frac 12(\scal zz+|y|^2) +\sqrt 2\scal zy \right ) .
\end{align}
The definition of $\mathfrak V _d$ extends
uniquely to $\mascS '(\rr d)$, and we observe
that the image of $\mathfrak V _d$ is contained
in $A(\cc d)$, the space of all entire functions
on $\cc d$.

\par

The Bargmann-Fock space, $A^2(\cc d)$, related to \eqref{Eq:BargTransf}, is
the Hilbert space of all entire functions $F$ on $\cc d$ such that
$$
\nm F{A^2} \equiv \left (
\int _{\cc d} |F(z)|^2\, d\mu _d(z)
\right )^{\frac 12}
$$
is finite. Here $d\mu _d$ is the Gaussian measure
$$
d\mu _d(z) = \pi ^{-d}e^{-|z|^2}d\lambda (z),
\qquad z\in \cc d,
$$
where $d\lambda =d\lambda _d$ is the Lebesgue measure on $\cc d$.
We recall that
\begin{equation}
\label{Eq:BargmannL2Map}
\mathfrak V_d : L^2(\rr d)\to A^2(\cc d)
\end{equation}
is unitary map (see \cite{Ba1}). 

\par

The image of the Bargmann transform on
other spaces can be further specified as well.
For example, the following is contained in
\cite{Ba2}. Here
$\eabs z=(1+|z|^2)^{\frac 12}$ when $z\in \cc d$,
\begin{alignat*}{1}
\maclA _{{}_{{}^\infty}} (\cc d)
&\equiv
\sets {F\in A(\cc d)}{|F(z)|\lesssim e^{\frac 12|z|^2}\eabs z^{-N},\ \text{for every}\ N_0>0},
\intertext{and}
\maclA _{{}_{{}^\infty}} '(\cc d)
&\equiv
\sets {F\in A(\cc d)}{|F(z)|\lesssim e^{\frac 12|z|^2}\eabs z^{N_0},\ \text{for some}\ N_0>0},
\end{alignat*}
with natural projective limit topology for
$\maclA _{{}_{{}^\infty}} (\cc d)$,
and inductive limit topology for
$\maclA _{{}_{{}^\infty}} '(\cc d)$.

\par

\begin{prop}
\label{Prop:BasicBargMap}
The mappings
\begin{equation*}
\mathfrak V_d : \mascS (\rr d)
\to
\maclA _{{}_{{}^\infty}} (\cc d)
\quad \text{and}\quad
\mathfrak V_d : \mascS '(\rr d)
\to
\maclA _{{}_{{}^\infty}} '(\cc d),
\end{equation*}
are homeomorhpisms.
\end{prop}

\par

\subsection{Symplectic Bargmann transform}

\par

An alternative Bargmann transform for even 
dimensions, called the \emph{symplectic Bargmann 
transform} is given by
\begin{align}
(\mathfrak V_{\sigma ,d}\fka )(z,w)
&\equiv
\left (
\frac 2\pi
\right )^{\frac d2}
\iint _{\rr {2d}} \mathfrak A_{\sigma ,d}(z,w,y,\eta )
\fka (y,\eta ) \, dyd\eta , 
\label{Eq:SymplBargTransf}
\intertext{for $z, w \in \cc d$, when $\fka \in \mascS (\rr {2d})$.
Here $\mathfrak A_{\sigma ,d}$
is the kernel}
\mathfrak A_{\sigma ,d}(z,w,y,\eta )
&{\phantom =}
\notag
\\[1ex]
=\exp \Big ( -( &\scal zz+\scal ww+|y|^2+|\eta |^2) 
+\sqrt 8 (\scal zy +\scal w\eta ) \Big )
\end{align}
The definition of $\mathfrak V _{\sigma ,d}$
extends uniquely to $\mascS '(\rr {2d})$,
with image contained in $A(\cc {2d})$.

\par

We observe that
\begin{equation}
\label{Eq:BargSymplBargKernels}
\mathfrak A_{\sigma ,d}(z,w,y,\eta )
=
\mathfrak A_{d}\big ( \sqrt 2\, z,\sqrt 2\, y \big )
\mathfrak A_{d}\big ( \sqrt 2\, w,\sqrt 2\, \eta
\big ),
\end{equation}
which implies that several basic properties for the
Bargmann transform, carry over to the symplectic
Bargmann transform. For example, by
\eqref{Eq:BargmannL2Map} it gives that
the map
\begin{equation}
\label{Eq:BargmannL2MapSympl}
\mathfrak V_{\sigma ,d} : L^2(\rr {2d})
\to
A_\sigma ^2(\cc {2d})
\end{equation}
is unitary. Here, the Hilbert space
$A_\sigma ^2(\cc {2d})$, denotes the
\emph{symplectic} Bargmann-Fock space.
It consists of all entire functions
$F$ on $\cc {2d}$ such that
\begin{equation}
\label{Eq:SymplBargFockNorm}
\nm F{A_\sigma ^2} \equiv \left (
\int _{\cc {2d}} |F(z,w)|^2\, d\mu _{\sigma ,d}(z,w)
\right )^{\frac 12}
\end{equation}
is finite. Above, $d\mu _{\sigma ,d}$ is
the Gaussian measure
\begin{equation}
\label{Eq:SymplMeasure}
d\mu _{\sigma ,d}(z,w) = \left (\frac 2\pi \right ) ^{2d}e^{-2(|z|^2+|w|^2)}d\lambda (z,w),
\qquad z,w\in \cc d.
\end{equation}
A combination
of this fact,
\eqref{Eq:SymplBargFockNorm},
and
\eqref{Eq:SymplMeasure}, gives that
the inverse and adjoint of
$\mathfrak V_{\sigma ,d}$ are given by
\begin{equation}
\label{Eq:SymplBargmInverse}
\begin{aligned}
(\mathfrak V_{\sigma ,d}^{-1}F)(x,\xi )
&=
(\mathfrak V_{\sigma ,d}^*F)(x,\xi )
\\[1ex]
&=
\left (
\frac 2\pi
\right )^{\frac d2}
\int _{\cc {2d}}
\mathfrak A_{\sigma ,d}(\overline z,
\overline w,x,\xi )F(z,w)
d\mu _{\sigma ,d}(z,w).
\end{aligned}
\end{equation}

\par

By Proposition \ref{Prop:BasicSympBargMap}
and \eqref{Eq:BargSymplBargKernels}
we also get the following.

\par

\begin{prop}
\label{Prop:BasicSympBargMap}
The mappings
\begin{equation*}
\mathfrak V_{\sigma ,d} : \mascS (\rr {2d})
\to
\maclA _{{}_{{}^\infty}} (\cc {2d})
\quad \text{and}\quad
\mathfrak V_{\sigma ,d} : \mascS '(\rr {2d})
\to
\maclA _{{}_{{}^\infty}} '(\cc {2d}),
\end{equation*}
\end{prop}

\par

\subsection{Gelfand triples and Quasi-Banach function
spaces}\label{subsec1.5}

\par

A convenient and condensed way to describe 
certain
relationships between two topological vector 
spaces and an
intermediate Hilbert space, is to use so-called 
Gelfand triples.

\par

\begin{defn}
\label{Def:GelfandTripple}
Let
$\mascV _1,\mascV _2$ be
topological vector spaces
and $\maclH$ be a Hilbert space such that
$$
\mascV _1\hookrightarrow \maclH \hookrightarrow
\mascV _2.
$$
Then
$(\mascV _1,\maclH ,\mascV _2)$ is called a 
\emph{Gelfand triple}, if the following is
true:
\begin{itemize}
\item there is a hypo-continuous map
$(u,v)\mapsto (u,v)_{\maclH}$ from
$\mascV _2\times \mascV _1$
to $\mathbf C$ such that
$u\mapsto (u,v)_{\maclH}$ is linear
for every $v\in \mascV _1$ and
$v\mapsto (u,v)_{\maclH}$ is conjugate
linear for every $u\in \mascV _2$;

\vrum

\item $\mascV _1$ is dense in $\maclH$;

\vrum

\item if $u,v\in \maclV _1$, then
$(u,v)_{\maclH}$ is the scalar product
of $u$ and $v$ with respect to $\maclH$.
\end{itemize}
\end{defn}

\par

\begin{rem}
For the Gelfand triple in Definition 
\ref{Def:GelfandTripple} we observe the
following:
\begin{itemize}
\item
the map $(u,v)\mapsto 
(u,v)_{\maclH}$ from
$\mascV _1\times \mascV _1$ to $\mathbf C$ 
extends uniquely to a continuous
map from $\maclH\times \maclH$
to $\mathbf C$, and that this
extension agrees with the scalar product in
$\maclH$.

\vrum

\item
$\mascV _2$ does not necessarily be
isomorphic with the dual $\mascV _1'$
of $\mascV _1$,
but is always continuously embedded 
in $\mascV _1'$, when using
the conjugate of the form
$(\cdo ,\cdo )_{\maclH}$.
\end{itemize}
\end{rem}

\par

Next we recall the definitions of quasi-normed spaces and quasi-Banach
spaces. The functional
$\nm \cdo{\mascB}$, defined on the
vector space $\mascB$ is called a
\emph{quasi-norm} of order
$r _0 \in (0,1]$, or an \emph{$r_0$-norm}, if the conditions
\begin{alignat}{2}
\nm {f+g}{\mascB}
&\le
2^{\frac 1 {r _0} -1}
(\nm {f}{\mascB} + \nm {g}{\mascB}),
& \quad f,g &\in \mascB ,
\label{Eq:WeakTriangle1}
\\[1ex]
\nm {\alpha \cdot f}{\mascB}
&=
|\alpha| \cdot \nm f{\mascB}\ge 0,
& \quad \alpha &\in \mathbf{C},
\quad  f \in \mascB ,
\notag
\intertext{and}
\nm f{\mascB} &= 0
\quad  \Leftrightarrow \quad
f=0, & &
\notag
\end{alignat}
hold true. A vector space together with 
a quasi-norm (of order $r_0$) on it
is called a quasi-norm space (of order 
$r_0$), or an $r_0$-normed space.
A complete quasi-norm space
(of order $r_0$) is called a
\emph{quasi-Banach space} (of order $r_0$),
or an \emph{$r_0$-Banach space}.
If $\mascB$ is a quasi-Banach space 
with quasi-norm satisfying 
\eqref{Eq:WeakTriangle1}
then by \cite{Aik,Rol} there is an 
equivalent quasi-norm to $\nm \cdo 
{\mascB}$ of order $r_0$, which 
additionally satisfies
\begin{align}\label{Eq:WeakTriangle2}
\nm {f+g}{\mascB} ^{r _0} \le \nm {f}
{\mascB} ^{r _0} + \nm {g}{\mascB} 
^{r _0}, 
\quad f,g \in \mascB .
\end{align}
From now on, we always assume that 
the  quasi-norm of the quasi-Banach 
space  $\mascB$ is chosen in such a 
way that both 
\eqref{Eq:WeakTriangle1} and 
\eqref{Eq:WeakTriangle2} hold.

\par

The quasi-Banach
space $\mascB$ is called a
\emph{quasi-Banach function
space} (or \emph{QBF space}) on
$\rr d$, whenever $\mascB$ is
contained in
$\maclM (\rr d)$, the space of measurable functions.
It is here
understood that the zero element
in $\maclM (\rr d)$ (which is
an equivalent class), is the zero
element in $\mascB$.

\par

We always assume that the QBF spaces
are \emph{solid}, as described
in the following definition.
(See, e.g., (A.3) in \cite{Fei3A}.)

\par

\begin{defn}
\label{Def:SolidBF}
Let $\mascB$ be a
\emph{QBF space} on
$\rr d$.
Then $\mascB$ is
called \emph{solid} if the
following conditions hold:
\begin{enumerate}
\item
for any
$f,g\in \maclM (\rr d)$
which satisfy
$g\in \mascB$ and $|f|
\le |g|$, then $f\in \mascB$ and
\begin{equation}\label{Eq:Solid}
\nm f{\mascB}\le C\nm g{\mascB},
\end{equation}
for some constant $C>0$, independent of $f$ and $g$ above;

\vrum

\item
for any $x\in \rr d$,
there is a constant $c>0$ and a function
$f\in \mascB$ such that $|f|>c$ on a non-null convex
set in $\rr d$ containing $x$.
\end{enumerate}
The space space $\mascB$ is called \emph{strict solid}
if it is solid and \eqref{Eq:Solid} holds for $C=1$.
\end{defn}

\par

In the following definition, we put restrictions
on QBF spaces, which are often needed.

\par

\begin{defn}\label{Def:BFSpaces}
Let $\mascB$ be a solid QBF space of
order $r_0\in (0,1]$ on $\rr d$,
$C>0$ be the same as in 
\eqref{Eq:Solid},
and let $v_0\in \mascP _E(\rr d)$
be submultiplicative.
\begin{enumerate}
\item The space $\mascB$
is called a
\emph{solid translation invariant
quasi-Banach function space on
$\rr d$} 
(with respect to $r_0$ and $v_0$), or 
\emph{invariant QBF space on $\rr d$},
if for any $x\in \rr d$ and
$f\in \mascB$, then $f(\cdo -x)\in
\mascB$, and 
\begin{equation}\label{translmultprop1}
\nm {f(\cdo -x)}{\mascB}
\le
Cv_0(x)\nm {f}{\mascB}\text ;
\end{equation}

\vrum

\item The space $\mascB$
is called a \emph{solid translation 
invariant Banach function space on
$\rr d$} 
(with respect to $v_0$), or 
\emph{invariant BF space on $\rr d$},
if $\mascB$ is a Banach space and
an invariant QBF space on $\rr d$
with respect to $v_0$.

\vrum

\item The space $\mascB$
is called a \emph{solid translation 
invariant normal quasi-Banach function 
space on $\rr d$} (with respect to $r_0$ and $v_0$), or 
\emph{a normal QBF space on $\rr d$},
if there is an invariant BF space
$\mascB _0$ on $\rr d$
with respect to $v_0^{r_0}$ such that
\begin{equation}
\label{Eq:NormQBFNorm}
\nm f{\mascB}
\equiv
\nm {\, |f|^{r_0}\, }{\mascB _0}^{1/r_0},
\quad
f\in \mascB .
\end{equation}
\end{enumerate}
\end{defn}

\par

\subsection{Sequence spaces}
\label{subsec1.6}

\par

The set of all $d$-variate sequences in $\mathbf C$ is denoted
by $\ell _0'(\nn d)$, that is,
\begin{align*}
\ell _0'(\nn d)
&=
\sets { c = \{ c(\alpha )\} _{\alpha \in \nn d}}{c(\alpha )\in \mathbf C,\ \alpha \in \nn d}.
\intertext{Also, let}
\ell _{0,[N]}(\nn d)
&=
\sets {c = \{ c(\alpha )\} _{\alpha \in \nn d}\in \ell _0'(\nn d)}{c(\alpha )=0
\ \text{when}\ |\alpha |>N,\, \alpha \in \nn d},
\intertext{and}
\ell _0(\nn d) &= \bigcup _{N\ge 0} \ell _{0,[N]}(\nn d)
\end{align*}
when $N\ge 0$ is an integer.
We equip $\ell _0'(\nn d)$ with the topology derived from the semi-norms
\begin{equation}
\label{Eq:SemiNormBigPilSp}
\nm c{p,[N]}
\equiv 
\nm {\{ 1_{[0, N]^d}(\alpha) c(\alpha)\}_{\alpha \in \nn d}}{\ell^p},
\qquad
c \in \ell _0'(\nn d),\ N\in \mathbf N,
\end{equation}
when $p\in (0,\infty ]$. 
Here, $1_\Omega$
is the characteristic function of
the set $\Omega \subseteq \nn d$.
%Here $\overline{B_r}$
%is the closed ball in $\rr d$, centered at
%origin and radius $r\ge 0$.
%We let the topology on $\ell _0'(\nn d)$ be defined by the
%semi-norms in \eqref{Eq:SemiNormBigPilSp}.
It follows that
the topology on $\ell _0'(\nn d)$ is independent of the choice of $p$,
and by choosing $p\ge 1$, it follows that $\ell _0'(\nn d)$ is a Fr{\'e}chet
space under this topology.
%, and that
%$\ell _0'(\nn d)$ is a Fr{\'e}chet space.
For $N\in \mathbf N$ fixed, one has that $\ell _{0,[N]}(\nn d)$
is a quasi-Banach space
under the norm
\eqref{Eq:SemiNormBigPilSp}, and that different of choices of $p$ gives
rise to equivalent quasi-norms. Again, by choosing $p\ge 1$ it follows that
$\ell _{0,[N]}(\nn d)$ is a Banach space.

\par

We equip $\ell _{0}(\nn d)$ with the inductive limit topology of the spaces
$\ell _{0,[N]}(\nn d)$. Evidently, the map $(a,b)\mapsto (a,b)_{\ell ^2}$ from
$\ell _0(\nn d)\times \ell _0(\nn d)$ to $\mathbf C$, extends uniquely to
a hypo-continuous map from 
$\ell _0'(\nn d)\times \ell _0(\nn d)$ to $\mathbf C$.
It follows that
\begin{equation}
\label{Eq:Gelftriple1}
(\ell _0(\nn d), \ell ^2(\nn d),\ell _0'(\nn d))
\end{equation}
is a Gelfand triple. In particular, the (strong) dual of $\ell _0(\nn d)$ can be
identified by $\ell _0'(\nn d)$ through the form $(\cdo ,\cdo )_{\ell ^2}$.

\par

An important type of subspace of $\ell _0'(\nn d)$
concerns the (discrete) Lebesgue space $\ell _{[\vartheta ]}^p(\nn d)$,
which is paramerized by the positive (fixed) function $\vartheta$
on $\nn d$ and $p\in (0,\infty ]$, and consists of all
$c$ in \eqref{Eq:SemiNormBigPilSp}
such that
$$
\nm c{\ell _{[\vartheta ]}^p}
\equiv
\nm {\{ c(\alpha )\vartheta (\alpha )\} _{\alpha \in \nn d}}{\ell ^p}
$$
is finite. We let $\ell _{[\vartheta ]}^\sharp (\nn d)$ be the set of
all $c\in \ell _{[\vartheta ]}^\infty (\nn d)$ such that
$$
\lim _{|\alpha | \to \infty} c(\alpha )\vartheta (\alpha )=0,
\quad
c=\{ c(\alpha )\} _{\alpha \in \nn d}.
$$
One has that $\ell _{[\vartheta ]}^p(\nn d)$ is a quasi-Banach space.
If, more restrictive, $p\in [1,\infty ]$, then $\ell _{[\vartheta ]}^p(\nn d)$
and $\ell _{[\vartheta ]}^\sharp (\nn d)$ are Banach spaces.
We also emphasize the inclusions
\begin{align*}
\ell _{[\vartheta ]}^{p_1}(\nn d)
\subsetneq 
\ell _{[\vartheta ]}^{p_2}(\nn d)
\subsetneq 
\bigcup _{p\in (0,\infty )}\ell _{[\vartheta ]}^{p}(\nn d)
&\subsetneq 
\ell _{[\vartheta ]}^\sharp (\nn d)
\subsetneq 
\ell _{[\vartheta ]}^\infty (\nn d),
\\[1ex]
0&<p_1<p_2<\infty .
\end{align*}

\par

For conveniency we set
$$
\ell ^{p}=\ell _{[\vartheta ]}^{p}
\quad \text{and}\quad
\ell ^{\sharp}=\ell _{[\vartheta ]}^{\sharp}
\quad \text{when}\quad
\vartheta \equiv 1.
$$

\par

\subsection{Schatten-von Neumann classes}

\par

Let $\maclL (V_1,V_2)$ denote the set of linear and continuous
mappings from the topological vector space $V_1$ to
the topological vector space $V_2$.
If $V$ is a quasi-Banach space, then
$\maclL (V_1,V_2)$ is equipped with the
topology induced by the operator norms. We
set $\maclL (V)=\maclL (V,V)$, when $V=V_1=V_2$.
We also write 
$V_1 \hookrightarrow V_2$ when
$V_1 \subseteq V_2$ with
continuous inclusion.

\par

Let $T\in \maclL (L^2(\rr d))$ and
$\ON =\ON (L^2(\rr d))$ be the set of all orthonormal sequences in $L^2(\rr d)$.
Then $T$ is called a Schatten-von Neumann
operator of order $p\in (0,\infty ]$ (on $L^2(\rr d)$) if
\begin{equation}
\label{Eq:SchattNorm1}
\nm T{\mascI _p}
\equiv
\sup \nm {\{(Tf_j,g_j)_{L^2}\} _{j\in \mathbf N}}{\ell ^p}
\end{equation}
is finite. Here the supremum is taken over all $\{f_j\} _{j\in \mathbf N}\in \ON$
and $\{g_j\} _{j\in \mathbf N}\in \ON$. The set of all Schatten-von Neumann
operators of order $p\in (0,\infty ]$ on $L^2(\rr d)$ is denoted by $\mascI _p$.
We also let $\mascI _\sharp$ be the set of all $T\in \maclL (L^2(\rr d))$
such that
$$
\{(Tf_j,g_j)_{L^2}\} _{j\in \mathbf N}\in \ell ^\sharp (\mathbf N),
$$
for every choices of $\{f_j\} _{j\in \mathbf N}\in \ON$
and $\{g_j\} _{j\in \mathbf N}\in \ON$. It follows that
$$
\mascI _\infty ,\quad \mascI _\sharp ,\quad \mascI _2 ,
\quad \text{and}\quad \mascI _1 ,
$$
and their norms, are the sets of continuous, compact, Hilbert-Schmidt,
and trace-class operators, respectively, and their respective norms.

\par

We remark that Schatten-von Neumann classes are usually defined by
suitable estimates on the singular values of the involved operators. However,
for the operator classes above, this leads to the same classes, as well as the same 
norms (see, e.{\,}g., \cite{Sim}).

\par

In the following we introduce a broader family of Schatten-von Neumann classes.

\par

\begin{defn}
\label{Def:ExtendedSchattFamily}
Let $\maclB \hookrightarrow \ell ^\infty (\mathbf N)$ be a topological vector space, $T\in \mascI _\infty$,
and
\begin{equation}
\label{Eq:OmegaTDef}
\Omega _T =\sets {\{ (Tf_j,g_j ) _{L^2}\}_{j=0}^\infty }
{\{ f_j\}_{j=0}^\infty , \{ g_j\}_{j=0}^\infty \in \ON}.
\end{equation}
\begin{enumerate}
\item $T$ is called a Schatten-von Neumann
operator of order $\maclB$, if $\Omega _T$ is a bounded subset of
$\maclB$.
%there is a bounded set $\Omega \subseteq \maclB$
%such that
%$$
%\{ (Tf_j,g_j ) _{L^2}\}_{j=0}^\infty \in \Omega 
%\quad \text{when}\quad
%\{ f_j\}_{j=0}^\infty \in \ON
%\ \text{and}\ 
%\{ g_j\}_{j=0}^\infty \in \ON .
%$$
The set of Schatten-von Neumann operators of order $\maclB$ is denoted
by $\mascI _\maclB$.

\vrum

\item If in addition $\maclB$ is a quasi-Banach space with quasi-norm
$\nm \cdo{\maclB}$, then $\mascI _\maclB$ is equipped with the
topology, defined by the quasi-norm
\begin{equation}
\label{Eq:SchattNorm}
\nm T{\mascI _\maclB}
\equiv
\sup _{c\in \Omega _T}\nm c{\maclB},
\qquad
T\in \mascI _\maclB .
\end{equation}
\end{enumerate}
\end{defn}

\par

Evidently, \eqref{Eq:SchattNorm} is the same as
\begin{equation}
\tag*{(\ref{Eq:SchattNorm})$'$}
\label{Eq:SchattNormAlt}
\nm T{\mascI _\maclB}
=
\sup \nm {\{(Tf_j,g_j)_{L^2}\} _{j\in \mathbf N}}{\maclB},
\qquad
T\in \mascI _\maclB ,
\end{equation}
where the supremum is taken over all
$\{f_j\} _{j\in \mathbf N}\in \ON$
and $\{g_j\} _{j\in \mathbf N}\in \ON$.
Furthermore, it is evident that
\begin{equation}
\mascI _p = \mascI _{\ell ^p}
\quad \text{and}\quad
\mascI _\sharp = \mascI _{\ell ^\sharp},
\qquad p\in (0,\infty ],
\end{equation}
also in (quasi-)norms.

\par

The space $\mascI _\maclB$ in Definition
\ref{Def:ExtendedSchattFamily} inherits several
properties from $\maclB$, which is indicated
in the following proposition.
In particular, we may endow $\mascI _\maclB$ with the topology induced by $\maclB$.

\par

\begin{prop}
Let $\maclB \hookrightarrow \ell^\infty (\mathbf N)$ be a topological vector space.
\begin{enumerate}
\item $\mascI _\maclB$ is a vector space.

\vrum

\item If in addition $\maclB$ is a quasi-Banach space,
then $\mascI _\maclB$ is a quasi-Banach space of the same
order as $\maclB$.
\end{enumerate}
\end{prop}

\par

\begin{proof}
The assertion (1) follows from the fact that bounded sets remain bounded under continuous operations.

\par

Suppose that $\maclB$ is a quasi-Banach space of order $r>0$, and
let $T_1,T_2\in \mascI _\maclB$, $\{ f_j\} _{j\in \mathbf N}\in \ON$,
and
$\{ g_j\} _{j\in \mathbf N}\in \ON$. Then
\begin{align*}
\nm { \{ ((T_1+T_2)f_j,g_j)\} _{j\in \mathbf N}}{\maclB}^r
&\le
\nm { \{ (T_1f_j,g_j)\} _{j\in \mathbf N}}{\maclB}^r
+
\nm { \{ (T_2f_j,g_j)\} _{j\in \mathbf N}}{\maclB}^r
\\[1ex]
&\le
\nm {T_1}{\mascI _\maclB}^r + \nm {T_2}{\mascI _\maclB}^r.
\end{align*}
By taking the supremum over all possible 
$\{ f_j\} _{j\in \mathbf N}\in \ON$
and
$\{ g_j\} _{j\in \mathbf N}\in \ON$ of the left-hand side, one obtains
$$
\nm {T_1+T_2}{\mascI _\maclB}^r
\le
\nm {T_1}{\mascI _\maclB}^r + \nm {T_2}{\mascI _\maclB}^r,
$$
which shows that $\mascI _\maclB$ is a quasi-normed space
of the same order as $\maclB$.

\par

We need to prove that $\mascI _\maclB$ is complete.
Suppose that $\{T_k \} _{k\in \mathbf N}$ is a Cauchy sequence
in $\mascI _\maclB$. Since $\maclB \hookrightarrow \ell ^\infty (\mathbf N)$,
it follows that $\mascI _\maclB \hookrightarrow \mascI _\infty$,
which in turn implies that $\{T_k \} _{k\in \mathbf N}$ is a Cauchy sequence
in $\mascI _\infty$, also. Since $\mascI _\infty$ is complete, there is a unique
$T\in \mascI _\infty$ such that
\begin{equation}
\label{Eq:OpNormLimit}
\lim _{k\to \infty} \nm {T_k-T}{\mascI _\infty}=0.
\end{equation}

\par

Now let $\{ f_j\} _{j\in \mathbf N}\ON$ and $\{ g_j\} _{j\in \mathbf N}\in \ON$
be fixed, and let
\begin{alignat*}{3}
\kappa &= \{\kappa (j)\} _{j\in \mathbf N} &
\quad &\text{and} & \quad
\kappa _k&= \{\kappa _k(j)\} _{j\in \mathbf N},
\intertext{where}
\kappa (j) &= (Tf_j,g_j)_{L^2} &
\quad &\text{and} & \quad
\kappa _{k}(j) &= (T_kf_j,g_j)_{L^2}.
\end{alignat*}
Then \eqref{Eq:OpNormLimit} shows that
\begin{equation}
\label{Eq:kappaLimit}
\lim _{k\to \infty} \kappa _k= \kappa ,
%\quad \text{uniformly with respect 
%$\{ f_j\} _{j\in \mathbf N},\{ g_j\} _{j\in \mathbf N}\in \ON$}.
\end{equation}
with convergence in $\ell ^\infty (\mathbf N)$.
Since $T_k\in \mascI _\maclB$ and since $\maclB$
is a quasi-Banach space,
one has 
\begin{equation}
\label{Eq:AlternativeLimit}
\lim _{k\to \infty}
\nm {\kappa _k-\tilde \kappa }{\maclB}=0,
\end{equation}
for some $\tilde \kappa \in \maclB$, which depends on the choice of
$\{ f_j\} _{j\in \mathbf N},\{ g_j\} _{j\in \mathbf N}\in \ON$. Moreover,
since $\{T_k\}_{k \in \mathbb N}$ is a Cauchy sequence in $\mascI_\maclB$,
it follows that \eqref{Eq:AlternativeLimit} holds uniformly with
respect to $\{ f_j\} _{j\in \mathbf N}$ and $\{ g_j\} _{j\in \mathbf N}\in \ON$.

\par

A combination of \eqref{Eq:kappaLimit}, \eqref{Eq:AlternativeLimit},
and the fact that $\maclB \hookrightarrow \ell ^\infty (\mathbf N)$,
gives $\tilde \kappa = \kappa$, and since
\eqref{Eq:AlternativeLimit} holds uniformly with
respect to $\{ f_j\} _{j\in \mathbf N}$ and
$\{ g_j\} _{j\in \mathbf N}\in \ON$, one obtains
$$
\lim _{k\to \infty} \nm {T_k-T}{\mascI _\maclB}=0.
$$
This gives the result.
\end{proof}

\par

\subsection{The Weyl quantization of pseudo-differential operators}
\label{subsec1.4}

\par

For any $\fka \in \mascS (\rr {2d})$, its Weyl quantization $\op ^w(\fka )$
is the linear and continuous operator on $\mascS (\rr d)$, which is given by
$$
(\op ^w(\fka )f)(x)
=
(2\pi )^{-d}\iint _{\rr {2d}}\fka (\textstyle {\frac 12}(x+y),\xi )
f(y)e^{i\scal {x-y}\xi}\, dyd\xi .
$$
The definition of $\op ^w(\fka )$ is uniquely extendable to a continuous map from
$\mascS '(\rr d)$ to $\mascS (\rr d)$.

\par

If $\fka \in \mascS '(\rr {2d})$, then $\op ^w(\fka )$ is
the linear and continuous
operator from $\mascS (\rr d)$ to $\mascS '(\rr d)$,
with the kernel
$$
K_{\fka}(x,y) = (2\pi )^{-\frac d2}(\mascF _2^{-1}\fka ) 
(\textstyle {\frac 12}(x+y),x-y).
$$
Since $\mascF _2$ and the mapping
$$
U(x,y)\mapsto U(\textstyle {\frac 12}(x+y),x-y)
$$
are homeomorphisms on $\mascS '(\rr {2d})$, it follows that the map
$\fka \mapsto K_\fka$ is a homeomorphism on
$\mascS '(\rr {2d})$.
Since any linear and continuous operator from
$\mascS (\rr d)$ to
$\mascS '(\rr d)$ has a unique (Schwartz) kernel in
$\mascS '(\rr {2d})$, we get the following. 

\par

\begin{prop}
\label{Prop:SchwartzKerWeylOp}
The map
\begin{equation}
\label{Eq:SymbToWeylOp}
\fka \mapsto \op ^w(\fka ).
\end{equation}
is a bijection
from $\mascS '(\rr {2d})$
to $\maclL (\mascS (\rr d),\mascS '(\rr d))$.
\end{prop}

\par

Several aspects of Weyl quantizations are described
by Wigner distributions. Firstly, it follows by straight-forward
computations that the action of Weyl quantization for
$\fka \in L^2(\rr {2d})$ is given by
\begin{equation}
\label{Eq:WeylWigner1}
(\op ^w(\fka )f,g)_{L^2(\rr d)} 
=
(2\pi )^{-\frac d2}(\fka ,W_{g,f})_{L^2(\rr {2d})}
\end{equation}
As a consequence of this we observe that $T=\op ^w(\fka )$, for $\fka \in \mascS^\prime(\rr {2d})$, is a rank-one
operator, if and only if
$$
\fka =W_{f,g},
$$
for some $f,g\in L^2(\rr d)$. Hence, $T=\op ^w(\fka )$ is of rank at most $N\ge 1$,
if and only if $\fka$ belongs to $s_{0,N}^w(\rr {2d})$, the set of all linear combinations
\begin{equation}
\label{Eq:FiniteRankSymb}
\fka = \sum _{j=1}^N c_jW_{f_j,g_j},\qquad
c_j\in \mathbf C,\ f_j,g_j\in L^2(\rr d),\ j=1,\dots ,N.
\end{equation}
We also set $s_{0,0}^w(\rr {2d}) =\{ 0 \}$.

\medspace

For any topological vector space $\maclB \hookrightarrow \ell ^\infty (\mathbf N)$
we set
$$
s_\maclB ^w(\rr {2d})
\equiv
\sets {\fka \in \mascS '(\rr {2d})}{\op ^w(\fka )\in \mascI _\maclB} ,
$$
and endow $s_\maclB ^w(\rr {2d})$ with the topology induced by $\mascI _\maclB$.
If $\maclB$ is a quasi-Banach space, then we put
\begin{alignat}{2}
\nm \fka{s_\maclB ^w}
&\equiv
\nm {\op ^w(\fka )}{\mascI _\maclB}, &
\quad
a &\in \mascS '(\rr {2d}),
\label{Eq:SchattSymbNorm}
\intertext{and let  the topology of $s_\maclB ^w(\rr {2d})$
be defined through this quasi-norm.
\newline
\indent
In the special case $\maclB =\ell ^p(\mathbf N)$, for some
$p\in (0,\infty ]$, we set $s_p^w=s_\maclB ^w$.
By \eqref{Eq:SchattSymbNorm} it follows that}
%and let the quasi-norm on $s_p^w(\rr {2d})$ be defined by the formula
\nm \fka{s_p^w}
&=
\nm {\op ^w(\fka )}{\mascI _p},&
\quad
a &\in \mascS '(\rr {2d}).
\notag
\end{alignat}
In the same way, we let
$s_\sharp ^w=s_{\maclB}^w$ when
$\maclB =\ell ^\sharp (\mathbf N)$.
We also define
$$
s _\flat^w(\rr {2d}) =\bigcup _{j=0}^\infty 
s_{0,j}^w(\rr {2d}).
$$

\par

By Proposition \ref{Prop:SchwartzKerWeylOp}
we get the following.

\par

\begin{prop}
Let $\maclB \hookrightarrow \ell ^\infty (\mathbf N)$ be a topological vector space.
\begin{enumerate}
\item The map \eqref{Eq:SymbToWeylOp}
is an isomorphism from $s_\maclB ^w(\rr {2d})$
to $\mascI _\maclB$.

\vrum

\item If in addition
$\maclB$ is a quasi-Banach space, then
the map \eqref{Eq:SymbToWeylOp} is an
isometric bijection from $s_\maclB ^w(\rr {2d})$
to $\mascI _\maclB$.
\end{enumerate}
\end{prop}

\par

Combining \eqref{Eq:WeylWigner1} with the previous definitions and proposition, we get the following.

\par

\begin{prop}
\label{Prop:SchatSymbChar}
Let $\maclB \hookrightarrow \ell ^\infty (\mathbf N)$ be a topological vector space,
$\fka \in s_\infty ^w(\rr {2d})$, and
\begin{equation}
\label{Eq:OmegaaDef}
\Omega _\fka = \sets {\{ (\fka ,W_{f_j,g_j})_{L^2} \} _{j\in \mathbf N}}
{\{ f_j\} _{j\in \mathbf N},\{ g_j\} _{j\in \mathbf N}\in \ON}.
\end{equation}
\begin{enumerate}
\item $\fka \in s_\maclB ^w(\rr {2d})$, if and only if
$\Omega _\fka$ is a bounded subset of $\maclB$.

\par

\item If, in addition, $\maclB$ is a quasi-Banach space,
then
\begin{equation}
\label{Eq:SchattSymbNorm2}
\nm \fka{s_\maclB ^w} = (2\pi )^{-\frac d2}
\sup _{c\in \Omega _{\fka}}\nm c{\maclB},
\end{equation}
where the supremum is taken over all
$\{ f_j\} _{j\in \mathbf N}$ and $\{ g_j\} _{j\in \mathbf N}\in \ON$.
\end{enumerate}
\end{prop}

\par

We observe that \eqref{Eq:SchattSymbNorm2} is the same as
\begin{equation}
\tag*{(\ref{Eq:SchattSymbNorm2})$'$}
\label{Eq:SchattSymbNormAlt2}
\nm \fka{s_\maclB ^w} = (2\pi )^{-\frac d2}
\sup \nm {\{ (\fka ,W_{f_j,g_j})_{L^2} \} _{j\in \mathbf N}}
{\maclB},
\end{equation}
where the supremum is taken over all
$\{ f_j\} _{j\in \mathbf N}$ and $\{ g_j\} _{j\in \mathbf N}\in \ON$.

\par

\subsection{The Weyl product and the twisted
convolution}

\par

For any $T_1, T_2 \in L(\mascS (\rr d), \mascS '(\rr d))$, Proposition
\ref{Prop:SchwartzKerWeylOp} shows that
there are unique symbols $\fka _1$
and $\fka _2$ such that
$T_j=\op ^w(\fka _j)$,
$j=1,2$. The \emph{Weyl product}
$\fka _1\wpr \fka _2$ of $\fka _1$
and $\fka _2$ is defined by
\begin{equation}
\label{Eq:WeylProdDef}
\op ^w(\fka _1\wpr \fka _2)
=
\op ^w(\fka _1) \circ \op ^w(\fka _2),
\end{equation}
provided the right-hand side makes 
sense as an element of $L(\mascS (\rr d), \mascS '(\rr d))$.

\par

The Weyl product can, in convenient ways,
be described by twisted convolutions and
symplectic Fourier transforms. In fact,
for any $\fka _j\in \mascS (\rr {2d})$,
$j=1,2$, the twisted convolution between
$\fka _1$ and $\fka _2$ is given by
$$
(\fka _1*_\sigma \fka _2)(X)
\equiv
\left (
\frac 2\pi 
\right )^{\frac d2}
\int_{\rr {2d}} \fka _1(X-Y)\fka _2(Y)
e^{2i\sigma (X,Y)}\, dY,
\qquad X\in \rr {2d}.
$$
It follows that the mappings
$(\fka _1,\fka _2)\mapsto
\fka _1*_\sigma \fka _2$
and
$(\fka _1,\fka _2)\mapsto
\fka _1\wpr \fka _2$
are continuous from
$\mascS (\rr {2d})\times \mascS (\rr {2d})$
to
$\mascS (\rr {2d})$,
and uniquely extendable to
continuous mappings from 
$\mascS '(\rr {2d})\times \mascS (\rr {2d})$
or
$\mascS (\rr {2d})\times \mascS '(\rr {2d})$
to
$\mascS '(\rr {2d})$.

\par

The links between Weyl products, symplectic
Fourier transforms, and twisted convolutions
are then explained by the formulae
\begin{equation}
\label{Eq:TwistConvWeylProdLinks1}
\begin{aligned}
\fka _1\wpr \fka _2
&=
(2\pi )^{-\frac d2}
\fka _1*_\sigma (\mascF _\sigma \fka _2)
=
(2\pi )^{-\frac d2}
(\mascF _\sigma \fka _1)*_\sigma \check \fka _2,
\\[1ex]
\mascF _\sigma(\fka _1\wpr \fka _2)
&=
(\mascF _\sigma \fka _1)\wpr \fka _2
=
\check \fka _1 \wpr (\mascF _\sigma \fka _2),
\\[1ex]
\mascF _\sigma(\fka _1\wpr \fka _2)
&=
(2\pi )^{-\frac d2}
(\mascF _\sigma \fka _1)
*_\sigma
(\mascF _\sigma\fka _2),
\end{aligned}
\end{equation}
when $\fka _1,\fka _2\in \mascS '(\rr {2d})$
satisfy $\fka _1\in \mascS (\rr {2d})$
or
$\fka _2\in \mascS (\rr {2d})$. Here and in
what follows we let $\check \fka (X)=\fka (-X)$.

\par

For future references we observe that a
combination of the equalities in
\eqref{Eq:TwistConvWeylProdLinks1} gives
\begin{equation}
\label{Eq:TwistConvWeylProdLinks2}
\begin{alignedat}{2}
\phi _1(\ep _1\cdo )&\wpr \check \fka \wpr 
\phi _2(\ep _2\cdo )
=
%(2\pi )^{-d}
\psi _{1,\ep _1}*_\sigma \fka
*_\sigma \psi _{2,\ep _2},&
\quad
\ep _1,\ep _2 &> 0,
\\[1ex]
\psi _j
&=
(2\pi )^{-\frac d2}\mascF _\sigma \phi _j, 
\quad
\psi _{j,\ep}
=
\ep ^{-2d}\psi _j(\ep ^{-1}\cdo ), &
\quad
\ep &>0,
\end{alignedat}
\end{equation}
and that
\begin{equation}
\label{Eq:TwistConvWeylProdLinks2A}
\phi _j(0)=1
\quad \Leftrightarrow \quad
\int _{\rr {2d}}\psi _j(X)\, dX
=
\left (
\frac \pi 2
\right )^{\frac d2}.
\end{equation}

\par

%%%%%%%%%%%%%%%%%%%%%%%%%%%%%%%%%%%%%%%%%
\section{Fractional symplectic Fourier transform}
\label{sec2}
%%%%%%%%%%%%%%%%%%%%%%%%%%%%%%%%%%%%%%%%%

\par

In this section, we introduce the (partial) fractional symplectic Fourier transform.
Following this, we deduce explicit formulae for the kernel of the (partial) fractional symplectic
Fourier transform. 
We start this section with a review
of the Bargmann images of fractional Fourier transforms.

\par

Let $f\in \mascS '(\rr d)$.
It is well-known that 
\begin{equation}
\label{Eq:BargmImFT}
\mathfrak V_d(\mascF f)(z)
=
\mathfrak V_d(f)(-iz)
=
\mathfrak V_d(f)(e^{-i\frac \pi 2}z).
\end{equation}
More generally, if $\mascF _\rho$ is the fractional Fourier transform of
order $\rho \in \mathbf R$, then
$$
\mathfrak V_d(\mascF _\rho f)(z)
=
\mathfrak V_df(e^{-i\rho \frac \pi 2}z)
=
\mathfrak V_df(A_\rho ^0z),
$$
where
$$
A_\rho ^0
\equiv
e^{-i\rho \frac \pi 2}I_d,
\quad 
\rho \in \mathbf R.
$$
In the same way, the partial fractional Fourier transform of
order $\rho \in \rr d$ can be expressed as
\begin{equation}
\mathfrak V_d(\mascF _\rho f)(z)
=
\mathfrak V_df(A_\rho ^0z),
\label{Eq:BargmannPartialFracFT}
\end{equation}
where
$$
A_\rho ^0
\equiv
\left (
\begin{matrix}
e^{-i\rho _1\frac \pi 2} & \dots & 0
\\
\vdots & \ddots & \vdots
\\
0 & \dots & e^{-i\rho _d\frac \pi 2}
\end{matrix}
\right ),
\quad
\rho \in \rr d.
$$

\par

Since the map $F\mapsto F(A_\rho ^0\cdo )$
is unitary on $A^2(\cc d)$, and a homeomorphism
on $\maclA _{{}_{{}^\infty}}(\cc d)$
and $\maclA _{{}_{{}^\infty}}'(\cc d)$, respectively, the next
result follows from the unitarity of
\eqref{Eq:BargmannL2Map} and Proposition
\ref{Prop:BasicBargMap}. (See also
\cite{ToBhMa} and the references therein.)

\begin{prop}
Let $\rho \in \rr d$. Then $\mascF _\rho$
is unitary on $L^2(\rr d)$, and
a homeomorphism
on $\mascS (\rr d)$
and $\mascS '(\rr d)$, respectively.
\end{prop}

\medspace

In the symplectic case, the relationship \eqref{Eq:BargmImFT}
translates to
\begin{align}
&(\mathfrak V_{\sigma ,d}\mascF _\sigma \fka )(z,w)
=
\mathfrak V_{\sigma ,d}\fka (iw,-iz)
\notag
\\[1ex]
&=
(\mathfrak V_{\sigma ,d}\fka )
(e^{i\frac \pi 2}(\cos (\textstyle {\frac \pi 2})z
+\sin (\textstyle {\frac \pi 2}) w),
e^{i\frac \pi 2}(-\sin (\textstyle {\frac \pi 2})z
+\cos (\textstyle {\frac \pi 2})w))
\notag
\\[1ex]
&=
(\mathfrak V_{\sigma ,d}\fka )(\alpha _{1,1}z
+\beta _{1,1}w, -\beta _{1,1}z +\alpha _{1,1}w),
\label{Eq:BargmImSymplFT}
\end{align}
where
\begin{alignat}{3}
\alpha _{\rho ,\theta}
&=
e^{i\rho \frac \pi 2}
\cos (\theta \textstyle {\frac \pi 2})&
\qquad &\text{and} &\qquad
\beta _{\rho ,\theta}
&=
e^{i\rho \frac \pi 2}
\sin (\theta \textstyle {\frac \pi 2}).
\label{Eq:AlphaBetaDef}
\intertext{We observe that}
\alpha _{\rho ,\theta} +i\beta _{\rho ,\theta}
&=
e^{i(\rho +\theta )\frac \pi 2} &
\quad &\text{and} &\quad
\alpha _{\rho ,\theta} -i\beta _{\rho ,\theta}
&=
e^{i(\rho -\theta )\frac \pi 2},
\label{Eq:AlphaBetaEuler}
\end{alignat}
in view of Euler's identity.
Also note that
$\alpha _{0,\theta}= \alpha _\theta$
and
$\beta _{0,\theta}= \beta _\theta$ (cf. 
\eqref{Eq:DefAlphaBeta}).

\medspace

We now let \eqref{Eq:BargmImSymplFT}
inspire us to define the fractional
\emph{symplectic} Fourier 
transform of order $(\rho ,\theta )\in \rr {2}$
implicitly by the formula
\begin{align}
&(\mathfrak V_{\sigma ,d}
\mascF _{\sigma ,(\rho ,\theta )} \fka )(z,w)
\notag
\\[1ex]
&=
(\mathfrak V_{\sigma ,d}\fka )(\alpha _{\rho ,\theta} z
+\beta _{\rho ,\theta} w, -\beta _{\rho ,\theta} z
+\alpha _{\rho ,\theta} w)
\notag
\\[1ex]
&=
(\mathfrak V_{\sigma ,d}\fka )(e^{i\rho \frac \pi 2}
(\cos (\theta \textstyle {\frac \pi 2})z
+\sin (\theta \textstyle {\frac \pi 2}) w),
e^{i\rho \frac \pi 2}
(-\sin (\theta \textstyle {\frac \pi 2})z
+\cos (\theta \textstyle {\frac \pi 2})w)).
\notag
\end{align}
In terms of the rotation matrix
\begin{equation}
\begin{aligned}
A_{\rho ,\theta}
&=
\left (
\begin{matrix}
\phantom -\alpha _{\rho ,\theta} & \dots & 0 &
\beta _{\rho ,\theta}  & \dots & 0
\\
\vdots & \ddots & \vdots &
\vdots & \ddots & \vdots
\\
0 & \dots & \phantom -\alpha _{\rho ,\theta} &
0 & \dots & \beta  _{\rho ,\theta} 
\\
-\beta _{\rho ,\theta} & \dots & 0 &
\alpha _{\rho ,\theta}  & \dots & 0
\\
\vdots & \ddots & \vdots &
\vdots & \ddots & \vdots
\\
0 & \dots & -\beta _{\rho ,\theta} &
0 & \dots & \alpha _{\rho ,\theta}
\end{matrix}
\right ),
\\[1ex]
\rho  &\in \mathbf R,
\qquad
\theta \in \mathbf R,
\end{aligned}
\end{equation}
and identifying $(z,w)\in \cc d\times \cc d\simeq \cc {2d}$, 
in usual ways, as column matrices
$$
\left (
\begin{matrix}
z\\
w
\end{matrix}
\right )
=
\left (
\begin{matrix}
z_1\\
\vdots \\
z_d \\
w_1 \\
\vdots \\
w_d
\end{matrix}
\right ),
\qquad z=(z_1,\dots ,z_d)\in \cc d,\ w=(w_1,\dots ,w_d)\in \cc d,
$$
the fractional symplectic Fourier transform can also be expressed as
\begin{align}
\mathfrak V_{\sigma ,d}
(\mascF _{\sigma ,(\rho ,\theta )} \fka )(z,w)
&=
\mathfrak V_{\sigma ,d}\fka (A_{\rho ,\theta} (z,w)).
\label{Eq:FracSymplFT}
%\notag
%\\
%T_\rho (z,w) &\equiv e^{-i\rho \frac \pi 2}z ,\ z\in \cc d,\ \rho \in \mathbf R.
\end{align}

\par

More generally, in terms of the rotation matrices
\begin{equation}
\label{Eq:ArhothetaDef}
\begin{aligned}
A_{\rho ,\theta}
&=
\left (
\begin{matrix}
\phantom -\alpha _{\rho _1,\theta _1} & \dots & 0 &
\beta _{\rho _1,\theta _1}  & \dots & 0
\\
\vdots & \ddots & \vdots &
\vdots & \ddots & \vdots
\\
0 & \dots & \phantom -\alpha _{\rho _d,\theta _d} &
0 & \dots & \beta _{\rho _d,\theta _d}
\\
-\beta _{\rho _1,\theta _1} & \dots & 0 &
\alpha _{\rho _1,\theta _1}  & \dots & 0
\\
\vdots & \ddots & \vdots &
\vdots & \ddots & \vdots
\\
0 & \dots & -\beta _{\rho _d,\theta _d} &
0 & \dots & \alpha _{\rho _d,\theta _d} 
\end{matrix}
\right ),
\\[1ex]
\rho &=(\rho _1,\dots ,\rho _d)\in \rr d ,
\qquad
\theta =(\theta _1,\dots ,\theta _d)\in \rr d ,
%\\[1ex]
%\alpha _{\rho _j,\theta _j}
%&=
%e^{i\rho _j\frac \pi 2}
%\cos (\theta _j\textstyle {\frac \pi 2}),
%\quad
%\text{and} \quad
%\beta _{\rho _j,\theta _j}
%=
%e^{i\rho _j\frac \pi 2}
%\sin (\theta _j\textstyle {\frac \pi 2}),
%\qquad
%j=1,\dots ,d.
\end{aligned}
\end{equation}
the \emph{partial}
fractional symplectic Fourier transform is defined
in the following.

\par

\begin{defn}
\label{Def:GeneralPartFracSympFT}
Let $\rho ,\theta \in \rr d$,
and let $T_{\rho ,\theta}$ be the map on
$A(\cc {2d})$, given by
\begin{equation}
\label{eq:DefTrhotheta}
(T_{\rho ,\theta}F)(z,w)=F(A_{\rho ,\theta }(z,w)).
\end{equation}
%%%
The partial fractional symplectic Fourier transform
on $\mascS '(\rr {2d})$ of
order $(\rho ,\theta )\in \rr {2d}$
is the operator
\begin{equation}
\label{Eq:PartFracSymplFTFormula}
\mascF _{\sigma ,(\rho ,\theta )}
\equiv
\mathfrak V_{\sigma ,d}^{-1}
\circ
T_{\rho ,\theta}
\circ
\mathfrak V_{\sigma ,d}.
\end{equation}
\end{defn}

\par

The next result shows that
$\mascF _{\sigma ,(\rho ,\theta )}$
in
\eqref{Eq:PartFracSymplFTFormula} is well-defined
on $\mascS '(\rr {2d})$.
We also observe that \eqref{Eq:FracSymplFT}
holds for $\rho ,\theta \in \rr d$
and $\fka \in \mascS '(\rr {2d})$.

\par

\begin{prop}
\label{Prop:MapPartFracSymplFTL2Schw}
Let $\rho ,\theta \in \rr d$. Then
$\mascF _{\sigma ,(\rho ,\theta )}$
is unitary on $L^2(\rr {2d})$, and
a homeomorphism on $\mascS (\rr {2d})$
and on $\mascS '(\rr {2d})$.
\end{prop}

\par

Proposition
\ref{Prop:MapPartFracSymplFTL2Schw}
is a special case of
Theorem \ref{Thm:MapPartFracSymplFTGS}
in Section \ref{sec4}
and
Theorem 
\ref{Thm:FracSymplFourGenPil}
in Section \ref{sec5}.
For the readers' convenience, we here
give a proof that $\mascF _\rho$ is
unitary on $L^2(\rr {2d})$.

\par

\begin{proof}[Proof of Proposition
\ref{Prop:MapPartFracSymplFTL2Schw} for
$L^2(\rr {2d})$]
% We only prove that
% $\mascF _{\sigma ,(\rho ,\theta )}$
% is unitary on $L^2(\rr {2d})$. The other
% cases follow from the definitions and
% similar arguments. (See also Section 
% \ref{sec5}.)
Since \eqref{Eq:BargmannL2MapSympl}
is unitary, it suffices to prove that
$T_{\rho ,\theta}$ is unitary on
$A^2_\sigma (\cc {2d})$.

\par

Let
$$
A_j=e^{-i\rho _j\frac \pi 2}
\left (
\begin{matrix}
\alpha _{\theta _j} & \beta _{\theta _j}
\\
-\beta _{\theta _j} & \alpha _{\theta _j}
\end{matrix}
\right ),
\qquad
j=1,\dots ,d,
$$
which are $2\times 2$ rotation matrices.
Then it follows by straight-forward
computations that
\begin{align}
d\mu _{\sigma ,d}(z,w)
&=
d\mu _{\sigma ,1}(z_1,w_1)
\cdots
d\mu _{\sigma ,1}(z_d,w_d),
\notag
\intertext{and}
d\mu _{\sigma ,1}(A_j^{-1}(z_j,w_j))
&=
d\mu _{\sigma ,1}(z_j,w_j),
\qquad
j=1,\dots ,d,
\notag
\intertext{which give}
d\mu _{\sigma ,d}(A_{(\rho ,\theta)}^{-1}(z,w))
&=
d\mu _{\sigma ,d}(z,w).
\label{eq:GaussianMeasureRotationInvariant}
\end{align}
This, in turn, gives
\begin{align*}
\nm {T_{\rho ,\theta}F}{A^2_\sigma}^2
&=
\int _{\cc {2d}}
|F(A_{(\rho ,\theta )}(z,w))|^2
\, d\mu _{\sigma ,d}(z,w)
\\[1ex]
&=
\int _{\cc {2d}}
|F(z,w)|^2
\,
d\mu _{\sigma ,d}(A_{(\rho ,\theta )}^{-1}(z,w))
\\[1ex]
&=
\int _{\cc {2d}}
|F(z,w)|^2
\,
d\mu _{\sigma ,d}(z,w)
=
\nm F{A^2_\sigma}^2,
\qquad
F\in A^2_\sigma (\cc {2d}). \qedhere
\end{align*}
%%
%and the result follows.
\end{proof}

\par

We put
$$
\mascF _{\sigma ,\rho}
\equiv
\mascF _{\sigma ,(\rho ,\rho )}
$$
when $\rho \in \mathbf R$ or $\rho \in \rr d$. Since
$$
\mascF _{\sigma} = \mascF _{\sigma ,1} = \mascF _{\sigma ,(1,1)},
$$
we call $\mascF _{\sigma ,\rho}$ the \emph{canonical} fractional
symplectic Fourier transform (of order $\rho$).

\par

\begin{defn}
\label{Def:KernPartSymplFourT}
Let $\rho _0,\theta _0\in \mathbf R$
and $\rho ,\theta \in \rr d$. Then
the kernels of $\mascF _{\sigma ,(\rho _0,\theta _0)}$
and $\mascF _{\sigma ,(\rho ,\theta )}$ are denoted
by
$K_{\sigma ,(\rho _0,\theta _0)}$ and
$K_{\sigma ,(\rho ,\theta )}$, respectively.
\end{defn}

\par

\subsection{The kernel of
symplectic fractional Fourier transform}

\par

We show how the kernels in Definition \ref{Def:KernPartSymplFourT} can be described
in terms of kernels of fractional Fourier transforms.
In the following, $A^*$ denotes the adjoint of the matrix $A$.

\par

\begin{thm}
\label{Thm:FracSymplFourTransf}
Let $\rho ,\theta \in \rr d$, and $K_\rho$
be given by \eqref{Eq:FracFourKernel}. Then
\begin{multline}
\label{Eq:FracSymplFourTransf}
K_{\sigma ,(\rho ,\theta )}(x,\xi ,y,\eta )
\\[1ex]
=
2^d\big (
K_{-\rho}(\sqrt 2\, x,\cdo )
\otimes
K_{-\rho}(\sqrt 2\, \xi ,\cdo )
\big )
(\sqrt 2\, A_{0,\theta}^*(y,\eta )).
\end{multline}
\end{thm}

\par

In terms of the matrices
\begin{equation}
\label{Eq:CosSineMatrices}
C_\theta
=
\left (
\begin{matrix}
\alpha _{\theta _1} & \dots & 0
\\
\vdots & \ddots & \vdots
\\
0 & \dots & \alpha _{\theta _d}
\end{matrix}
\right )
\quad \text{and}\quad
D_\theta
=
\left (
\begin{matrix}
\beta _{\theta _1} & \dots & 0
\\
\vdots & \ddots & \vdots
\\
0 & \dots & \beta _{\theta _d}
\end{matrix}
\right ),
\end{equation}
it follows that \eqref{Eq:FracSymplFourTransf}
is the same as
\begin{align}
&K_{\sigma ,(\rho ,\theta )}(x,\xi ,y,\eta )
\notag
\\[1ex]
&=
2^d K_{-\rho}(\sqrt 2\, x,\sqrt 2\,
(C_\theta y-D_\theta \eta ))
K_{-\rho}(\sqrt 2\, \xi ,\sqrt 2\, 
(D_\theta y+C_\theta \eta )),
\notag
\\[1ex]
&=
2^d \prod _{j=1}^d
K_{-\rho _j}(\sqrt 2 ( x_j,\alpha _{\theta _j}y_j-\beta _{\theta _j} \eta _j))
K_{-\rho _j}(\sqrt 2( \xi _j, 
\beta _{\theta _j} y_j+\alpha _{\theta _j} \eta _j)),
\notag
\\[1ex]
&=
\prod _{j=1}^d
K_{\sigma ,(\rho _j,\theta _j)}(x_j,\xi _j,y_j,\eta _j),
\label{Eq:FracSymplFourTransfC}
\intertext{where}
&K_{\sigma ,(\varrho ,\vartheta )}(x_0,\xi _0,y_0,\eta _0)
\notag
\\[1ex]
&=
2 K_{-\varrho }(\sqrt 2 ( x_0,
\alpha _{\vartheta }\, y_0-\beta _{\vartheta}\,  \eta _0))
K_{-\varrho}(\sqrt 2( \xi _0, 
\beta _{\vartheta} \, y_0+\alpha _{\vartheta}\,  \eta _0)),
\notag
\\[1ex]
& \hspace{6cm}
x_0,y_0,\xi _0,\eta _0,\varrho ,\vartheta \in \mathbf R .
\label{Eq:FracSymplFourTransfC1Dim}
\end{align}

\par

We need the following lemma
for the proof of Theorem
\ref{Thm:FracSymplFourTransf}.

\par

\begin{lemma}
\label{Lemma:FracFourKernBargIdent}
Let $\rho ,\theta \in \rr d$.
Then the following are true:
\begin{enumerate}
\item
$K_{\rho}$ in \eqref{Eq:FracFourNoPartial}
is given by
\begin{equation}
\label{Eq:FracFourKernBargIdent}
K_{\rho} (x,y)
=
\pi ^{-\frac d2}\int _{\cc d}
\mathfrak A_d(\overline z,x)
\mathfrak A_d(A_\rho ^0 z,y)\, d\mu _d(z),
\quad x,y\in \rr d \text ;
\end{equation}

\vrum

\item $K_{\sigma ,(\rho ,\sigma )}$
in Definition \ref{Def:KernPartSymplFourT} is given by
\begin{multline}
\label{Eq:FracSymplFourKernBargIdent}
K_{\sigma ,(\rho ,\theta )} (x,\xi ,y,\eta )
\\[1ex]
=
\left (
\frac 2\pi
\right )^{d}
\iint _{\cc {2d}}
\mathfrak A_{\sigma ,d}
(\overline z,\overline w,x,\xi )
\mathfrak A_{\sigma ,d}
( A_{\rho ,\theta}(z,w),y,\eta )
\, d\mu _{\sigma ,d}(z,w),
\\[1ex]
x,y,\xi ,\eta \in \rr d \text .
\end{multline}
\end{enumerate}
\end{lemma}

\par

\begin{proof}
The proof of (1) is only performed for
$\rho \in (\mathbf R\setminus 2\mathbf Z)^d$.
The other cases follows by straight-forward 
modifications and are left to the reader.

\par

The integrands in
\eqref{Eq:FracFourKernBargIdent}
and
\eqref{Eq:FracSymplFourKernBargIdent}
belong to $L^1(d\mu _d)$ and
$L^1(d\mu _{\sigma ,d})$, respectively.
Hence, if $f,g\in \mascS (\rr d)$, then
\begin{align*}
(\mascF _{\rho}f,g)_{L^2}
&=
(\mathfrak V_d(\mascF _{\rho}f),
\mathfrak V_d g)_{A^2}
\\[1ex]
&=
\int \limits_{\cc d}
(\mathfrak V_df)(A_\rho ^0z)
\overline {(\mathfrak V_d g)(z)}\, d\mu _d(z)
\\[1ex]
&=
\pi ^{-\frac d2}
\int _{\cc d}
\left (
\int _{\rr d}
\mathfrak A_d(A_\rho ^0z,y)f(y)\, dy
\right )
\left (
\int _{\rr d}
\mathfrak A_d(\overline z,x)\overline {g(x)}\, dx\right )
d\mu _d(z)
\\[1ex]
&=
\iint _{\rr {2d}}K(x,y)f(y)\overline{g(x)}\, dxdy,
\intertext{where}
K(x,y)
&=
\pi ^{-\frac d2}
\int _{\cc d}
\mathfrak A_d(A_\rho ^0z,y)
\mathfrak A_d(\overline z,x)
d\mu _d(z).
\end{align*}
Evidently, $K$ is the kernel $K_\rho$ of
$\mascF _{\rho}$, giving \eqref{Eq:FracFourKernBargIdent}.
Therefore, (1) follows. In similar ways, one
deduces (2). The details are left to the
reader.
\end{proof}

\par

\begin{proof}[Proof of Theorem 
\ref{Thm:FracSymplFourTransf}]
The Bargmann transforms are
given by tensor products of one-dimensional
Bargmann transforms, and the matrix
\eqref{Eq:ArhothetaDef}
is a tensor product of $d$ $2\times 2$ matrices,
where each one acts on $2$ dimensional
phase spaces. This implies that
$K_{\sigma ,(\rho ,\theta)}$ is a
product of the same type of coordinate
kernels, given by \eqref{Eq:FracSymplFourTransfC}.
The result therefore follows if we verify
\eqref{Eq:FracSymplFourTransfC1Dim}, which is
the same as verifying the result for $d=1$.

\par

Therefore assume that $d=1$, and let
$$
x,y,\xi ,\eta ,\rho ,\theta \in \mathbf R,
\quad
\alpha =\alpha _{\rho,\theta}
\quad \text{and}\quad 
\beta =\beta _{\rho,\theta},
$$
be as in \eqref{Eq:AlphaBetaEuler}.
By Lemma \ref{Lemma:FracFourKernBargIdent}
(2) one has
% Since $K_{\sigma ,(\rho ,\theta )}$
% are given by tensor products of 
%%
\begin{align*}
&K_{\sigma ,(\rho ,\theta )}(x,\xi ,y,\eta )
\\[1ex]
&=
% \left ( \frac 2\pi \right )^{2}
\frac 2\pi
\iint _{\cc {2}} 
\mathfrak A_{\sigma ,1}
(\alpha z+\beta w,-\beta z+\alpha w,y,\eta )
\mathfrak A_{\sigma ,1}(\overline z,\overline w,x,\xi )
\, d\mu _{\sigma ,1}(z,w).
\end{align*}
If $\gamma =e^{i\rho \frac \pi 2}$,
then the fact that
$\alpha ^2+\beta ^2=\gamma ^2$
gives
\begin{align}
&\mathfrak A_{\sigma ,1}
(\alpha z+\beta w,-\beta z+\alpha w,y,\eta )
\notag
\\[1ex]
&=
\exp \left (
-\big (
(\alpha z+\beta w)^2 +(-\beta z+\alpha w)^2 
+y^2+\eta ^2
\big )
+
\sqrt 8 \big ( (\alpha z+\beta w)y
+
(-\beta z+\alpha w)\eta \big )
\right )
\notag
\\[1ex]
&=
\exp \left (
-\gamma ^2\big (
z^2 +w^2 
\big )
-
(y^2+\eta ^2)
+
\sqrt 8 \gamma
\big ( (\alpha _\theta z+\beta _\theta w)y
+
(-\beta _\theta z+\alpha _\theta w)\eta \big )
\right )
\notag
\\[1ex]
&=
\exp \left (
-\gamma ^2\big (
z^2 +w^2 
\big )
-
(y^2+\eta ^2)
+
\sqrt 8  \big ( (\gamma z)
(\alpha _\theta y-\beta _\theta \eta )
+
(\gamma w)(\beta _\theta y+\alpha _\theta \eta )\big )
\right ).
\label{Eq:SympFracFourKernEval}
\end{align}
Since
$$
y^2+\eta ^2=(\alpha _\theta y-\beta _\theta \eta )^2
+
(\beta _\theta y+\alpha _\theta \eta )^2,
$$
it follows that the right-hand side 
of \eqref{Eq:SympFracFourKernEval}
% in the
% evaluation of $\mathfrak A_{\sigma ,1}
% (\alpha z+\beta w,-\beta z+\alpha w,x,\xi )$
is equal to
$$
\exp \left (
-\big (
(\gamma z)^2 +(\gamma w)^2 
\big )
-
(y_1^2+\eta _1^2)
+
\sqrt 8 \big ( (\gamma z)y_1
+
(\gamma w)\eta _1\big )
\right ),
$$
where
$$
y_1=\alpha _\theta y-\beta _\theta \eta
\quad \text{and}\quad
\eta _1
=
\beta _\theta y+\alpha _\theta \eta .
$$
By combining all these identities, it follows
that
\begin{equation}
\mathfrak A_{\sigma ,1}
(\alpha z+\beta w,-\beta z+\alpha w,y,\eta )
=
\mathfrak A_{\sigma ,1}
(\gamma z,\gamma w,y_1
,\eta _1 ).
\end{equation}

\par

If $c=\sqrt 2$, then a combination of
these identities with \eqref{Eq:BargSymplBargKernels}
gives
\begin{align*}
&K_{\sigma ,(\rho ,\theta )}(x,\xi ,y,\eta )
\\[1ex]
&=
% \left ( \frac 2\pi \right )^{\! 2}
\frac 2\pi
\iint _{\cc {2}} 
\mathfrak A_1(c\gamma z,cy_1)
\mathfrak A_1(c\gamma w,c\eta _1)
\mathfrak A_1(c\overline z,cx)
\mathfrak A_1(c\overline w,c\xi )
d\mu _{\sigma ,1}(z,w)
\\[1ex]
&=\left ( \frac 2\pi \right )^{\! 3}
H_1(x,y,\eta )H_2(\xi ,y,\eta ),
\intertext{where}
&H_1(x,y,\eta )
=
\int _{\mathbf C} 
\mathfrak A_1(c\gamma z,cy_1)
\mathfrak A_1(c\overline z,cx)
e^{-2|z|^2}\, d\lambda (z)
\intertext{and}
&H_2(\xi ,y,\eta )
=
\int _{\mathbf C} 
\mathfrak A_1(c\gamma w,c\eta _1)
\mathfrak A_1(c\overline w,c\xi )
e^{-2|w|^2}\, d\lambda (w).
\end{align*}

\par

By taking $cz=\sqrt 2z$ and $cw=\sqrt 2w$ as new variables of
integrations for $H_1$ and $H_2$, respectively, and using
\eqref{Eq:FracFourKernBargIdent}, we obtain
\begin{align*}
H_1(x,y,\eta )
&=
2^{-1}\int _{\mathbf C} 
\mathfrak A_1(\gamma z,cy_1)
\mathfrak A_1(\overline z,cx)
e^{-|z|^2}\, d\lambda (z)
\\[1ex]
&=
\frac {\pi ^{\frac 32}} 4
K_{-\rho }
(\sqrt 2\, x,\sqrt 2(\alpha _\theta y-\beta _\theta \eta )),
\intertext{and}
H_2(\xi ,y,\eta )
&=
2^{-1}\int _{\mathbf C} 
\mathfrak A_1(\gamma w,c\eta _1)
\mathfrak A_1(\overline w,c\xi )
e^{-|w|^2}\, d\lambda (w)
\\[1ex]
&=
\frac {\pi ^{\frac 32}} 4
K_{-\rho }
(\sqrt 2\, \xi ,
\sqrt 2(\beta _\theta y+\alpha _\theta \eta )).
\end{align*}
A combination of these identities gives
\eqref{Eq:FracSymplFourTransf}, and the
result follows.
\end{proof}

\par

If $d=1$, then \eqref{Eq:FracFourKernelOneDim}
and \eqref{Eq:FracSymplFourTransf} show that
\begin{multline*}
K_{\sigma ,(\varrho ,\vartheta )}(x,\xi ,y,\eta)
\\[1ex]
=
\begin{cases}
\left (
\frac {\beta _\varrho +i\alpha _\varrho}
{\pi \beta _\varrho}
\right )
\exp \left (
-i\cdot \frac {\alpha _\varrho
(x^2+\xi ^2+y^2+\eta ^2)
-
2\alpha _\vartheta (xy+\xi \eta )
-
2\beta _\vartheta (y\xi -x\eta )}
{\beta _\varrho}
\right ), & 
\varrho \in \mathbf R\setminus 2\mathbf Z
\\[1ex]
2\delta (\sqrt 2(x-\alpha _\vartheta y
+\beta _\vartheta \eta ))
\delta (\sqrt 2(\xi -\beta _\vartheta y
-\alpha  _\vartheta \eta )), &
\varrho \in 4\mathbf Z
\\[1ex]
2\delta (\sqrt 2(x+\alpha _\vartheta y
-\beta _\vartheta \eta ))
\delta (\sqrt 2(\xi +\beta _\vartheta y
+\alpha  _\vartheta \eta )), &
\varrho \in 2+4\mathbf Z,
\end{cases}
\end{multline*}
which in a more condensed form can be
written as
\begin{multline*}
%\begin{equation*}
K_{\sigma ,(\varrho ,\vartheta )}(X,Y)
\\[1ex]
=
\begin{cases}
\left (
\frac {\beta _\varrho +i\alpha _\varrho}
{\pi \beta _\varrho}
\right )
\exp \left (
-i\cdot \frac {\alpha _\varrho |(X,Y)|^2
-
2\alpha _\vartheta \scal XY-2\beta _\vartheta \sigma (X,Y)}
{\beta _\varrho}
\right ), & 
\varrho \in \mathbf R\setminus 2\mathbf Z
\\[1ex]
\delta (X-A^*_{0,\vartheta}(Y)),
& \varrho \in 4\mathbf Z
\\[1ex]
\delta (X+A^*_{0,\vartheta}(Y)),
& \varrho \in 2+4\mathbf Z,
\end{cases}
\end{multline*}
%\end{equation*}
%%
when $A_{0,\vartheta}$ is the linear map
on $\rr {2}$, given by \eqref{Eq:ArhothetaDef} for $d=1$.
%As usual, we have here identified linear mappings
%with their matrices. Here $X=(x,\xi )\in \rr 2$ and
%$Y=(y,\eta )\in \rr 2$ as usual.

\par

For general $d$, it follows that
the kernel
$K_{\sigma ,(\rho ,\theta )}$ is given by
\begin{equation}
\begin{aligned}
K_{\sigma ,(\rho ,\theta )}(X,Y)
&=
\prod _{j=1}^d
K_{\sigma ,(\rho _j,\theta _j)}(X_j,Y_j),
\\[1ex]
X_j &= (x_j,\xi _j)\in \rr 2,\quad
Y_j = (y_j,\eta _j)\in \rr 2,
\\[1ex] 
X &=(x,\xi )\in \rr {2d},\quad
Y=(y,\eta )\in \rr {2d},
\\[1ex]
x &= (x_1,\dots ,x_d)\in \rr d,
\quad
\xi =(\xi _1,\dots ,\xi _d)\in \rr d,
\\[1ex]
\rho &=(\rho _1,\dots ,\rho _d)\in \rr d,
\ \text{and}\  
\theta = (\theta _1,\dots ,\theta _d)
\in \rr d.
\end{aligned}
\end{equation}

\par

\begin{rem}
\label{Rem:FrSyFTLacksSympl}
We recall that a $(2d)\times (2d)$ (real) matrix
$M$ is called symplectic (invariant)
if and only if
$$
\sigma (MX,MY) = \sigma (X,Y),
\qquad
X,Y\in \rr {2d}.
$$
Furthermore, suppose that
$T\in \maclL (\mascS (\rr {2d}),\mascS '(\rr {2d}))$ has kernel $K$. Then $T$
is called symplectic (invariant) if
$$
T(F(M\cdo )) = (TF)(M\cdo ),
$$
for every symplectic $(2d)\times (2d)$
matrix $M$. We observe that $T$ is symplectic
if and only if $K$ satisfies
$$
K(MX,MY)=K(X,Y),
$$
for every symplectic matrix $M$.

\par

The symplectic Fourier transform is symplectic 
invariant, because its kernel satisfies
$$
K_{\sigma ,(1,1)}(MX,MY)
=
\pi ^{-d}e^{2i\sigma (MX,MY)}
=
\pi ^{-d}e^{2i\sigma (X,Y)}
=
K_{\sigma ,(1,1)}(X,Y),
$$
for every symplectic matrix $M$. Similar
arguments show that any fractional
symplectic Fourier transform
$\mascF _{\sigma ,(\rho ,\theta )}$ are
symplectic invariant when
\begin{equation}
\label{Eq:ParameterIntegerChoices}
\rho ,\theta \in 2\mathbf Z
\quad \text{or}\quad
\rho ,\theta \in 2\mathbf Z +1.
\end{equation}
Furthermore, straight-forward evaluations
for other cases, show that these are the
only possibilities that fractional 
symplectic Fourier transforms are 
symplectic invariant.
Summing up, we have:
\\[2ex]
\emph{If $\rho ,\theta \in \mathbf R$, then
$\mascF _{\sigma ,(\rho ,\theta )}$ is
symplectic invariant, if and only if
\eqref{Eq:ParameterIntegerChoices} holds.}
\end{rem}

\par

%%%%%%%%%%%%%%%%%%%%%%%%%
\section{Mapping properties of fractional symplectic Fourier transforms on Wigner 
distributions and Schatten-von Neumann symbol classes}
\label{sec3}
%%%%%%%%%%%%%%%%%%%%%%%%%

\par

In this section, we establish mapping properties 
of the partial fractional symplectic Fourier 
transforms on Wigner distributions. 
In particular, we show that any
Wigner distribution $W_{f,g}$ is transformed into
a Wigner distribution of suitable fractional
Fourier transforms of $f$ and $g$, see 
\eqref{Eq:FracSympAct0} and 
\eqref{Eq:FracSympAct}. 
Some of these investigations
are related to analysis presented
in \cite{AliBas,Alm}, while other parts
are related to \cite[Section 5]{Ful}.
As a consequence, we establish some group properties of these operators.
Finally, we apply our results to prove that 
the Schatten-von Neumann symbol classes from Subsection
\ref{subsec1.4} are invariant under partial fractional symplectic Fourier transforms.

\par

\subsection{Mapping properties of fractional symplectic Fourier transforms on Wigner distributions}

\par

As a consequence of the following, one has that
\emph{canonical} fractional symplectic Fourier
transforms fulfill
\begin{equation}
\label{Eq:FracSympAct0}
\mascF _{\sigma ,\rho}W_{f,g}
=
W_{\mascF _{-2\rho}f,g},
\end{equation}
when $\rho \in \mathbf R$ or $\rho \in \rr d$ and
$f,g\in \mascS '(\rr d)$.

\par

\begin{thm}
\label{Thm:Main}
Let $\rho ,\theta \in \rr d$.
Then
\begin{align}
\mascF _{\sigma ,(\rho ,\theta )} \circ \mascW 
&=
\mascW \circ \mascF _{(-\rho -\theta ,
-\rho +\theta )}
\label{Eq:FracSympActAdv}
\end{align}
as operators on
$\mascS '(\rr {2d})$.
\end{thm}

\par

It follows from \cite{ToBhMa} that
\eqref{Eq:FracSympActAdv} is the same as
\begin{align}
\mascF _{\sigma ,(\rho ,\theta )} \circ \mascW 
&=
\mascW \circ (\mascF _{-\rho -\theta }
\otimes
\mascF _{-\rho +\theta }).
\tag*{(\ref{Eq:FracSympActAdv})$'$}
\label{Eq:FracSympActAdvTensor}
\end{align}
By \ref{Eq:FracSympActAdvTensor}
we would have
\begin{alignat}{2}
\mascF _{\sigma ,(\rho ,\theta )}W_{f,g}
&=
W_{\mascF _{-\rho -\theta }f,\mascF _{\rho -\theta }g},&
\qquad
f,g
&\in
\mascS '(\rr d),
\label{Eq:FracSympAct}
\intertext{which in particular implies
that}
\mascF _{\sigma}W_{f,g}
&=
W_{\mascF _{-2}f,\mascF _{0}g},&
\qquad
f,g
&\in
\mascS '(\rr d),
\notag
\intertext{which is the same as the well-known
relation}
\mascF _{\sigma}W_{f,g}
&=
W_{\check f,g},& &
\notag
\end{alignat}
because of \eqref{Eq:FTandReflexion} and since $\mascF _0$ is the identity map (see 
\cite{Fol,Gro1,Toft2}).

\par

The proof of Theorem \ref{Thm:Main} is based on the following lemma.

\par

\begin{lemma}
\label{Lemma:BargmannWigner}
Let $f,g\in \mascS '(\rr d)$
and $F\in \mascS '(\rr {2d})$.
Then
\begin{align}
(\mathfrak V _{\sigma ,d}W_{f,g})(z,w)
&=
2^{-\frac d2}
(\mathfrak V _df)(z-iw)\cdot (\mathfrak V _d\overline g)(z+iw)
\label{Eq:BargmannWigner}
\intertext{and}
(\mathfrak V _{\sigma ,d}(\mascW \! F))(z,w)
&=
2^{-\frac d2}
(\mathfrak V _{2d}F)(z-iw,z+iw)
\label{Eq:BargmannWignerAdv}
\end{align}
\end{lemma}

\par

\begin{proof}
The equality \eqref{Eq:BargmannWigner} is a 
special case of \eqref{Eq:BargmannWignerAdv}.
Furthermore, since $\mascS (\rr {2d})$
is dense in $\mascS '(\rr {2d})$, and since
all involved mappings in 
\eqref{Eq:BargmannWignerAdv} are linear
and continuous, it suffices to prove
\eqref{Eq:BargmannWignerAdv} when
$F\in \mascS (\rr {2d})$.

\par

One has
\begin{align}
&(\mathfrak V _{\sigma ,d}(\mascW \! F))(z,w) \notag \\
&\qquad =
\pi ^{-d}
\iiint _{\rr {3d}}
\mathfrak A_{\sigma ,d}(z,w,y,\eta )F(y+\textstyle {\frac t2},y-\textstyle {\frac t2} )
e^{-i\scal t\eta}\, dtdyd\eta
\notag
\\[1ex]
&\qquad =
\pi ^{-d}
\iint _{\rr {2d}}
\mathfrak A_{\sigma ,d}(z,w,y,0) I_1(w,t,y)F(y+\textstyle {\frac t2},
y-\textstyle {\frac t2} )\, dtdy ,
\label{Eq:IntExpr}
\end{align}
where
$$
I_1(w,t) = 
\int _{\rr d}
e^{-|\eta |^2+\sqrt 8\scal w\eta -i\scal t\eta}\, d\eta .
$$
By straight-forward computations, one obtains
$$
I_1(w,t) = \pi ^{\frac d2}
e^{2\scal ww-i\sqrt 2 \scal wt-\frac 14|t|^2}.
$$
Inserting this into \eqref{Eq:IntExpr} gives
$$
(\mathfrak V _{\sigma ,d}F)(z,w)
=
\pi ^{-\frac d2}
\iint _{\rr {2d}}
I_2(z,w,t,y)F(y+\textstyle {\frac t2},
y-\textstyle {\frac t2} )\, dtdy ,
$$
where
$$
I_2(z,w,t,y)
=
\exp
\left (
\scal ww -\scal zz +\sqrt 8\scal zy -i\sqrt 2\scal wt -(|y|^2+\frac 14|t|^2)
\right ).
$$
Taking $y-\frac t2$ and $y+\frac t2$ as new variables of integrations, we
obtain
$$
(\mathfrak V _{\sigma ,d}F)(z,w)
=
\pi ^{-\frac d2}
\iint _{\rr {2d}}
I_3(z,w,t,y)F(y,t)\, dtdy ,
$$
where
\begin{multline*}
I_3(z,w,t,y)
\\[1ex]
=
\exp
\left (
\scal ww -\scal zz +\sqrt 2\scal z{y+t} -i\sqrt 2\scal w{y-t} -\frac 12(|y|^2+|t|^2)
\right ).
\end{multline*}
If we let $z_1=z+iw$ and $z_2=z-iw$, this can be reformulated as
\begin{align*}
&I_3(z,w,t,y)
\\[1ex]
&=
\exp
\left (
-\frac 12(\scal {z_1}{z_1}+\scal {z_2}{z_2} +|t|^2+|y|^2) +\sqrt 2(\scal {z_1}t +\scal {z_2}y)
\right )
\\[1ex]
&=
\mathfrak A_{2d}(z_2,z_1,y,t)
%=
%\mathfrak A_d(z+iw,t)\mathfrak A_d(z-iw,y)
=
\mathfrak A_{2d}(z-iw,z+iw,y,t).
\end{align*}
A combination of these identities gives
\begin{align*}
(\mathfrak V _{\sigma ,d}(\mascW \! F))(z,w)
&=
\pi ^{-\frac d2}
\iint _{\rr {2d}}
\mathfrak A_{2d}(z-iw,z+iw,y,t)
F(y,t)\, dtdy
\\[1ex]
&=
2^{-\frac d2}
\mathfrak V_{2d}F(z-iw,z+iw),
\end{align*}
which gives the result.
\end{proof}

\par

\begin{proof}[Proof of Theorem \ref{Thm:Main}]
We only present a proof for the result when $\rho _1=\cdots =\rho _d=\rho _0\in \mathbf R$,
leaving its modification into the general case to the reader.
Let
$$
\alpha =\alpha _{\rho ,\theta}
\quad \text{and}\quad
\beta =\beta _{\rho ,\theta}
$$
be as above.
%$$
%\tilde \alpha _\rho =\cos (\rho \textstyle{\frac \pi 2})
%\quad \text{and}\quad
%\beta _\rho = \sin (\rho \textstyle{\frac \pi 2}).
%$$
Then it follows from the definitions, Lemma
\ref{Lemma:BargmannWigner}, and \eqref{Eq:AlphaBetaEuler}
that
\begin{align*}
&\mathfrak V_{\sigma ,d}( \mascF _{\sigma ,(\rho ,\theta )}(\mascW \! F))(z,w)
=
(\mathfrak V_{\sigma ,d}(\mascW \! F))(\alpha  z+\beta  w ,
-\beta  z+\alpha  w )
\\[1ex]
&=
2^{-\frac d2}
\mathfrak V_{2d}F( 
\alpha  z  +\beta  w
-i(-\beta  z  +\alpha  w),
\alpha  z  +\beta  w+
i(-\beta  z  +\alpha  w))
\\[1ex]
&=
2^{-\frac d2}
\mathfrak V_{2d}F( 
(\alpha _{\rho ,\theta} +i\beta _{\rho ,\theta} )(z -iw),
(\alpha _{\rho ,\theta} -i\beta _{\rho ,\theta} )(z +iw)) 
\\[1ex]
&=
2^{-\frac d2}
\mathfrak V_{2d}F( 
e^{i(\rho +\theta )\frac \pi 2}(z -iw),
e^{i(\rho -\theta )\frac \pi 2}(z +iw)) 
\\[1ex]
&=
2^{-\frac d2}
\mathfrak V_{2d}(
\mascF _{(-\rho -\theta ,-\rho +\theta )}
F)(z-iw,z+iw)
\\[1ex]
&=
\mathfrak V_{\sigma ,d}
(\mascW \! (
\mascF _{(-\rho -\theta ,
-\rho +\theta )}F))(z,w).
\end{align*}
Hence
$$
\mathfrak V_{\sigma ,d}(\mascF _{\sigma ,(\rho ,\theta )}(\mascW \! F))(z,w)
=
\mathfrak V_{\sigma ,d}(\mascW \!
(\mascF _{(-\rho -\theta ,
-\rho +\theta )}F))(z,w).
$$
The result now follows from the fact that 
the symplectic Bargmann transform
is injective.
\end{proof}

\par

The next result is a straight-forward consequence of Theorem
\ref{Thm:Main}, and the group and
periodic properties
\begin{equation}
\begin{aligned}
\mascF _{\rho _1}\circ \mascF _{\rho _2} &= \mascF _{\rho _1+\rho _2},
\quad \mascF _\rho ^{-1} = \mascF _\rho ^* = \mascF _{-\rho},
\\[1ex]
\mascF _0 &=\operatorname{Id}_{L^2},
\quad \text{and}\quad
\mascF _{\rho +\gamma}=\mascF _{\rho},
\\[1ex]
\rho ,\rho _1,\rho _2 &\in \rr d,\ \gamma \in 4\zz d
\end{aligned}
\end{equation}
for fractional Fourier transforms, when acting on $\mascS '(\rr d)$.
The details are left to the reader.

\par

\begin{prop}
\label{Prop:FracSymplFourGroupProp}
It holds
\begin{equation}
\begin{aligned}
\mascF _{\sigma ,(\rho _1,\theta _1)}\circ \mascF _{\sigma ,(\rho _2,\theta _2)}
&= \mascF _{\sigma ,(\rho _1+\rho _2,\theta _1+\theta _2)},
\\[1ex]
\mascF _{\sigma ,(\rho ,\theta )} ^{-1} &= \mascF _{\sigma ,(\rho ,\theta )} ^*
= \mascF _{\sigma ,(-\rho ,-\theta )},
\\[1ex]
\mascF _{\sigma ,(0,0)} &=\operatorname{Id}_{L^2},
\quad \text{and}\quad
\mascF _{\sigma (\rho +\gamma ,\theta +\delta )}=\mascF _{\sigma ,(\rho ,\theta )},
\\[1ex]
\rho ,\rho _j,\theta ,\theta _j &\in \rr d,\ \gamma ,\delta \in 2\zz d, j=1,2,
\end{aligned}
\end{equation}
as operators on $\mascS '(\rr {2d})$.
\end{prop}

\par

\subsection{Mapping properties of fractional symplectic Fourier transforms on Schatten-von Neumann 
symbol classes}

We have the following result.

\par

\begin{thm}
\label{Thm:FracSymplFourSchattInv}
Let $\rho ,\theta \in \rr d$, $\maclB \hookrightarrow \ell ^\infty (\mathbf N)$ be a topological vector space, and $\Omega _\fka$ be as in
\eqref{Eq:OmegaaDef}.
Then the following are true:
\begin{enumerate}
\item $\mascF _{\sigma ,(\rho ,\delta )}$ is bijective on $s_{0,N}^w(\rr {2d})$
for every $N\in \mathbf N$;

\vrum

\item $\Omega _{\mascF_{\sigma ,(\rho ,\theta)}\fka}=\Omega _{\fka}$;

\vrum

\item the map $\fka \mapsto \mascF_{\sigma ,(\rho ,\theta )}\fka$ on
$\mascS '(\rr {2d})$, restricts to an isomorphism on $s_\maclB ^w(\rr {2d})$;

\vrum

\item if, in addition, $\maclB$ is a quasi-Banach space, then
\begin{equation}
\nm {\mascF_{\sigma ,(\rho ,\theta )}\fka}{s_\maclB ^w}
=
\nm {\fka}{s_\maclB ^w}.
\end{equation}
\end{enumerate}
\end{thm}

\par

\begin{proof}
Let $\fka \in s_{0,N}^w(\rr {2d})$ be given by
\eqref{Eq:FiniteRankSymb}. Then \eqref{Eq:FracSympAct}
shows that
$$
\mascF _{\sigma ,(\rho ,\theta)}\fka
=
\sum _{j=1}^N c_jW_{\mascF _{-\rho -\theta}f_j,\mascF _{\rho -\theta}g_j},
$$
and (1) follows from the fact that fractional Fourier transforms are
isometries on $L^2(\rr d)$.

\par

As the partial fractional Fourier transforms are bijections on the family of orthonormal sets, by \eqref{Eq:FracSympAct}, this gives
\begin{align*}
\Omega _{\mascF_{\sigma ,(\rho ,\theta)}\fka}
&=
\sets {\{ (\mascF_{\sigma ,(\rho ,\theta)}\fka ,W_{f_j,g_j})_{L^2} \} _{j\in \mathbf N}}
{\{ f_j\} _{j\in \mathbf N},\{ g_j\} _{j\in \mathbf N}\in \ON}
\\[1ex]
&=
\sets {\{ (\fka ,
\mascF_{\sigma ,(-\rho ,-\theta)}W_{f_j,g_j})_{L^2} \} _{j\in \mathbf N}}
{\{ f_j\} _{j\in \mathbf N},\{ g_j\} _{j\in \mathbf N}\in \ON}
\\[1ex]
&=
\sets {\{ (\fka ,
W_{\mascF _{\rho +\theta}f_j,
\mascF _{-\rho +\theta}g_j})_{L^2} \} _{j\in \mathbf N}}
{\{ f_j\} _{j\in \mathbf N},\{ g_j\} _{j\in \mathbf N}\in \ON}
\\[1ex]
&=
\sets {\{ (\fka ,
W_{f_j,g_j})_{L^2} \} _{j\in \mathbf N}}
{\{ \mascF _{-\rho -\theta}f_j\} _{j\in \mathbf N},
\{ \mascF _{\rho -\theta}g_j\} _{j\in \mathbf N}\in \ON}
\\[1ex]
&=
\sets {\{ (\fka ,
W_{f_j,g_j})_{L^2} \} _{j\in \mathbf N}}
{\{ f_j\} _{j\in \mathbf N},
\{ g_j\} _{j\in \mathbf N}\in \ON}
\\[1ex]
&= \Omega _\fka ,
\end{align*}
and (2) follows.

\par

The assertions (3) and (4) now follow from (2) and straight-forward
computations. The details are left to the reader.
\end{proof}

\par

%%%%%%%%%%%%%%%%%%%%%%%%%%%%%
\section{Mapping properties of
of fractional symplectic Fourier transforms on Gelfand-Shilov
spaces}
\label{sec4}
%%%%%%%%%%%%%%%%%%%%%%%%%%%%%

\par

In this section, we show that partial fractional symplectic Fourier
transforms are homeomorphisms on classical Fourier invariant
Gelfand-Shilov spaces, and their strong dual spaces, the spaces of Gelfand-Shilov
(ultra-)distributions. 
We start by recalling basic properties
of such spaces.

\par

We first introduce the classical Fourier invariant
Gelfand-Shilov spaces. Let $r,s\in \mathbf R_+$.
The space $\maclS _{s,r} (\rr d)$
consists of all
$f\in C^\infty (\rr d)$
such that
\begin{equation}
\label{Eq:GSNorm}
\nm f{\maclS _{s,r}}
\equiv
\sup _{\alpha ,\beta \in \nn d}
\left (
\frac {\nm {x^\alpha \partial ^\beta f}
{L^\infty}}
{r^{|\alpha +\beta |}
(\alpha !\beta !)^s}
\right )
<\infty. 
\end{equation}
We let the topology
of $\maclS _{s,r} (\rr d)$
be defined by the norm \eqref{Eq:GSNorm},
turning it into a Banach space.

\par

The \emph{Gelfand-Shilov
spaces} $\maclS _s (\rr d)$
and $\Sigma _s (\rr d)$,
of \emph{Roumieu} and \emph{Beurling types},
respectively, are given by
$$
\maclS _s (\rr d)
=
\bigcup _{r>0}
\maclS _{s,r} (\rr d)
\quad \text{and}\quad
\Sigma _s (\rr d)
=
\bigcap _{r>0}
\maclS _{s,r} (\rr d).
$$
We equip $\maclS _s (\rr d)$
and $\Sigma _s (\rr d)$
with the corresponding inductive and projective limit
topologies, respectively. It follows that
$\Sigma _s (\rr d)$
is a Fr{\'e}chet space (under
the norms \eqref{Eq:GSNorm}).
(See e.g. \cite{GeSh,Pil2}.)

\par

There are several convenient characterizations
of $\maclS _s (\rr d)$
and $\Sigma _s (\rr d)$.
For example:
if
$f\in \mascS '(\rr d)$,
then 
$f\in \maclS _s (\rr d)$
($f\in \Sigma _s (\rr d)$)
if and only if
\begin{equation}\label{Eq:GSFtransfChar}
|f(x)|\lesssim e^{-r|x|^{\frac 1s}}
\quad \text{and}\quad
|\widehat f(\xi )|
\lesssim
e^{-r|\xi |^{\frac 1s}},
\quad \text{for some (every)}\quad
r>0,
\end{equation}
cf. \cite{ChuChuKim, Eij}.
Consequently, by the Hardy uncertainty 
principle,
it follows that $\maclS _{s_1}(\rr d)$ and
$\Sigma _{s_2}(\rr d)$ are non-trivial 
(i.{\,}e.,
they contain non-zero elements) if and only if
$s_1\ge \frac 12$ and $s_2>\frac 12$. (See 
\cite{GeSh,Pet1}.)

\par

The \emph{Gelfand-Shilov
distribution spaces} 
$\maclS _{s}'(\rr d)$
and
$\Sigma _{s}'(\rr d)$
(of Roumieu and Beurling types, respectively)
are the strong duals of
$\maclS _{s}(\rr d)$
and
$\Sigma _{s}(\rr d)$.

\par

If $s\ge \frac 12$, then
the $L^2$ scalar product
$(\cdo ,\cdo )_{L^2}$ on
$\maclS _s (\rr d)\times
\maclS _s (\rr d)$ is uniquely
extendable to a hypocontinuous sesqui-linear
form $(\cdo ,\cdo )$ on $\maclS _s'
(\rr d)\times \maclS _s (\rr d)$.
Similar facts hold with $\Sigma _s $
in place of $\maclS _s$
at each occurrence, provided $s>\frac 12$.
We have
\begin{equation}
\label{Eq:GSEmb}
\begin{aligned}
\maclS _s (\rr d)
&\hookrightarrow
\Sigma _t (\rr d)
\hookrightarrow
\maclS _t (\rr d)
\\[1ex]
&\hookrightarrow
\mascS (\rr d)
\hookrightarrow
L^2(\rr d)
\hookrightarrow
\mascS '(\rr d)
\\[1ex]
&\hookrightarrow
\maclS _t'(\rr d)
\hookrightarrow
\Sigma _t'(\rr d)
\hookrightarrow
\maclS _s'(\rr d),
\\[1ex]
&\text{when}\quad
\frac 12 \le s<t.
\end{aligned}
\end{equation}
Furthermore,
the inclusions in \eqref{Eq:GSEmb}
are dense and strict. 
For example, as a consequence of
\eqref{Eq:GSFtransfChar}, we have
\begin{equation}\label{Eq:vrhoEmb}
v_r = e^{r |\cdo |}
\in
\Sigma _1'(\rr d)\setminus
\mascS '(\rr d),
% \quad x\in \rr d,
\quad %\text{for every}\quad
r \in \mathbf R_+.
\end{equation}

\par

The next result explains mapping properties
for partial fractional Fourier transforms on Gelfand-Shilov
spaces. We omit the proof, since the result is
available in more general forms in \cite{ToBhMa}.
We also observe that the result is a special
case of Theorem \ref{Thm:FracFourGenPil}
in Section \ref{sec5}. (See also Remark
\ref{Rem:FracFourGSSpaces} in Section
\ref{sec5}.)

\par

\begin{prop}
\label{Prop:BasicMapsOnGS}
Let $s\ge \frac 12$, $t > \frac12$, and $\rho 
\in \rr d$.
%Then the following are true:
\begin{enumerate}
\item $\mascF _\rho$ restricts
to homeomorphisms on $\maclS _s (\rr d)$,  $\Sigma_t(\rr d)$, and $\mascS (\rr d)$.

\vrum

\item $\mascF _\rho$ is unitary on
$L^2(\rr d)$.

\vrum

\item $\mascF _\rho$ extends
uniquely to homeomorphisms on
$\maclS _s'(\rr d)$, $\Sigma_t'(\rr d)$
, and $\mascS '(\rr d)$.
\end{enumerate}
\end{prop}

\par

The mapping properties of the partial fractional 
symplectic
Fourier transform on the Gelfand-Shilov spaces 
may be described as follows.
Recall \eqref{Eq:WigMapDef} for the definition of
the map $\mascW$ on $\mascS '(\rr {2d})$.
In particular, we get another proof of
Proposition \ref{Prop:MapPartFracSymplFTL2Schw}.

\par

\begin{thm}
\label{Thm:MapPartFracSymplFTGS}
Let $s\ge \frac 12$, $t>\frac 12$, and
$\rho ,\theta \in \rr d$.
\begin{enumerate}
\item The mappings $\mascW$ and
$\mascF _{\sigma ,(\rho ,\theta )}$
restrict to homeomorphisms on
$\maclS _s(\rr {2d})$,
$\Sigma _t(\rr {2d})$, and $\mascS(\rr {2d})$.

\vrum

\item $\mascF _{\sigma ,(\rho ,\theta )}$
is unitary on $L^2(\rr {2d})$.

\vrum

\item The mappings $\mascW$ and
$\mascF _{\sigma ,(\rho ,\theta )}$
uniquely extend to homeomorphisms on
$\maclS _s'(\rr {2d})$,
$\Sigma _t'(\rr {2d})$, and $\mascS '(\rr {2d})$.
\end{enumerate}
\end{thm}

\par

Theorem \ref{Thm:MapPartFracSymplFTGS}
is a special case of Theorem \ref{Thm:FracSymplFourGenPil}
in Section \ref{sec5}. (See also Remark
\ref{Rem:FracFourGSSpaces} in Section
\ref{sec5}.) Here, we give an alternative and shorter proof,
based on Proposition \ref{Prop:BasicMapsOnGS} and
some properties deduced thus far.

\par

\begin{proof}
The assertions for $\mascW$ follow from
the facts that the mappings in
\eqref{Eq:WigMapDefSplit} are homeomorphisms
on
$$
\maclS _s(\rr {2d}),
\quad
\Sigma _t(\rr {2d}),
\quad
\mascS (\rr {2d}),
\quad
\mascS '(\rr {2d}),
\quad
\Sigma _t'(\rr {2d}),
\quad \text{and}\quad
\maclS _s'(\rr {2d}),
$$
and unitary on $L^2(\rr {2d})$.
The assertions for
$\mascF _{\sigma ,(\rho ,\theta )}$
now follow by combining the analogous mapping
properties for $\mascW$ with
Theorem \ref{Thm:Main} and
Proposition \ref{Prop:BasicMapsOnGS}.
\end{proof}

\par

%%%%%%%%%%%%%%%%%%%%%%%%%%%%%
\section{Mapping properties of fractional symplectic
Fourier transforms on a broad family of Pilipovi{\'c} spaces}
\label{sec5}
%%%%%%%%%%%%%%%%%%%%%%%%%%%%%

\par

In this section, we extend the classes of 
Pilipovi{\'c} spaces, given in 
\cite{FeGaTo,NeToVi2,Toft18}, introduce Wigner 
versions of such spaces, and prove that 
partial fractional symplectic Fourier transforms
easily become homeomorphisms on them.
The first part is devoted to a general discussion
on sequence spaces contained in $\ell _0'(\nn d)$. Thereafter,
corresponding Pilipovi{\'c} spaces and Wigner-Pilipovi{\'c} spaces
are defined by imposing  that their Hermite coefficients
and Wigner-Hermite coefficients, respectively, belong to
these sequence spaces. Following this, we deduce that partial fractional Fourier transforms and partial fractional symplectic
Fourier transforms are homeomorphisms on them. We remark that, for some
of the spaces, such transforms of \emph{complex order}
remain homeomorphisms.

\par

\subsection{Smallest and largest spaces, and sequence spaces}
\label{subsec5.1}

\par

The family of Pilipovi{\'c} spaces and Wigner-Pilipovi{\'c}
spaces contain the family
of Fourier invariant Gelfand-Shilov spaces and their distribution
spaces. The definitions of these families rely on formal 
Hermite function expansions and Hermite-Wigner function
expansions, given by
\begin{equation}
\label{Eq:FormalExp}
f
=
\sum _{\alpha \in \nn d}c(f;\alpha )h_\alpha
\quad \text{and}\quad
\fka
=
\sum _{\alpha ,\beta
\in \nn d}c_\sigma (\fka ; \alpha ,\beta )
h_{\scW, (\alpha ,\beta )},
\end{equation}
respectively. Here
$$
h_\alpha(x)
=
(-1)^\alpha \pi^{-\frac{d}{4}}
(2^{|\alpha|} \alpha!)^{-\frac12}
e^{\frac12 |x|^2} \partial^\alpha( e^{-|x|^2})
, \qquad \alpha \in \nn d , 
$$
is the (classical) \emph{Hermite function}
of order $\alpha \in \nn d$, and
$$
h_{\scW ,(\alpha, \beta)}
\equiv
W_{h_\alpha, h_\beta}
$$
is the \emph{Hermite-Wigner function} of order
$(\alpha ,\beta )\in \nn {2d}$.
In view of \eqref{Eq:WigMapDef}
one has
\begin{equation}
\label{Eq:HermAndSympHermRel}
h_{\scW ,(\alpha, \beta)} 
=
\mascW (h_{(\alpha, \beta)})
\end{equation}

\par

We recall that the set $\{ h_\alpha \} _{\alpha \in \nn d}$ 
forms an orthonormal basis of $L^2(\rr d)$,
and that
$$
h_{(\alpha, \beta)} =
h_\alpha \otimes h_\beta ,
\qquad \alpha, \beta \in \nn d . 
$$
Since $\mascW$ is unitary on $L^2(\rr {2d})$, it follows that
$\{ h_{\scW, (\alpha ,\beta)} \} _{\alpha ,\beta \in \nn d}$ 
also forms an orthonormal basis of $L^2(\rr {2d})$.

\par

We let $\maclH _0'(\rr d)$ and $\maclWH _{0}'(\rr {2d})$ be
the sets of all \emph{formal} expansions $f$ and $\fka$, respectively, in
\eqref{Eq:FormalExp}. Also, let $\maclH _0(\rr d)$ and $\maclWH _{0}(\rr 
{2d})$ be the sets of all \emph{finite} expansions $f$ and $\fka$, respectively, in
\eqref{Eq:FormalExp}. By letting
% $$
% \mathfrak {A\ B\ C\ D\ E\ F\ G\ H\ I\ J\ K\ L\ 
% M\ N\ O\ P\ Q\ R\ S\ T\ U\ V\ X\ Y\ Z}
% $$
%
%
%%
\begin{alignat}{2}
\mathfrak H (c)
&\equiv
\sum _{\alpha \in \nn d}c(\alpha )h_\alpha ,&
\quad
c &= \{ c(\alpha )\} _{\alpha \in \nn d}\in \ell _0'(\nn d),
\intertext{and}
\mathfrak H _\sigma (c_\sigma )
&\equiv
\sum _{\alpha ,\beta
\in \nn d}c_\sigma (\alpha ,\beta )
h_{\scW,(\alpha ,\beta )},&
\quad
c_\sigma &= \{ c(\alpha ,\beta )\} _{\alpha ,\beta \in \nn d}\in \ell _0'(\nn {2d}),
\end{alignat}
%%
%$$
%\mathfrak {A\ B\ C\ D\ E\ F\ G\ H\ I\ J\ K\ L\ M\ N\ O\ P\ Q\ R\ S\ T\ U\ V\ X\ Y\ Z}
%$$
it follows that the mappings
\begin{alignat}{4}
\mathfrak H
&: &
\ell _0(\nn d) &\to \maclH _0(\rr d),&
\qquad
\mathfrak H
&: &
\ell _0'(\nn d) &\to \maclH _0'(\rr d),
\label{Eq:SmallestHomeomorph1}
% \\[1ex]
% \mathfrak H
% &: &
% \ell _0'(\nn d) &\to \maclH _0'(\rr d),
% \label{Eq:LargestHomeomorph1}
\intertext{and}
\mathfrak H_\sigma
&: &
\, \ell _0(\nn {2d})
&\to \maclWH _{0}(\rr {2d}), &
%\label{Eq:SmallestHomeomorph2}
%\intertext{and}
\qquad
\mathfrak H_\sigma
&: &
\, \ell _0'(\nn {2d}) &\to \maclWH _{0}'(\rr {2d}),
\label{Eq:LargestHomeomorph2}
\end{alignat}
are bijections. Hence by letting the topologies 
and algebraic
structures at the spaces on the right in
\eqref{Eq:SmallestHomeomorph1}--\eqref{Eq:LargestHomeomorph2}
be inherited from the corresponding spaces on the left, it follows that
the mappings in 
\eqref{Eq:SmallestHomeomorph1}--\eqref{Eq:LargestHomeomorph2}
are homeomorphisms.

\par

In order to specify the links to $f$ and $\fka$
in these expansions, we put
$c(f;\alpha )=c(\alpha )$ and
$c_\sigma (\fka ;\alpha ,\beta )
=
c_\sigma (\alpha ,\beta )$ in \eqref{Eq:FormalExp}.
Then \eqref{Eq:FormalExp} becomes
\begin{align}
f
&=
\mathfrak H(c(f))
=
\sum _{\alpha \in \nn d}
c(f; \alpha )h_\alpha
\label{Eq:FormalExp2First}
\intertext{and}
\fka
&=
\mathfrak H_\sigma (c_\sigma (\fka ))
\sum _{\alpha ,\beta
\in \nn d}c_\sigma (\fka ;\alpha ,\beta )
h_{\scW ,(\alpha ,\beta )}.
\label{Eq:FormalExp2Sec}
\end{align}

\par

The definition of the broad class of spaces is based on suitable estimates on the 
(symplectic) Hermite coefficients in
(symplectic) Hermite series expansions. Therefore, we first continue
our discussion from Subsection \ref{subsec1.6}, and consider analogous
topological spaces of sequences on $\nn d$.
Suitable conditions on such spaces are given in the following
definition. Here we especially consider mappings of the form
\begin{equation}
    \label{Eq:SeqSpRotInvariant}
    \{ c(\alpha )\} _{\alpha \in \nn d}
    \mapsto 
    \{ e^{-i \scal \varrho \alpha \frac{\pi}{2}} c( \alpha )\} _{\alpha \in \nn d} ,
    \qquad \varrho \in \mathbf C^d .
\end{equation}

\par

\begin{defn}
\label{Def:PointWiseInv}
Let $\Omega \subseteq \cc d$ and $\maclB
\hookrightarrow \ell_0'(\mathbf{N}^{d})$ be a topological vector space.
\begin{enumerate}
\item $\maclB$ is called \emph{pointwize rotational invariant with
respect to $\Omega$}, or $\Omega$-invariant, if the map
\eqref{Eq:SeqSpRotInvariant} is a homeomorphism on $\maclB$,
for every $\varrho \in \Omega$.

\vrum

\item $\maclB$ is called \emph{pointwize rotational semi-invariant with
respect to $\Omega$}, or $\Omega$-semi-invariant, if the map
\eqref{Eq:SeqSpRotInvariant} is continuous on $\maclB$,
for every $\varrho \in \Omega$.
\end{enumerate}
\end{defn}

\par

In the following definition, we give 
examples on rotational invariant
subspaces of $\ell _0'(\nn d)$, which 
can be found in, e.{\,}g.,
\cite{NeToVi2,Toft18,ToBhMa}. For the 
index set, we extend
$\mathbf R_+$ into
\begin{alignat*}{2}
\mathbf R _\flat
&\equiv
\mathbf R _+ \,\bigcup\, \sets {\flat _\sigma}{\sigma >0}, &
\quad
\overline {\mathbf R _\flat} &= \mathbf R _\flat \cup \{ 0\} ,
\\[1ex]
\mathbf R _{\flat ,\infty}
&\equiv
\mathbf R _\flat \cup \{ \infty \} , &
\qquad \text{and}\qquad
\overline{\mathbf R _{\flat ,\infty}}
&=
\overline{\mathbf R _\flat} \cup \{ \infty \} ,
\end{alignat*}
and order the involved elements as
$$
x_1 <\flat _{\sigma _1}<\flat _{\sigma _2}<x_2<\infty ,
$$
when $\sigma _1,\sigma _2\in \mathbf R _+$ and $x_1,x_2\in \overline{\mathbf R _+}$
satisfy $\sigma _1 <\sigma _2$ and $x_1<\frac 12\le x_2$.
%%
%for $\mathbf R _\flat$.
(See also \cite{Toft18}.)

\par

\begin{defn}
\label{Def:SeqSpaces1}
Let $p\in (0,\infty ]$, $r\in \mathbf R _+$,
$s\in \mathbf R _{\flat ,\infty}$, $\vartheta$ be a
positive function on $\nn d$, and let
\begin{equation}
\vartheta _{r,s}(\alpha )
\equiv
\begin{cases}
e^{r|\alpha |^{\frac 1{2s}}}, & s\in \mathbf R _+,
\\[1ex]
r^{|\alpha |}\alpha !^{\frac 1{2\sigma}}, & s=\flat _\sigma ,
\\[1ex]
(1+|\alpha |)^r, & s= \infty , 
\end{cases}
\qquad \alpha \in \nn d
\end{equation}
\begin{enumerate}
\item $\ell _{[\vartheta ]}^p(\nn d)$ is the quasi-Banach
space which consists of
all $c\in \ell _0'(\nn d)$ such that
$$
\nm c{\ell _{[\vartheta ]}^p}
\equiv
\nm {\{ c(\alpha )\vartheta (\alpha )\} _{\alpha \in \nn d}}{\ell ^p}
$$
is finite.

\vrum

\item $\ell _{0,s}(\nn d)\equiv \bigcap _{r>0}\ell _{[\vartheta _{r,s}]}^p(\nn d)$
and $\ell _{s}(\nn d)\equiv \bigcup _{r>0}\ell _{[\vartheta _{r,s}]}^p(\nn d)$,
with projective and inductive limit topologies, respectively
%of $\ell _{[\vartheta _{r,s}]}^p(\nn d)$ with respect to $r>0$.

\vrum

\item $\ell _{0,s}'(\nn d)\equiv \bigcup _{r>0}\ell _{[1/\vartheta _{r,s}]}^p(\nn d)$
and $\ell _{s}'(\nn d)\equiv \bigcap _{r>0}\ell _{[1/\vartheta _{r,s}]}^p(\nn d)$,
with inductive and projective limit topologies, respectively
%of $\ell _{[\vartheta _{r,s}]}^p(\nn d)$ with respect to $r>0$.
\end{enumerate}
\end{defn}

\par

We observe that the spaces in Definition \ref{Def:SeqSpaces1}
(2) and (3) are independent of $p$, and by choosing $p\ge 1$,
it follows that $\ell _{0,s}(\nn d)$ and $\ell _s'(\nn d)$ are Fr{\'e}chet
spaces. Furthermore, by straight-forward computations it follows
that
\begin{alignat}{2}
&(\ell _s(\nn d), \ell ^2(\nn d),\ell _s'(\nn d)),&
\qquad
s&\in \overline{\mathbf R _{\flat}},
\label{Eq:Gelftriple2}
\intertext{and}
&(\ell _{0,s}(\nn d), \ell ^2(\nn d),\ell _{0,s}'(\nn d)),&
\qquad
s&\in \mathbf R _{\flat ,\infty},
\label{Eq:Gelftriple2A}
\end{alignat}
are Gelfand triples. (See Subsection \ref{subsec1.5}.)

\par

If $s_1\in \mathbf R _{\flat ,\infty}$ and $s_2\in \overline{\mathbf R _{\flat ,\infty}}$,
then $\ell _{[\vartheta ]}^p(\nn d)$ and all the spaces
\begin{equation}
\label{Eq:StandPilSpaces}
\ell _{0,s_1}(\nn d),
\quad
\ell _{s_2}(\nn d),
\quad
\ell _{s_2}'(\nn d),
\quad \text{and}\quad
\ell _{0,s_1}'(\nn d),
\end{equation}
in Definition \ref{Def:SeqSpaces1} are 
rotation invariant with respect to $\rr d$. The following
proposition shows, among others, that some of the spaces are
even rotation invariant with respect to $\cc d$. The result
follows by straight-forward computations, and are implicitly
given in \cite{ToBhMa}. The proof is therefore omitted. Here,
$$
\cc d _+ =\sets {(z_1,\dots ,z_d)\in \cc d}{\IM(z_j)>0}
$$
and
$$
\cc d _- =\sets {(z_1,\dots ,z_d)\in \cc d}{\IM(z_j)<0}.
$$

\par

\begin{prop}
Let $s_1\in \mathbf R _{\flat ,\infty}$ and
$s_2\in \overline{\mathbf R _{\flat ,\infty}}$. Then the following are true:
\begin{enumerate}
\item all spaces in \eqref{Eq:StandPilSpaces} are rotation invariant
with respect to $\rr d$;

\item if, more restricted, $s_1\le \frac 12$ and $s_2<\frac 12$,
then all spaces in \eqref{Eq:StandPilSpaces} are rotation invariant
with respect to $\cc d$;

\vrum

\item if, more restricted, $s_1> \frac 12$ and $s_2\ge \frac 12$,
then all spaces in \eqref{Eq:StandPilSpaces} are rotation semi-invariant
with respect to $\cc d_-$;

\vrum

\item if $\varrho \in \cc d_+$, $T_\varrho$ is
the map \eqref{Eq:SeqSpRotInvariant},
and, more restricted,
$s_1> \frac 12$ and $s_2\ge \frac 12$,
then
$$
T_\varrho : \ell _{0.s_1}(\nn d)\to \ell _{0.s_1}'(\nn d)
\quad \text{and}\quad
T_\varrho : \ell _{s_2}(\nn d)\to \ell _{s_2}'(\nn d)
$$ 
are discontinuous.
\end{enumerate}
\end{prop}

\par

\subsection{Broad classes of Pilipovi{\'c}
spaces and Wigner-Pilipovi{\'c} spaces}
\label{subsec5.2}

\par

We are now prepared to give the definitions of Pilipovi{\'c} and
Pilipovi{\'c}-Wigner spaces.
 
\par

\begin{defn}
\label{Def:GenPilSpaces}
Let $\maclB \hookrightarrow \ell _0'(\nn d)$
and $\maclC \hookrightarrow \ell _0'(\nn {2d})$ be topological
vector spaces.
\begin{enumerate}
\item The \emph{Pilipovi{\'c} space with respect to $\maclB$},
or the \emph{$\maclB$-Pilipovi{\'c} space},
is given by
\begin{align}
\maclH _\maclB (\rr d)
&\equiv
\sets{f\in \maclH _0'(\rr d)}
{\{ c(f;\alpha )\} _{\alpha \in \nn d}\in \maclB},
\end{align}
with the topology inherited from $\maclB$.

\vrum

\item The \emph{Pilipovi{\'c}-Wigner space with respect to $\maclC$},
or the \emph{$\maclC$-Pilipovi{\'c}-Wigner space},
is given by
\begin{align}
\maclWH _{\maclC} (\rr {2d})
&\equiv
\sets{\fka \in \maclWH _{0}'(\rr {2d})}
{\{ c_\sigma (\fka ;\alpha ,\beta )\} _{\alpha ,\beta \in \nn d}\in \maclC},
\end{align}
with the topology inherited from $\maclC$.
\end{enumerate}
\end{defn}

\par

By the topological properties of $\maclH _\maclB (\rr d)$
and $\maclWH _{\maclC}(\rr {2d})$ in Definition
\ref{Def:GenPilSpaces}, it follows that the mappings
\begin{alignat}{2}
\mathfrak H
& : & 
\maclB &\to \maclH _{\maclB}(\rr d)
\label{Eq:SeqHermExpMapRest}
\intertext{and}
\mathfrak H_\sigma
& :& 
\maclC 
&\to 
\maclWH _{\maclC}(\rr {2d}),
\label{Eq:SeqWigHermExpMapRest}
\end{alignat}
are homeomorphisms.

\par

If, more restrictive, $\maclB$ and $\maclC$
are quasi-normed spaces, then $\maclH _\maclB (\rr d)$
and $\maclWH _{\maclC}(\rr {2d})$ are quasi-normed spaces
of the same respective degrees. Furthermore, by defining
the quasi-norms on $\maclH _\maclB (\rr d)$
and $\maclWH _{\maclC}(\rr {2d})$ as
\begin{alignat}{2}
\nm f{\maclH _\maclB}
&\equiv
\nm {c(f)}{\maclB}, &
\quad
c(f)&=\{c(f;\alpha)\} _{\alpha \in \nn d} \in \ell _0'(\nn d),
\label{Eq:HermExpIsom}
\intertext{and}
\nm \fka {\maclWH _{\maclC}}
&\equiv
\nm {c_\sigma (\fka)}{\maclC}, &
\quad
c_\sigma (\fka )&=\{ c_\sigma (\fka ;\alpha ,\beta )\}
_{\alpha ,\beta \in \nn d} \in \ell _0'(\nn {2d}),
\label{Eq:HermWigExpIsom}
\end{alignat}
it follows that the mappings in \eqref{Eq:SeqHermExpMapRest}
and \eqref{Eq:SeqWigHermExpMapRest} are isometric bijections.
Especially, \eqref{Eq:HermExpIsom} and
\eqref{Eq:HermWigExpIsom} take the forms
\begin{alignat}{2}
\nm f{\maclH _\maclB}
&\equiv
\nm {c(f)}{\maclB}, &
\qquad
f&\in \maclH _\maclB (\rr d),
\tag*{(\ref{Eq:HermExpIsom})$'$}
\intertext{and}
\nm f{\maclH _\maclC}
&\equiv
\nm {c_\sigma (\fka)}{\maclC}, &
\qquad
\fka &\in \maclWH _{\maclC} (\rr {2d}),
\tag*{(\ref{Eq:HermWigExpIsom})$'$}
\end{alignat}

\par

Since $\{ h_\alpha \} _{\alpha \in \nn d}$ and
$\{ h_{\scW,(\alpha ,\beta }\} _{\alpha ,\beta \in \nn d}$
are orthonormal bases for $L^2(\rr d)$ and $L^2(\rr {2d})$, respectively,
it follows that
\begin{equation*}
\maclH _\maclB (\rr d) = L^2(\rr d)
\quad \text{and}\quad 
\maclWH _{\maclC} (\rr {2d}) = L^2(\rr {2d}),
\end{equation*}
also in norms, when $\maclB = \ell ^2(\nn d)$ and
$\maclC = \ell ^2(\nn {2d})$, and that
\begin{align}
(f,g)_{L^2} &= (c(f),c(g))_{\ell ^2}
=
\sum _{\alpha \in \nn d} c(f;\alpha )\overline{c(g;\alpha )},
\quad
f,g\in L^2(\rr d)
\intertext{and}
(\fka ,\fkb )_{L^2}
&=
(c_\sigma (\fka ),c_\sigma (\fkb ))_{\ell ^2}
=
\sum _{\alpha ,\beta \in \nn d} c_\sigma (f;\alpha ,\beta )
\overline{c_\sigma (g;\alpha ,\beta )},
\quad
\fka ,\fkb \in L^2(\rr {2d}).
\end{align}

\par

For $s_1\in \mathbf R _{\flat ,\infty}$ and
$s_2\in \overline{\mathbf R _{\flat ,\infty}}$, we set
\begin{alignat}{4}
\maclH _{\maclB _1} &= \maclH _{0,s_1} ,&
\ %quad
\maclH _{\maclB _2} &= \maclH _{s_2} ,&
\ %\quad
\maclH _{\maclB _3} &= \maclH _{s_2}' ,&
\ %\quad
\maclH _{\maclB _4} &= \maclH _{0,s_1}' ,
\intertext{when}
\maclB _1 &= \ell _{0,s_1}(\nn d), &
\ %\quad
\maclB _2 &= \ell _{s_2}(\nn d),&
\ %\quad
\maclB _3 &= \ell _{s_2}'(\nn d), &
\ %\quad
\maclB _4 &= \ell _{0,s_1}'(\nn d),
\end{alignat}
respectively, and define
$$
(f,g)
=
(f,g)_{L^2}
\equiv
\sum _{\alpha \in \nn d}
c(f;\alpha )\overline{c(\phi ;\alpha )},
$$
when $f\in \maclH _{0,s_1}'(\rr d)$ and $g\in \maclH _{0,s_1}(\rr d)$,
or when $f\in \maclH _{s_2}'(\rr d)$ and $g\in \maclH _{s_2}(\rr d)$.
In similar ways, set
\begin{alignat}{4}
\maclWH _{\maclC _1} &= \maclWH _{0,s_1} ,&
\ %quad
\maclWH _{\maclC _2} &= \maclWH _{s_2} ,&
\ %\quad
\maclWH _{\maclC _3} &= \maclWH _{s_2}' ,&
\ %\quad
\maclWH _{\maclC _4} &= \maclWH _{0,s_1}' ,
\notag
\intertext{when}
\maclC _1 &= \ell _{0,s_1}(\nn {2d}), &
\ %\quad
\maclC _2 &= \ell _{s_2}(\nn {2d}),&
\ %\quad
\maclC _3 &= \ell _{s_2}'(\nn {2d}), &
\ %\quad
\maclC _4 &= \ell _{0,s_1}'(\nn {2d}),
\notag
\end{alignat}
respectively, and define
$$
(\fka ,\fkb )
=
(\fka ,\fkb )_{L^2}
\equiv
\sum _{\alpha ,\beta \in \nn d}
c(\fka ;\alpha ,\beta ),
$$
when $\fka \in \maclWH _{0,s_1}'(\rr {2d})$ and
$\fkb \in \maclWH _{0,s_1}(\rr {2d})$,
or when $\fka \in \maclWH _{s_2}'(\rr {2d})$
and $\fkb \in \maclWH _{s_2}(\rr {2d})$.

\par

The following two propositions explains relationships
between Pilipovi{\'c} spaces and Gelfand-Shilov spaces.

\par

\begin{prop}
\label{Prop:GelShilPilIdent}
Let $s_1>\frac 12$ and $s_2\ge \frac 12$.
Then
\begin{equation}
\begin{alignedat}{2}
\mascS (\rr d) &= 
\maclH _{0,\infty}(\rr d), &
\quad
\mascS '(\rr d) &= 
\maclH _{0,\infty}'(\rr d),
\\[1ex]
\Sigma _{s_1}(\rr d)
&=
\maclH _{0,s_1}(\rr d), &
\quad
\Sigma _{s_1}'(\rr d)
&=
\maclH _{0,s_1}'(\rr d),
\\[1ex]
\maclS _{s_2}(\rr d)
&=
\maclH _{s_2}(\rr d), &
\qquad \text{and}\qquad
\maclS _{s_2}'(\rr d)
&=
\maclH _{s_2}'(\rr d),
\end{alignedat}
\end{equation}
also in topological sense.
\end{prop}

\par

\begin{prop}
\label{Prop:PilSpEmb}
Suppose $s,t\in \mathbf R _{\flat ,\infty}$. Then the inclusions
\begin{align}
\maclH _0(\rr d)
&\hookrightarrow
\maclH _{0,s}(\rr d)
\hookrightarrow
\maclH _{s}(\rr d)
\hookrightarrow
\maclH _{0,t}(\rr d)
\notag
\\[1ex]
&\hookrightarrow
\mascS (\rr d)
\hookrightarrow
L^2(\rr d)
\hookrightarrow
\mascS '(\rr d)
\notag
\\[1ex]
&\hookrightarrow
\maclH _{0,t}'(\rr d)
\hookrightarrow
\maclH _{s}'(\rr d)
\hookrightarrow
\maclH _{0,s}'(\rr d)
\hookrightarrow
\maclH _{0}'(\rr d),
\quad s<t,
\label{Eq:PilSpaceIncl}
\end{align}
and
\begin{align}
\maclWH _{0}(\rr {2d})
&\hookrightarrow
\maclWH _{0,s}(\rr {2d})
\hookrightarrow
\maclWH _{s}(\rr {2d})
\hookrightarrow
\maclWH _{0,t}(\rr {2d})
\notag
\\[1ex]
&\hookrightarrow
\mascS (\rr {2d})
\hookrightarrow
L^2(\rr {2d})
\hookrightarrow
\mascS '(\rr {2d})
\notag
\\[1ex]
&\hookrightarrow
\maclWH _{0,t}'(\rr {2d})
\hookrightarrow
\maclWH _{s}'(\rr {2d})
\hookrightarrow
\maclWH _{0,s}'(\rr {2d})
\hookrightarrow
\maclWH _{0}'(\rr {2d}),
\quad s<t,
\label{Eq:SymplPilSpaceIncl}
\end{align}
are dense.
\end{prop}

\par

As a consequence of Proposition \ref{Prop:GelShilPilIdent} it
follows that the embeddings in \eqref{Eq:PilSpaceIncl}
extend the embeddings in \eqref{Eq:GSEmb}.

\par

A proof of Proposition 
\ref{Prop:GelShilPilIdent} is given in
\cite{Toft18}. (See also \cite{JaEi}
for related investigations, and 
\cite{NeToVi2} for extensions and generalizations.)
The embeddings in \eqref{Eq:PilSpaceIncl} and \eqref{Eq:SymplPilSpaceIncl}
follow from the definitions. The density of these embeddings are also
straight-forward, see e.{\,}g. \cite{Toft18} for more details.

\par

By Proposition \ref{Prop:GelShilPilIdent}, it follows that
a Gelfand-Shilov space of functions and distributions agrees
with corresponding classical
Pilipovi{\'c} spaces, if and only if the Gelfand-Shilov space
is non-trivial.

\par

By \eqref{Eq:HermAndSympHermRel} and the density of
the mappings in Proposition \ref{Prop:PilSpEmb}, it follows
that
\begin{align}
\mascW : \maclH _0(\rr {2d})
&\to
\maclWH _{0}(\rr {2d})
\intertext{is a homeomorphism, which is uniquely extendable
to a homeomorphism}
\mascW : \maclH _0'(\rr {2d})
&\to
\maclWH _{0}'(\rr {2d}).
\intertext{The homeomorphism can also be described by}
\mascW &= \mathfrak H _\sigma \circ \mathfrak H ^{-1},
\end{align}
or alternatively, by
\begin{equation}
\begin{aligned}
\mascW F
&=
\sum _{\alpha ,\beta \in \nn d}
c(F; (\alpha ,\beta ))h_{\scW,(\alpha ,\beta)},
\\[1ex]
\text{when}\quad
F&=\sum _{\alpha ,\beta \in \nn d}
c(F; (\alpha ,\beta ))h_{(\alpha ,\beta)} \in \maclH _0'(\rr {2d}).
\end{aligned}
\end{equation}

\par

From these observations it follows that the homeomorphisms
in \eqref{Eq:SeqHermExpMapRest}
and
\eqref{Eq:SeqWigHermExpMapRest} are extended as in the
following proposition. The verifications are left to the reader.

\par

\begin{prop}
\label{Prop:BHBSymplHBRel}
Let $\maclB \hookrightarrow \ell _0'(\nn {2d})$ be
a topological vector space. Then the mappings in
the commutative diagram
\begin{equation}
\label{Eq:BHBSymplHBRel}
\begin{tikzcd}
\maclB \arrow[dr , "\mathfrak H _\sigma" '] \arrow[r, "\mathfrak H"] 
& \maclH _\maclB (\rr {2d})\arrow[d, "\mascW"]
\\
& \maclWH _{\maclB} (\rr {2d})
\end{tikzcd}
\end{equation}
are homeomorphisms.
\end{prop}

\par

\subsection{Mapping properties of fractional
symplectic Fourier transforms on Pilipovi{\'c} spaces}
\label{subsec5.3}

\par

We are now prepared to discuss extensions of
fractional Fourier transforms and fractional symplectic
Fourier transforms. Let $\rho ,\theta \in \rr d$. One has
\begin{equation}
\label{Eq:FracFTHermite}
\mascF _\rho h_\alpha = e^{-i\scal \rho \alpha \frac \pi 2}h_\alpha .
\end{equation}
(See e.{\,}g. 
\cite{KutOza,OzaZalKut,ToBhMa}.)
Using this fact in combination
with \eqref{Eq:FracSympAct},
%and \eqref{Eq:FracFTHermite},
one obtains
\begin{equation}
\label{eq:PartFracSymplFTSymplHermite}
\mascF_{\sigma, (\rho, \theta)} h_{\scW ,(\alpha, \beta)}
=
e^{i (\langle \rho + \theta, \alpha \rangle
+
\langle \rho - \theta, \beta \rangle) \frac{\pi}{2}} h_{\scW ,(\alpha, \beta)} ,
\qquad
\alpha, \beta \in \mathbf{N}^d . 
\end{equation}
In particular,
\[
\mascF_{\sigma} h_{\scW ,(\alpha, \beta)}
=
(-1)^{\alpha} h_{\scW ,(\alpha, \beta)} ,
\qquad \alpha, \beta \in \nn d .
\]
It follows from \eqref{Eq:FracFTHermite} that
\begin{align}
\mascF _\rho f
&=
\sum _{\alpha \in \nn d}
e^{-i\scal \rho \alpha}c(f;\alpha)h_\alpha ,
\label{Eq:FracFourOnExpan}
\intertext{or, equivalently, that}
c(\mascF _\rho f;\alpha)
&=
e^{-i\scal \rho \alpha}c(f;\alpha) ,
\label{Eq:FracFourExpanCoeff}
\intertext{when $f\in \maclH _0(\rr d)$. In the
same way it follows from 
\eqref{eq:PartFracSymplFTSymplHermite} that}
\mascF _{\sigma ,(\rho ,\theta )} \fka
&= \sum _{\alpha ,\beta \in \nn d}
e^{i\scal {\rho +\theta} \alpha +\scal {\rho -\theta}\beta }
c_\sigma (\fka ;\alpha ,\beta )h_{\scW,(\alpha ,\beta )}, 
\label{Eq:FracSymplFourOnExpan}
\intertext{or, equivalently, that}
c(\mascF _{\sigma ,(\rho ,\theta)} \fka ;\alpha ,\beta)
&=
e^{i\scal {\rho +\theta} \alpha +\scal {\rho -\theta}\beta }
c_\sigma (\fka ;\alpha ,\beta ), 
\label{Eq:FracSymplFourOnExpanCoeff}
\end{align}
when $\fka \in \maclWH _{0}(\rr {2d})$.

\par

Suppose, more generally, that $\rho ,\theta \in \cc d$. Then
we define $\mascF _\rho f$ and $\mascF _{\sigma ,(\rho ,\theta )}\fka$
by \eqref{Eq:FracFourOnExpan} and \eqref{Eq:FracSymplFourOnExpan},
when $f\in \maclH _0(\rr d)$ and $\fka \in \maclWH _{0}(\rr {2d})$.
Then $\mascF _\rho$ and $\mascF _{\sigma ,(\rho ,\theta )}$
remain homeomorphisms on $\maclH _{0}(\rr d)$ and
$\maclWH _{0}(\rr {2d})$, respectively.

\par

The following two theorems are shown by straight-forward combinations
of Proposition \ref{Prop:BHBSymplHBRel},
\eqref{Eq:FracFTHermite}, and
\eqref{eq:PartFracSymplFTSymplHermite}. The 
details are left to the reader. Here
we observe that for sequences on
$\nn {2d}$, \eqref{Eq:SeqSpRotInvariant}
becomes
\begin{equation}
    \label{Eq:SeqSpRotInvariantSympl}
    \{ c(\alpha ,\beta )\}
    _{\alpha ,\beta \in \nn d}
    \mapsto 
    \{ e^{-i(\scal \rho \alpha
    +\scal \theta \beta )\frac{\pi}{2}} c( \alpha ,\beta )\}
    _{\alpha ,\beta \in \nn d} ,
    \qquad \varrho \in \mathbf C^d .
\end{equation}

\par

\begin{thm}
\label{Thm:FracFourGenPil}
Let $\Omega \subseteq \cc d$ and let $\maclB \hookrightarrow \ell_0'(\nn d)$ 
be a topological vector space which is rotational invariant with respect to $\Omega$.
\begin{enumerate}
    \item If $\rho \in \cc d$, then $\mascF_{\rho}$ on $\maclH _0(\rr d)$
    extends uniquely to a homeomorphism on $\maclH_{0}'(\rr d)$.
    Furthermore, \eqref{Eq:FracFourOnExpan} holds when
    $f\in \maclH_0'(\rr d)$.
    \vrum
    \item If $\rho \in \Omega$, then $\mascF_{\rho}$ on $\maclH _0'(\rr d)$
    restricts to a homeomorphism on $\maclH _\maclB (\rr d)$.
    \vrum
    \item If in addition $\maclB$ is a
    quasi-Banach space and that for every
    $\varrho \in \Omega$, the map
    \eqref{Eq:SeqSpRotInvariant} is
    quasi-norm preserved on $\maclB$, then
    $\mascF _\rho$ is isometric on
    $\maclH _\maclB (\rr d)$.
\end{enumerate}
\end{thm}

\par

\begin{thm}
\label{Thm:FracSymplFourGenPil}
Let $\Omega \subseteq \cc d$ and let $\maclC \hookrightarrow \ell_0'(\nn {2d})$ 
be a topological vector space which is rotational invariant with respect to $\Omega \times \Omega$.
\begin{enumerate}
    \item If $\rho ,\theta \in \cc d$, then $\mascF_{\sigma ,(\rho ,\theta )}$
    on $\maclWH _{0}(\rr {2d})$
    extends uniquely to a homeomorphism on
    $\maclWH _{0}'(\rr {2d})$.
    Furthermore, \eqref{Eq:FracSymplFourOnExpan} holds when
    $\fka \in \maclWH _{0}'(\rr {2d})$.
    \vrum
    \item If $\rho ,\theta \in \Omega$, then
    $\mascF_{\sigma ,(\rho ,\theta )}$ on $\maclWH _{0}'(\rr {2d})$
    restricts to a homeomorphism on $\maclWH _{\maclC}(\rr {2d})$.
    \vrum
    \item If in addition $\maclC$ is a
    quasi-Banach space and that for every
    $\rho ,\theta \in \Omega$, the map
    \eqref{Eq:SeqSpRotInvariantSympl} is
    quasi-norm preserved on $\maclC$, then
    $\mascF _{\sigma ,(\rho ,\theta )}$
    is isometric.
\end{enumerate}
\end{thm}

\par

\begin{rem}
\label{Rem:FracFourGSSpaces}
Let $\maclB$ be as in \eqref{Eq:StandPilSpaces},
when $s_1>\frac 12$ and $s_2\ge \frac 12$.
Then Proposition \ref{Prop:BasicMapsOnGS}
in Section \ref{sec4}
follows by choosing
$\Omega =\rr d$ in Theorem \ref{Thm:FracFourGenPil}.

\par

In the same way it follows that Theorem 
\ref{Thm:MapPartFracSymplFTGS} 
in Section \ref{sec4}
follows by choosing $\Omega =\rr d$ in
Theorem \ref{Thm:FracSymplFourGenPil}.
\end{rem}

\medspace

In the following, we present a condition on $\maclB$ that is sufficient
for $\maclH _{\maclB}(\rr {2d})$ to agree with
$\maclWH _{\maclB}(\rr {2d})$. 

\par

\begin{prop}
\label{Prop:SuffCondPartFracSymplFTHomeo}
Let $\maclB \hookrightarrow \ell _0'(\nn {2d})$ be a topological vector space which is rotational invariant with respect to $\rr {2d}$ such that
$\ell _0(\nn {2d})\hookrightarrow \maclB$. 
Also, let $\rho ,\theta \in \rr d$
and suppose that the map
\begin{equation}
\label{Eq:WignerDilation}
F(x,y) \mapsto F(x+\textstyle{\frac y2},x-\textstyle{\frac y2})
\end{equation}
is a homeomorphism on $\maclH _\maclB (\rr {2d})$. Then
$\maclH _{\maclB} (\rr {2d})
=
\maclWH _{\maclB} (\rr {2d})$,
and $\mascF _{\sigma ,(\rho ,\theta )}$
on $\maclWH _0'(\rr {2d})$ restricts to 
a homeomorphism on
$\maclH _{\maclB} (\rr {2d})$.
\end{prop}

\par

\begin{prop}
\label{Prop:SuffCondPartFracSymplFTHomeo2}
Let $\maclB \hookrightarrow \ell _0'(\nn {2d})$ 
be a topological vector space which is 
rotational invariant with respect to $\rr {2d}$ 
such that
$\ell _0(\nn {2d})\hookrightarrow \maclB$, 
$\maclB ' \hookrightarrow \ell_0'(\nn {2d})$, and 
\begin{equation}
\label{Eq:GenGelfTripleSeq}
(\maclB ,\ell ^2(\nn {2d}),\maclB ')
\end{equation}
is a Gelfand triple.  %with $\maclB '\hookrightarrow \ell _0'(\nn {2d})$.
Also, suppose that the map
\begin{equation}
\label{Eq:WignerDilation2}
F(x,y) \mapsto F(x+\textstyle{\frac y2},x-\textstyle{\frac y2})
\end{equation}
is a homeomorphism on $\maclH _\maclB (\rr {2d})$. Then
the following are true:
\begin{enumerate}
\item it holds
$$
\maclH _{\maclB} (\rr {2d}) = \maclWH _{\maclB} (\rr {2d})
\quad \text{and}\quad
\maclH _{\maclB} '(\rr {2d})
=
\maclH _{\maclB '} (\rr {2d})
=
\maclWH _{\maclB '} (\rr {2d});
$$

\vrum

\item if $\rho ,\theta \in \rr d$, then $\mascF _{\sigma ,(\rho ,\theta )}$
on $\maclWH _0'(\rr {2d})$ restricts to homeomorphisms on
$\maclH _{\maclB} (\rr {2d})$ and on $\maclH _{\maclB '} (\rr {2d})$.
\end{enumerate}
\end{prop}

\par

\begin{rem}
Let $\maclB$ in Proposition
\ref{Prop:SuffCondPartFracSymplFTHomeo}
additionally satisfy $\ell _0(\nn {2d})
\hookrightarrow \maclB$, and that
$$
\lim _{N\to \infty }\Pi _{2d,N}c = c ,
$$
with convergence in $\maclB$, when $c\in \maclB$.
Here, $\Pi _{d,N}$ is the projection on
$\ell _0'(\nn d)$, defined by
\begin{equation}
(\Pi _{d,N}c)(\alpha )
\equiv
\begin{cases}
c(\alpha ), & |\alpha |\le N
\\[1ex]
0, & |\alpha |>N,\quad c=\{ c(\alpha )\}_{\alpha \in \nn d}.
\end{cases}
\end{equation}
Then $\maclB ' \hookrightarrow \ell _0'(\nn {2d})$, and
$$
\sum _{|\alpha +\beta |\le N}c(F;\alpha ,\beta )
h_{(\alpha ,\beta )}
\to
F
$$
uniformly in $\maclH _{\maclB}(\rr {2d})$,
when $N$ tends to infinity, for every
$F\in \maclH _{\maclB}(\rr {2d})$.

\par

It follows that
$\maclH _{\maclB}'(\rr {2d})
=
\maclH _{\maclB '}(\rr {2d})$. Furthermore,
if, in addition, \eqref{Eq:WignerDilation} is a 
homeomorphism
on $\maclH _{\maclB}'(\rr {2d})$, then
$\maclWH _{\maclB '}(\rr {2d})$
% A sufficient condition for \eqref{Eq:GenGelfTripleSeq}
% to be a Gelfand triple with
% $\maclB ' \hookrightarrow \ell _0'(\nn {2d})$,
% and thereby the conclusions for Proposition
% \ref{Prop:SuffCondPartFracSymplFTHomeo}
% to hold, is that
% $$
% \lim _{N\to \infty }\Pi _{2d,N}c = c ,
% $$
% with convergence in $\maclB$, when $c\in \maclB$.
% In fact, the latter condition implies that $\ell _0(\nn {2d})$,
% is dense in $\maclB$, which implies that
% $\maclB ' \hookrightarrow \ell _0'(\nn {2d})$.

Let $\maclB$ in Proposition
\ref{Prop:SuffCondPartFracSymplFTHomeo},
and suppose $\maclB ' \hookrightarrow
\ell _0'(\nn {2d})$, and that
\eqref{Eq:WignerDilation} is a homeomorphism
on $\maclH _{\maclB}'(\rr {2d})$. Then
$$
\maclH _{\maclB} '(\rr {2d})
=
\maclH _{\maclB '} (\rr {2d})
=
\maclWH _{\maclB '} (\rr {2d}),
$$
and $\mascF _{\sigma ,(\rho ,\theta )}$
on $\maclWH _0'(\rr {2d})$ restricts to a 
homeomorphism on
$\maclH _{\maclB '} (\rr {2d})$.

\end{rem}

\par

\begin{rem}
\label{Rem:GSTriplHomeomorph}
Suppose $\maclB \hookrightarrow \ell _0'(\nn {2d})$ be  a
topological vector space such that
$\ell _0(\nn {2d})\hookrightarrow \maclB$,
$\maclB '\hookrightarrow \ell _0'(\nn {2d})$, and that
\eqref{Eq:GenGelfTripleSeq} is a Gelfand triple. Then
it follows by straight-forward computations
that the mappings \eqref{Eq:BHBSymplHBRel}
induce homeomorphisms of Gelfand triples
\begin{equation}
\label{Eq:BHBSymplHBRelGeShTriple}
\begin{tikzcd}
\big ( \maclB ,\ell ^2(\nn {2d}), \maclB ' \big )
\arrow[dr , "\mathfrak H _\sigma" '] \arrow[r, "\mathfrak H"] 
& \big ( \maclH _\maclB (\rr {2d}), L^2(\rr {2d}), \maclH _{\maclB '} (\rr {2d}) \big )
\arrow[d, "\mascW"]
\\
& \big ( \maclWH _\maclB (\rr {2d}), L^2(\rr {2d}),
\maclWH _{\maclB '} (\rr {2d}) \big ) .
\end{tikzcd}
\end{equation}
The details are left for the reader.
\end{rem}

\par

\begin{rem}
We observe that
Theorem \ref{Thm:MapPartFracSymplFTGS}
follows by combining
Proposition
\ref{Prop:GelShilPilIdent} with
Proposition
\ref{Prop:SuffCondPartFracSymplFTHomeo},
which gives another proof of that 
theorem, compared
to what is performed in Section 
\ref{sec4}.
\end{rem}

\par

\begin{proof}[Proof of Proposition
\ref{Prop:SuffCondPartFracSymplFTHomeo}]
Since $\maclH_{\maclB}(\rr {2d})$ is pointwize rotational
invariant, it follows from  Theorem \ref{Thm:FracFourGenPil}
that $\mascF_\rho$, $\rho \in \rr {2d}$, restricts 
to a homeomorphism on $\maclH_{\maclB}(\rr {2d})$.
In particular, $I\otimes \mascF$ is a homeomorphism on
$\maclH_{\maclB}(\rr {2d})$, where $I$ is the identity map.
By using this and the fact that the map
\eqref{Eq:WignerDilation} is a homeomorphism on
$\maclH_{\maclB}(\rr {2d})$, it also follows that $\mascW$
is homeomorphic on $\maclH_{\maclB}(\rr {2d})$.
The assertion (1) now follows by duality
and the conclusions in Remark
\ref{Rem:GSTriplHomeomorph}.

\par

A combination of (1) and
Theorem \ref{Thm:FracSymplFourGenPil}
gives (2), and the result follows.
\end{proof}

\par

%%%%%%%%%%%%%%%%%%%%%%%%%%
\section{Mapping properties of fractional symplectic Fourier
transforms on a broad family of modulation spaces}
\label{sec6}
%%%%%%%%%%%%%%%%%%%%%%%%%%

\par

In this section, we deduce homeomorphisms
of partial fractional Fourier transforms and partial fractional
symplectic Fourier transforms on broad families
of modulation spaces. These modulation spaces
are parameterized by weight functions and
solid quasi-Banach function spaces on
the phase space. In the first part, we
provide necessary definitions.
We also pay some attention to
symplectic analogies of modulation spaces,
so-called symplectic modulation spaces.
Thereafter, we deduce homeomorphisms
for partial fractional Fourier transforms on modulation
spaces. The section is finished by finding
homeomorphisms for partial fractional symplectic
Fourier transforms on symplectic modulation
spaces.

\par

%\subsection{Solid quasi-Banach function spaces and modulation
%spaces}
%
%\par

\subsection{A broad family of modulation spaces}

\par

In order to define modulation spaces, we recall that
the short-time Fourier transform of $f\in \mascS (\rr d)$
with respect to the window function $\phi \in \mascS (\rr d)$
is given by
$$
V_\phi f(x,\xi )
=
(2\pi )^{-\frac d2}\int _{\rr d}f(y)\overline {\phi (y-x)}e^{-i\scal y\xi}\, dy.
$$
It follows that $V_\phi$ is continuous from $\mascS (\rr d)$
to $\mascS (\rr {2d})$.
The definition extends uniquely to any $f\in \mascS '(\rr d)$, using
the formula
\begin{equation}
V_\phi f(x,\xi )
=
(2\pi )^{-\frac d2}(f,\phi (\cdo -x) e^{i\scal \cdo \xi })_{L^2(\rr d)},
\end{equation}
and then $V_\phi$ is continuous from $\mascS ' 
(\rr d)$ to $\mascS '(\rr {2d})$

\par

The definition carries over to other situations. For example, using the
Fourier characterizations for Pilipovi{\'c} spaces in \cite{NeToVi2}, it
follows that if $s\in \mathbf R_\flat$ satisfies $s\ge \flat _1$
and $\phi \in \maclH _{s}(\rr d)$, then the definition of $V_\phi$ from
$\mascS (\rr d)$ to $\mascS (\rr {2d})$, extends uniquely to a continuous
map from $\maclH _{s}'(\rr d)$ to $C^\infty (\rr {2d})$.

\par

Next, we define our class of 
modulation spaces. In the most general situation,
one specifies the window function as
\begin{equation}
\label{Eq:GaussWindow}
\phi _0(x)= \phi _{0,d}(x) \equiv \pi ^{-\frac d4}e^{-\frac 12|x|^2},
\qquad x\in \rr d.
\end{equation}

\par

\begin{defn}\label{Def:GenModSpace}
Let $\mascB$ be a solid QBF space on $\rr {2d}$, 
$\omega \in\mascP _A(\rr {2d})$,
and let
$\phi _0$ be as in \eqref{Eq:GaussWindow}. 
Then the \emph{modulation space} $M(\omega ,\mascB )$ consists
of all $f\in \maclH _{\flat _1}'(\rr d)$ such that
\begin{equation}
\label{Eq:GenModNormDef}
\nm f{M(\omega ,\mascB )}
\equiv \nm {V_{\phi _0} f\cdot \omega }{\mascB}
\end{equation}
is finite.
\end{defn}

\par

Important choices of $\mascB$ in Definition \ref{Def:GenModSpace}
are $\mascB =L^{p,q}_{(\omega )}(\rr {2d})$, $p,q\in (0,\infty ]$,
which are often referred as \emph{standard} or \emph{classical}
modulation spaces, and one sets
\begin{align}
M^{p,q}_{(\omega )}(\rr {2d}) &= M(\omega ,\mascB),
\quad \text{when}\quad
\mascB = L^{p,q}_{(\omega )}(\rr {2d}),
\label{Eq:ClassModSpaces}
\intertext{and}
M^{p}_{(\omega )} &= M^{p,p}_{(\omega )},
\qquad
p,q\in (0,\infty ].
\end{align}

\par

It is well-known that the space in \eqref{Eq:ClassModSpaces}
is a quasi-Banach space of order $\min (1,p,q)$, for any
$\omega \in \mascP _A(\rr {2d})$ (see \cite{Toft18}). 

\par

We remark that modulation spaces
of the form $M^{p,q}_{(\omega )}(\rr d)$
for $p,q\in [1,\infty ]$ and
$\omega$ belonging to
certain subclasses of
$\mascP (\rr {2d})$, 
were introduced by Feichtinger in
\cite{Fei1}. The theory of such
spaces were thereafter extended
to the full interval $p,q\in (0,\infty ]$
by Galperin and Samarah in \cite{GaSa}.

\par

\begin{rem}
In \cite{Fei1}, Feichtinger
also present
ideas on a broader family of
modulation spaces. In our
setting, this should correspond
to $M(\omega ,\mascB)$ spaces,
where $\mascB$ should be a Banach
space, and with all weights belonging
to $\mascP (\rr {2d})$.
(See also \cite{Fei6}.)
The theory of modulation spaces was,
in several aspects, extended in
\cite{FeiGro1,FeiGro2}, where
Feichtinger and Gr{\"o}chenig
introduced the coorbit space theory.
\end{rem}

\par

Suppose, more restrictive, that
$\omega \in \mascP _E(\rr {2d})$ and $\mascB$ is a normal
QBF space, as in
Definition \ref{Def:BFSpaces}.
Then it follows that $M(\omega ,\mascB )$ is a quasi-Banach
space of the same order as $\mascB$,
$$
\Sigma _1(\rr d)\hookrightarrow M(\omega ,\mascB) \hookrightarrow
\Sigma _1'(\rr d),
$$
and that for any $\phi \in \Sigma _1(\rr d)\setminus \{0\}$,
\begin{equation}
\tag*{(\ref{Eq:GenModNormDef})$'$}
\label{Eq:GenModNormDefGen}
\nm f{M(\omega ,\mascB )}
\equiv \nm {V_{\phi} f\cdot \omega }{\mascB}
\end{equation}
defines a quasi-norm for
$M(\omega ,\mascB)$ which is
equivalent to the quasi-norm in 
\eqref{Eq:GenModNormDef}. If, more
restrictive, all involved weights
belong to $\mascP (\rr {2d})$,
then the more narrow inclusions
$$
\mascS (\rr d)
\hookrightarrow
M(\omega ,\mascB)
\hookrightarrow
\mascS '(\rr d),
$$
hold (cf., e.{\,}g., \cite{ToPfTe}).

\par

\subsection{Fractional Fourier transforms and modulation
spaces}

\par

In order to describe mapping properties for 
fractional Fourier
transforms on modulation spaces, it is 
convenient to introduce
scales of modulation spaces, parameterized by 
the orders
of the transforms. Here and in what follows, we 
identify
elements in $\rr d$ with corresponding column 
matrices. We write

\begin{equation}
\label{Eq:ArhothetaDefB}
\begin{aligned}
A_{\rho}
&=
\left (
\begin{matrix}
\phantom -\cos (\rho _1\textstyle {\frac \pi 2}) & \dots & 0 &
\sin (\rho _1\textstyle {\frac \pi 2}),  & \dots & 0
\\
\vdots & \ddots & \vdots &
\vdots & \ddots & \vdots
\\
0 & \dots & \phantom -\cos (\rho _d\textstyle {\frac \pi 2}) &
0 & \dots & \sin (\rho _d\textstyle {\frac \pi 2})
\\[1ex]
-\sin (\rho _1\textstyle {\frac \pi 2}), & \dots & 0 &
\cos (\rho _1\textstyle {\frac \pi 2})  & \dots & 0
\\
\vdots & \ddots & \vdots &
\vdots & \ddots & \vdots
\\
0 & \dots & -\sin (\rho _d\textstyle {\frac \pi 2}), &
0 & \dots & \cos (\rho _d\textstyle {\frac \pi 2}) 
\end{matrix}
\right ),
%&=
%\left (
%\begin{matrix}
%\phantom -\alpha _{\rho _1} & \dots & 0 &
%\beta _{\rho _1}  & \dots & 0
%\\
%\vdots & \ddots & \vdots &
%\vdots & \ddots & \vdots
%\\
%0 & \dots & \phantom -\alpha _{\rho _d} &
%0 & \dots & \beta _{\rho _d}
%\\
%-\beta _{\rho _1} & \dots & 0 &
%\alpha _{\rho _1}  & \dots & 0
%\\
%\vdots & \ddots & \vdots &
%\vdots & \ddots & \vdots
%\\
%0 & \dots & -\beta _{\rho _d} &
%0 & \dots & \alpha _{\rho _d} 
%\end{matrix}
%\right ),
\\[1ex]
\rho &=(\rho _1,\dots ,\rho _d)\in \rr d .
%\\[1ex]
%\alpha _{\rho _j}
%&=
%\cos (\rho _j\textstyle {\frac \pi 2}),
%\quad
%%\text{and} \quad
%\beta _{\rho _j}
%=
%\sin (\rho _j\textstyle {\frac \pi 2}),
%\qquad
%j=1,\dots ,d.
\end{aligned}
\end{equation}
Notice that $A_\rho$ equals $A_{0,\rho}$ in 
\eqref{Eq:ArhothetaDef}.

\par

\begin{defn}
\label{Def:FrFTScales}
Let $\mascB$ be a solid QBF space on $\rr {2d}$, and let
$A_{\rho}$ be as in \eqref{Eq:ArhothetaDefB}.
\begin{enumerate}
\item The space $\mascB _\rho$ is given by
\begin{equation}
\mascB _\rho
\equiv
\sets {F\in \maclM (\rr {2d})}{F(A_{\rho }^{-1} \cdo)\in \mascB}.
\end{equation}
The topology of $\mascB _\rho$ is given by the
quasi-norm
\begin{equation}
\nm F{\mascB _\rho}
\equiv 
\nm {F(A_{\rho}^{-1}\cdo )}{\mascB},
\qquad
F\in \mascB _\rho .
\end{equation}

\vrum

\item $\mascB$ is called \emph{phase rotational invariant}
if $\mascB _\rho =\mascB$
and $\nm F{\mascB _\rho}=\nm F{\mascB}$, for every $\rho \in \rr d$ and
$F\in \mascB$.
\end{enumerate}
\end{defn}

\par

Evidently, $\mascB _\rho$ in Definition \ref{Def:FrFTScales}
is also a solid QBF space on $\rr {2d}$. Furthermore, $\mascB _\rho$
inherits any property in Definition \ref{Def:BFSpaces}
possessed by $\mascB$. We also have
\begin{equation}
\label{Eq:FrFTScalesProp}
\nm {F(A_\rho \cdo )}{\mascB _\rho} = \nm F{\mascB},
\quad
A_\rho ^{-1}=A^*_\rho =A_{-\rho},
\quad F\in \maclM (\rr {2d}),\  \rho \in \rr d.
\end{equation}

\par

\begin{rem}
In general, a solid QBF space $\mascB$ on
$\rr {2d}$ is not phase rotational invariant.
For example, suppose that $\mascB =L^{p,q}(\rr {2d})$ for some
$p,q\in (0,\infty ]$. Then $\mascB _\rho \neq \mascB$ when
$p\neq q$ and $\rho \notin 2\zz d$.

\par

On the other hand, if $p=q$, then $\mascB _\rho = \mascB$
for every $\rho \in \rr d$. Consequently, $L^{p,p}(\rr {2d})=L^p(\rr {2d})$
is phase rotational invariant.
\end{rem}

\par

We now hav the following. Here, recall Definition
\ref{Def:SolidBF} concerning strict solid QBF spaces.

\par

\begin{thm}
\label{Thm:FracFourOnModSpaces}
Let $\mascB$ be a solid QBF space on $\rr {2d}$,
$\mascB _\rho$ be as in
Definition \ref{Def:FrFTScales}, and let
$\omega ,\omega _\rho \in \mascP _A(\rr {2d})$ be such that
$\omega _\rho =\omega (A_\rho \cdo )$, $\rho \in \rr d$. Then
$\mascF _\rho$ on $\maclH _{\flat _1}'(\rr d)$ restricts to a
homeomorphism from $M(\omega ,\mascB )$ to
$M(\omega _\rho ,\mascB _\rho)$,
and
\begin{equation}
\label{Eq:FrFTModNorm}
\nm {\mascF _\rho f}{M(\omega _\rho ,\mascB _\rho )}
\asymp
\nm {f}{M(\omega ,\mascB )},
\qquad f\in \maclH _{\flat _1}'(\rr d),\ \rho \in \rr d .
\end{equation}
If, in addition, $\mascB$ is strict solid, then
\eqref{Eq:FrFTModNorm} holds with strict equality.
\end{thm}

\par

\begin{proof}
It suffices to prove \eqref{Eq:FrFTModNorm}.
Let $\phi _0$ be as in \eqref{Eq:GaussWindow},
and let $f\in \maclH _{\flat _1}'(\rr d)$.
By (1.35) in \cite{ToBhMa} and \eqref{Eq:FrFTScalesProp},
we obtain
\begin{equation}
\label{Eq:FrFTonSTFT}
|V_{\phi _0}(\mascF _\rho f)| = |V_{\phi _0} f(A_\rho \cdo )| .
\end{equation}
This gives
$$
\nm {\mascF _\rho f}{M(\omega _\rho ,\mascB _\rho )}
=
\nm {V_\phi (\mascF _\rho f)\cdot \omega _\rho}{\mascB _\rho}
\asymp
\nm {V_\phi f(A_\rho \cdo )\cdot \omega (A_\rho \cdo )}{\mascB _\rho}
=
\nm {V_\phi f\cdot \omega}{\mascB},
$$
which gives \eqref{Eq:FrFTModNorm}. 
If, in addition, $\mascB$ is strict solid, then
the previous computations hold with equality at each step. This
implies that \eqref{Eq:FrFTModNorm} holds with equality, and the
result follows.
\end{proof}

\par

The next result is a straight-forward consequence of Theorem
\ref{Thm:FracFourOnModSpaces}. The details are left to the reader.

\par

\begin{cor}
Let $\mascB$ be a solid phase rotational invariant
QBF space, and let
$\omega \in \mascP _A(\rr {2d})$ be such that
$\omega  =\omega (A_\rho \cdo )$, $\rho \in \rr d$. Then
$\mascF _\rho$ is a
homeomorphism on $M(\omega ,\mascB )$,
and
\begin{equation}
\label{Eq:FrFTModNorm}
\nm {\mascF _\rho f}{M(\omega ,\mascB )}
\asymp
\nm {f}{M(\omega ,\mascB )},
\qquad f\in \maclH _{\flat _1}'(\rr d),\ \rho \in \rr d .
\end{equation}
If, in addition, $\mascB$ is strict solid,
then \eqref{Eq:FrFTModNorm} holds with strict equality.
\end{cor}

\par

\subsection{Fractional symplectic Fourier transforms
and modulation spaces}

\par

It is convenient to use symplectic short-time Fourier transforms
instead of the usual short-time Fourier transform when investigating
fractional symplectic Fourier transforms. For $p\in [1,\infty ]$,
the symplectic short-time Fourier  transform of $\fka \in L^p(\rr {2d})$
with respect to (the window function) $\psi \in \mascS (\rr {2d})\setminus \{0\}$
is given by
$$
\maclV _\psi \fka (X,Y) = \pi ^{-d}\int _{\rr {2d}}\fka (Z)\overline{\psi (Z-X)}
e^{2i\sigma (Y,Z)}\, dZ,
\qquad
X,Y\in \rr {2d}.
$$

\par

The definition of the map $(\fka ,\psi )\mapsto \maclV _\psi \fka$ is
extended in similar ways as for the map $(f,\phi )\mapsto V_\phi f$.

\par

The following lemma explains the relationship 
between symplectic short-time
Fourier transforms and symplectic Bargmann 
transforms.
The result follows by straight-forward
computations, and is left for the reader.
(Cf. Proposition 3.4.1 in \cite{Gro1},
(3.8) in \cite{NeToVi2},
or (1.25) in \cite{Toft18}.)

\par

% \begin{lemma}
% Let either $s\geq \frac 12$,
% $F \in \maclS _s'(\rr {2d})$, and $\phi \in 
% \maclS _s(\rr {2d})$ hold, 
% or $s > \frac 12$, $F \in \Sigma _s'(\rr {2d})$, 
% and $\phi \in \Sigma _s(\rr {2d})$.
% Let $\fka = \mascW F$ and $\psi = \mascW \phi$.
% Then
% %%
% \begin{align}
% \maclV _\psi \fka (X,Y)
% &=
% 2^de^{-2i\scal y\xi}
% (V_\phi F)(x-y,x+y,\xi -\eta ,-\xi -\eta ),
% \label{Eq:SymplSTFTtoSTFT}
% \intertext{and}
% V_\phi F(x,y,\xi ,-\eta )
% &=
% 2^{-d}e^{i \scal {y-x}{\xi + \eta }}
% \maclV _\psi \fka \big ( {\textstyle{\frac 12}} (Y+X) , 
% {\textstyle{\frac 12}} (Y-X) \big ).
% \label{Eq:STFTtoSymplSTFT}
% \intertext{where}
% X
% &=
% (x,\xi ),\quad  Y=(y,\eta ),\quad x,y,\xi ,\eta \in \rr d.
% \notag
% \end{align}
% %%
% \end{lemma}

\begin{lemma}
\label{Lemma:SymplSTFTSymplBarg}
Let
\begin{equation}
\label{Eq:SymplGaussWindow}
\begin{gathered}
\psi _0(X)
=
\left (\frac 2\pi \right )^{\frac d2}
e^{-|X|^2},
\quad
J=\left (
\begin{matrix}
0 & I_{d}
\\
-I _d & 0
\end{matrix}
\right ),
\\[1ex]
\text{and}\quad
\kappa(X,Y)=2^{-1/2}(X+iJY),
\qquad X,Y\in \rr {2d}.
\end{gathered}
\end{equation}
Then
$$
\maclV_{\psi_0} \fka(X, Y) = \pi^{-d}
e^{- |\kappa(X,Y)|^2} e^{i \sigma(Y, X)} (\mathfrak V_{\sigma, d} \fka)
\left( \kappa (X,Y)\right) . 
$$
\end{lemma}

\par

% Observe that it follows from 
% \eqref{Eq:SymplSTFTtoSTFT2} that
% $$
% V_\phi F(x,y,\xi ,\eta )
% =
% 2^{-d}e^{i\scal {y-x}{\xi -\eta }}
% \maclV _\psi \fka
% \left( \frac 12 (x+y,\xi +\eta ) ,
% \frac 12 (-x+y,-\xi -\eta ) \right).
% $$

\par

% \begin{proof}
% It suffices to prove \eqref{Eq:SymplSTFTtoSTFT}.
% We only prove the result when
% $F,\phi \in \mascS (\rr {2d})$. The general
% result follows by similar arguments and is left 
% to the reader.

% \par

% Let
% $$
% H(x,y_1,y_2,z) = F(z+\frac {y_1}2,
% z-\frac {y_1}2)
% \overline{\phi (z+\frac {y_2}2-x,
% z-\frac {y_2}2-x)}
% $$
% and
% $$
% \Phi (y_1,y_2,z,\xi ,\eta ,\zeta )
% =
% \scal {y_1}\zeta +\scal {y_2}{\xi -\zeta}
% +2(\scal y\zeta -\scal z\eta ).
% $$
% Then it follows from the definitions that
% $$
% \maclV _\psi \fka (X,Y)
% =
% (2\pi ^2)^{-d}\iint \left (
% \iint H(x,y_1,y_2,z)
% e^{-i\Phi (y_1,y_2,z,\xi ,\eta ,\zeta )}\, 
% dy_1dy_2 \right )
% \, dzd\zeta .
% $$
% By Fourier's inversion formula, we get
% \textcolor{blue}{$$
% \maclV _\psi \fka (X,Y)
% =
% \pi ^{-d}e^{-2i\scal y\xi }\int 
% \int H(x,y_1,y_1+2y,z)e^{-i\Phi (y_1,y_1+2y,z,\xi ,\eta ,0)}\, dy_1dz ,
% $$}
% \textcolor{red}{$$
% \maclV _\psi \fka (X,Y)
% =
% \pi ^{-d}e^{-2i\scal y\xi }\iint 
% \iint H(x,y_1,y_1+2y,z)e^{-i\Phi (y_1,y_1+2y,z,\xi ,\eta ,0)}\, dy_1dz ,
% $$}
% and taking $z+\frac {y_1}2$ and $z-\frac {y_1}2$ as new variables
% of integration, we obtain after straight-forward computations that
% %%
% \begin{multline*}
% \pi ^{d}e^{2i\scal y\xi }
% \maclV _\psi \fka (X,Y)
% \\[1ex]
% =
% \iint F(y_1,y_2)
% \overline{\phi (y_1-(x-y),y_2-(x+y))}
% e^{-i(\scal {y_1}{\xi -\eta}
% +\scal {y_2}{-\xi -\eta})}\, dy_1dy_2 ,
% \end{multline*}
% %%
% which is the same as \eqref{Eq:SymplSTFTtoSTFT}. This
% gives the result.
% \end{proof}

\par

The next result concerns symplectic
short-time Fourier
transforms under actions of partial fractional 
symplectic
Fourier transforms.
% Here, the symplectic
% analogy of
% $\phi _0$ in \eqref{Eq:GaussWindow} is
% %%
% \begin{equation}
% \label{Eq:SymplGaussWindow}
% \psi _0(X)= \left ( \frac 2\pi \right )^{\frac d2}e^{- |X|^2},
% \qquad X\in \rr {2d}.
% \end{equation}
% %%

\par

\begin{prop}
\label{Prop:FrFTSymplSTFT}
Let $\psi _0$ be as in 
\eqref{Eq:SymplGaussWindow},
$\fka \in \maclS _{1/2}'(\rr {2d})$,
and let
\begin{equation}
\label{Eq:MatrixMDef}
M_{\rho,\theta}=\frac12
 \left (
 \begin{matrix}
 A_{\theta-\rho}+A_{\theta+\rho}
 &
 A_{\theta-\rho}-A_{\theta+\rho}
 \\
 A_{\theta-\rho}-A_{\theta+\rho}
 &
 A_{\theta-\rho}+A_{\theta+\rho}
 \end{matrix}
 \right ).
\end{equation}
Then
\begin{equation}
\label{Eq:FrFTSymplSTFT}
 |\maclV _{\psi_0}\mathscr F_{\sigma,(\rho,\theta)} \fka(X,Y)|
 =
 |(\maclV _{\psi_0}\fka)(M_{\rho,\theta}(X,Y))|.
\end{equation}
\end{prop}

\par

\begin{proof}
By straight-forward computations we get
\begin{equation}
 \kappa \circ M_{\rho,\theta}
 =
 A_{\rho,\theta}\circ\kappa,
\end{equation}
The result follows by combining
the previous identity with
\eqref{Eq:FracSymplFT}, the
unitarity of $A_{\rho,\theta}$,
and Lemma 
\ref{Lemma:SymplSTFTSymplBarg}.
%The details are left for the reader.
\end{proof}

\par

We observe that \eqref{Eq:FrFTSymplSTFT}
can also be expressed as
\begin{equation}
\tag*{(\ref{Eq:FrFTSymplSTFT})$'$}
\label{Eq:FrFTSymplSTFT2}
\begin{aligned}
|(\maclV _{\psi _0}&\mascF _{\sigma ,(\rho ,\theta )}\fka )(X,Y)|
%\notag
\\[1ex]
&=
|(\maclV _{\psi _0} \fka )(A_{-\theta}
(C_{\sigma ,\rho} X+D_{\sigma ,\rho} Y),
A_{-\theta }(D_{\sigma ,\rho} X
+C_{\sigma ,\rho} Y)) )|, 
%\notag
\\[1ex]
&\mathllap{\text{where}} 
\\[1ex]
C_{\sigma ,\rho}
&=
\left (
\begin{matrix}
C_{\rho} & 0
\\[1ex]
0 & C_{\rho}
\end{matrix}
\right ),
% \qquad
% C_{0,\rho}
% =
% \left (
% \begin{matrix}
% \cos (\rho _1\textstyle {\frac \pi 2})  & \dots & 0
% \\
% \vdots & \ddots & \vdots 
% \\
% 0 & \dots & \cos (\rho _d\textstyle {\frac \pi 2}) 
% \end{matrix}
% \right )
%\notag
% \\[1ex]
% &\mathllap{\text{and}}
% \\[1ex]
\quad \text{and}\quad
D_{\sigma ,\rho}
=
\left (
\begin{matrix}
\phantom -0 & D_{\rho}
\\[1ex]
-D_{\rho} & 0
\end{matrix}
\right ),
% \qquad
% S_{0,\rho}
% =
% \left (
% \begin{matrix}
% \sin (\rho _1\textstyle {\frac \pi 2}) & \dots & 0
% \\
% \vdots & \ddots & \vdots 
% \\
% 0 & \dots & \sin (\rho _d\textstyle {\frac \pi 2})
% \end{matrix}
% \right ),
%\notag
%\left (
%\begin{matrix}
%\cos (\rho _1\textstyle {\frac \pi 2}) & \dots & 0 &
%0  & \dots & 0
%\\
%\vdots & \ddots & \vdots &
%\vdots & \ddots & \vdots
%\\
%0 & \dots & \cos (\rho _d\textstyle {\frac \pi 2}) &
%0 & \dots & 0
%\\[1ex]
%0  & \dots & 0 & \cos (\rho _1\textstyle {\frac \pi 2}) & \dots & 0
%\\
%\vdots & \ddots & \vdots &
%\vdots & \ddots & \vdots
%\\
%0 & \dots & 0 & 0 & \dots & \cos (\rho _d\textstyle {\frac \pi 2})
%\end{matrix}
%\right )
\\[1ex]
X,Y&\in \rr {2d},\ \rho ,\theta \in \rr d.
\end{aligned}
\end{equation}
where $C_\rho$ and $D_\rho$ are given by
\eqref{Eq:CosSineMatrices}.

\par

\medspace

When discussing mapping properties for fractional symplectic
Fourier transforms on modulation spaces, it is more convenient
to consider symplectic analogies of such spaces. They are essentially
defined by replacing the ordinary short-time Fourier transform
with the symplectic short-time Fourier transform, when defining the norms.

\par

\begin{defn}\label{Def:GenSymplModSpace}
Let $\mascB$ be a solid QBF space on $\rr {4d}$, 
$\omega \in\mascP _A(\rr {4d})$,
and let
$\psi _0$ be as in \eqref{Eq:SymplGaussWindow}. 
Then the \emph{symplectic modulation space}
$M_\sigma (\omega ,\mascB )$ consists
of all $\fka \in \maclH _{\sigma ,\flat _1}'(\rr {2d})$ such that
\begin{equation}
\label{Eq:GenSymplModNormDef}
\nm \fka{M_\sigma (\omega ,\mascB )}
\equiv \nm {\maclV _{\psi _0} \fka \cdot \omega }{\mascB}
\end{equation}
is finite.
\end{defn}

\par

\begin{rem}
If $\omega \in \mascP _E(\rr {4d})$ and
$\mascB$ is normal in Definition
\ref{Def:GenSymplModSpace}, then
$M_\sigma (\omega ,\mascB )$ stays between $\Sigma _1(\rr {2d})$
and $\Sigma _1'(\rr {2d})$. Since the symplectic Fourier transform is
only a rescaling of the ordinary Fourier transform, it follows that there
is a one-to-one correspondence between symplectic modulation
spaces and modulation spaces for such types of weight functions
$\omega$ and normal QBF spaces $\mascB$.

\par

Consequently, in the common situations, symplectic modulation
spaces are merely a reformulation of the usual modulation spaces.
\end{rem}

\par

\begin{defn}
\label{Def:FrSymplFTScales}
Let $\mascB$ be a solid QBF space on $\rr {4d}$, and let
$M_{\rho ,\theta}$ be as in \eqref{Eq:MatrixMDef}.
\begin{enumerate}
\item The space
$\mascB _{\sigma ,(\rho ,\theta )}$ is given by
\begin{equation}
\mascB _{\sigma ,(\rho ,\theta )}
\equiv
\sets {H\in \maclM (\rr {4d})}
{H(M_{\rho ,\theta } \cdo )\in 
\mascB}.
\end{equation}
The topology of $\mascB _{\sigma ,(\rho ,\theta )}$is given by the
quasi-norm
\begin{equation}
\nm H{\mascB _{\sigma ,(\rho ,\theta )}}
\equiv 
\nm {H(M_{\rho ,\theta } \cdo )}{\mascB},
\qquad
H\in \mascB _{\sigma ,(\rho ,\theta )} .
\end{equation}

\vrum

\item $\mascB$ is called \emph{symplectic phase rotational invariant},
if $\mascB _{\sigma ,(\rho ,\theta )} =\mascB$
and $\nm H{\mascB _{\sigma ,(\rho ,\theta )}}=\nm H{\mascB}$, for every
$\rho ,\theta \in \rr d$ and $H\in \mascB$.
\end{enumerate}
\end{defn}

\par

The following result explains how symplectic modulation
spaces behave under actions of fractional symplectic
Fourier transforms.

\par

\begin{thm}
\label{Thm:FracSymplFourOnModSpaces}
Let $\mascB$ be a solid QBF space on $\rr {4d}$, 
$\mascB _{\sigma ,(\rho ,\theta )}$ be as in
Definition \ref{Def:FrFTScales}, and let
$\omega ,\omega _{\rho ,\theta} \in \mascP _A(\rr {4d})$ be such that
$\omega _{\rho ,\theta}
=
\omega (M_{\rho ,\theta }^{-1} \cdo )$,
$\rho ,\theta \in \rr d$. Then
$\mascF _{\sigma ,(\rho ,\theta )}$ on $\maclH _{\sigma ,\flat _1}'(\rr {2d})$
 restricts to a homeomorphism from $M _\sigma (\omega ,\mascB )$ to
$M _\sigma (\omega _{\rho ,\theta} ,\mascB _{\sigma ,(\rho ,\theta )})$,
and
\begin{equation}
\label{Eq:FrSymplFTModNorm}
\nm {\mascF _{\sigma ,(\rho ,\theta )} \fka}
{M _\sigma (\omega _{\rho ,\theta} ,\mascB _{\sigma ,(\rho ,\theta )})}
\asymp
\nm {\fka}{M _\sigma (\omega ,\mascB )},
\qquad \fka \in \maclH _{\sigma ,\flat _1}'(\rr {2d}),\ \rho ,\theta \in \rr d .
\end{equation}
If in addition $\mascB$ is strict solid, then
\eqref{Eq:FrSymplFTModNorm} holds with strict equality.
\end{thm}

\par

\begin{proof}
The result follows by similar arguments as for the proof of
Theorem \ref{Thm:FracFourOnModSpaces},
where we use \eqref{Eq:FrFTSymplSTFT}
or \ref{Eq:FrFTSymplSTFT2} instead of \eqref{Eq:FrFTonSTFT}.
The details are left for the reader.
\end{proof}

\par

As a straight-forward consequence of Theorem
\ref{Thm:FracSymplFourOnModSpaces} we get the 
following. The details are left for the reader.
Here we let $M_{\sigma ,(\omega )}^p(\rr {2d})
= M_\sigma (\omega ,L^p(\rr {4d}))$,
when $p\in (0,\infty ]$ and
$\omega \in \mascP _A(\rr {4d})$.

\par

\begin{prop}
Let $p\in (0,\infty ]$, 
and
$\omega \in \mascP _A(\rr {4d})$ be such that
$\omega  =
\omega (M_{\rho ,\theta }^{-1} \cdo )$,
$\rho ,\theta \in \rr d$. Then
$\mascF _{\sigma ,(\rho ,\theta )}$ on
$\maclH _{\sigma ,\flat _1}'(\rr {2d})$
restricts to an isometric bijection on
$M _{\sigma ,(\omega )}^p(\rr {2d})$.
\end{prop}

\appendix

%%%%%%%%%%%%%%%%%%%%%%
\section{Lists of fractional symplectic
Fourier transforms}
\label{App:A}
%%%%%%%%%%%%%%%%%%%%%%

\par

In the following list we consider symplectic
fractional Fourier transforms
$\mascF _{\sigma ,(\rho ,\theta )}$ on
$\maclS _{1/2}'(\rr {2d})$ when
$\rho ,\theta \in \{ 0,1,2,3\}$. More 
precisely, for
the values $\alpha _\rho$, $\beta _\rho$,
$\alpha _\theta$, $\beta _\theta$,
the $(2d)\times (2d)$ matrix
$A_{0,\theta}^*$, and
$\mascF _{\sigma ,(\rho ,\delta )}$
we have the following table. Here let
$\sfR _j$, $j=0,1,2,\dots$, be the
operators on $\maclS _{1/2}'(\rr {2d})$,
defined by
$$
\sfR _0=\operatorname{Id},
\quad
(\sfR _1\fka )(x,\xi )=\fka (\xi ,-x),
\quad \text{and}\quad
\sfR _{j+1} = \sfR _1\circ \sfR _j,
\qquad j\ge 1.
$$
It follows that
$$
(\sfR _0\fka )(x,\xi )= \fka (x,\xi ),
\quad
(\sfR _2\fka )(x,\xi )= \fka (-x,-\xi )
\quad \text{and}\quad
(\sfR _3\fka )(x,\xi )=\fka (-\xi ,x).
$$
{\small{
\begin{table}[h]
\caption{Some values of
$\alpha _\rho$, $\beta _\rho$,
$\alpha _\theta$, $\beta _\theta$,
$A_{0,\theta}^*$, and $I=I_d$
is the $d\times d$ unit matrix.}
% , and
% $\fkb =
% \mascF _{\sigma ,(\rho ,\theta )}\fka$.}
\label{Table:ValuesOfCoeff}
\begin{tabular}{l || c | c | c | c | c | c | c | c }
$(\rho ,\theta )$ & $(0,0)$ & $(0,1)$ & $(0,2)$ & $(0,3)$ & $(1,0)$ & $(1,1)$ & $(1,2)$ & $(1,3)$
\\
\hline
%\hline
$\alpha _\rho$ & $1$ & $1$ & $\phm 1$
& $1$ & $0$ & $0$ & $\phm 0$ & $\phm 0$
\\[1ex]
$\beta _\rho$ & $0$ & $0$ & $\phm0$ &
$\phm 0$ & $1$ & $1$ & $\phm 1$ & $\phm 1$
\\[1ex]
$\alpha _\theta$ & $1$ & $0$ & $-1$ &
$\phm 0$ & $1$ & $0$ & $-1$ & $\phm 0$
\\[1ex]
$\beta _\theta$ & $0$ & $1$ & $\phm 0$ &
$-1$ & $0$ & $1$ & $\phm 0$ & $-1$
\\[2ex]
$A_{0,\theta}^*$
&
{\tiny{
$\!\!\!\! \left (
\begin{matrix}
I & \!\!\!\!  0
\\
0 & \!\!\!\!  I
\end{matrix}
\right )\!\!\!\!$
}}
& 
{\tiny{
$\!\!\!\! \left (
\begin{matrix}
0 & \!\!\!\!  -I
\\
I & \!\!\!\!  \phm 0
\end{matrix}
\right ) \!\!\!\!$
}}
& 
{\tiny{
$\!\!\!\! \left (
\begin{matrix}
-I & \!\!\!\!  \phm 0
\\
\phm 0 & \!\!\!\!  -I
\end{matrix}
\right ) \!\!\!\!$
}}
& 
{\tiny{
$\!\!\!\! \left (
\begin{matrix}
\phm 0 & \!\!\!\!  I
\\
-I & \!\!\!\!  0
\end{matrix}
\right ) \!\!\!\!$
}}
& 
{\tiny{
$\!\!\!\! \left (
\begin{matrix}
I & \!\!\!\!  0
\\
0 & \!\!\!\!  I
\end{matrix}
\right ) \!\!\!\!$
}}
& 
{\tiny{
$\!\!\!\! \left (
\begin{matrix}
0 & \!\!\!\!  -I
\\
I & \!\!\!\!  \phm 0
\end{matrix}
\right ) \!\!\!\!$
}}
&
{\tiny{
$\!\!\!\! \left (
\begin{matrix}
-I & \!\!\!\!  \phm 0
\\
\phm 0 & \!\!\!\!  -I
\end{matrix}
\right ) \!\!\!\!$
}}
& 
{\tiny{
$\!\!\!\! \left (
\begin{matrix}
\phm 0 & \!\!\!\!  I
\\
-I & \!\!\!\!  0
\end{matrix}
\right ) \!\!\!\!$
}}
\\[2ex]
%{\tiny{
$\mascF _{\sigma ,(\rho ,\theta )}$
&
$\sfR _0$
&
$\sfR _1$
&
$\sfR _2$
&
$\sfR _3$
&
$\sfR _3\circ \mascF _\sigma$
&
$\mascF _\sigma$
&
$\sfR _1\circ \mascF _\sigma$
&
$\sfR _2\circ \mascF _\sigma$
\end{tabular}

\vspace{0.8cm}

\vspace{0.8cm}

\begin{tabular}{l || c | c | c | c | c | c | c | c }
$(\rho ,\theta )$ & $(2,0)$ & $(2,1)$ & $(2,2)$ & $(2,3)$ & $(3,0)$ & $(3,1)$ & $(3,2)$ & $(3,3)$
\\
\hline
%\hline
$\alpha _\rho$ & $-1$ & $-1$ & $-1$ & $-1$ &
$\phm 0$ & $\phm 0$ & $\phm 0$ & $\phm 0$
\\[1ex]
$\beta _\rho$ & $\phm 0$ & $\phm 0$ &
$\phm 0$ & $\phm 0$ &
$-1$ & $-1$ & $-1$ & $-1$
\\[1ex]
$\alpha _\theta$ & $\phm 1$ & $\phm 0$ &
$-1$ & $\phm 0$ & $\phm 1$ & $\phm 0$ &
$-1$ & $\phm 0$
\\[1ex]
$\beta _\theta$ & $\phm 0$ & $\phm 1$ &
$\phm 0$ & $-1$ & $\phm 0$ & $\phm 1$ &
$\phm 0$ & $-1$
\\[2ex]
$A_{0,\theta}^*$
&
{\tiny{
$\!\!\!\! \left (
\begin{matrix}
I & \!\!\!\!  0
\\
0 & \!\!\!\!  I
\end{matrix}
\right )\!\!\!\!$
}}
& 
{\tiny{
$\!\!\!\! \left (
\begin{matrix}
0 & \!\!\!\!  -I
\\
I & \!\!\!\!  \phm 0
\end{matrix}
\right ) \!\!\!\!$
}}
& 
{\tiny{
$\!\!\!\! \left (
\begin{matrix}
-I & \!\!\!\!  \phm 0
\\
\phm 0 & \!\!\!\!  -I
\end{matrix}
\right ) \!\!\!\!$
}}
& 
{\tiny{
$\!\!\!\! \left (
\begin{matrix}
\phm 0 & \!\!\!\!  I
\\
-I & \!\!\!\!  0
\end{matrix}
\right ) \!\!\!\!$
}}
& 
{\tiny{
$\!\!\!\! \left (
\begin{matrix}
I & \!\!\!\!  0
\\
0 & \!\!\!\!  I
\end{matrix}
\right ) \!\!\!\!$
}}
& 
{\tiny{
$\!\!\!\! \left (
\begin{matrix}
0 & \!\!\!\!  -I
\\
I & \!\!\!\!  \phm 0
\end{matrix}
\right ) \!\!\!\!$
}}
&
{\tiny{
$\!\!\!\! \left (
\begin{matrix}
-I & \!\!\!\!  \phm 0
\\
\phm 0 & \!\!\!\!  -I
\end{matrix}
\right ) \!\!\!\!$
}}
& 
{\tiny{
$\!\!\!\! \left (
\begin{matrix}
\phm 0 & \!\!\!\!  I
\\
-I & \!\!\!\!  0
\end{matrix}
\right ) \!\!\!\!$
}}
\\[2ex]
%{\tiny{
$\mascF _{\sigma ,(\rho ,\theta )}$
&
$\sfR _2$
&
$\sfR _3$
&
$\sfR _0$
&
$\sfR _1$
&
$\sfR _1\circ \mascF _\sigma$
&
$\sfR _2\circ \mascF _\sigma$
&
$\sfR _3\circ \mascF _\sigma$
&
$\mascF _\sigma$
\end{tabular}
\end{table}
}}

\end{document}